\documentclass[11pt]{amsart}

\usepackage{
    amsmath,
    amsfonts,
    amssymb,
    amsthm,
    amscd,
    booktabs,
    comment,
    enumitem,
    etoolbox,
    gensymb,    
    mathtools,
    mathdots,
    stmaryrd,
}
\usepackage[usenames,dvipsnames]{xcolor}
\usepackage[all]{xy}

\usepackage[T1]{fontenc}
\usepackage{bbm}                     
\usepackage[colorlinks=true, linkcolor=blue, citecolor=blue, urlcolor=blue, breaklinks=true]{hyperref}

\DeclareFontFamily{OT1}{pzc}{}
\DeclareFontShape{OT1}{pzc}{m}{it}{<-> s * [1.10] pzcmi7t}{}
\DeclareMathAlphabet{\mathpzc}{OT1}{pzc}{m}{it}

\usepackage[capitalize]{cleveref}   

\crefname{assumption}{Assumption}{Assumptions}
\crefname{defin}{Definition}{Definitions}
\crefname{eg}{Example}{Examples}
\crefname{egs}{Example}{Examples}
\crefname{lem}{Lemma}{Lemmas}
\crefname{theo}{Theorem}{Theorems}
\crefname{equation}{}{}
\crefname{enumi}{}{}
\crefname{rem}{Remark}{Remarks}

\newcommand\C{\mathbb{C}}
\newcommand\N{\mathbb{N}}

\newcommand\Q{\mathbb{Q}}
\newcommand\R{\mathbb{R}}
\newcommand\Z{\mathbb{Z}}
\newcommand\kk{\Bbbk}
\newcommand\one{\mathbbm{1}}

\newcommand\bA{\mathbf{A}}
\newcommand\bB{\mathbf{B}}
\newcommand\bF{\mathbf{F}}
\newcommand\bFp{\mathbf{F}^\Tcat}

\newcommand\fg{\mathfrak{g}}
\newcommand\gl{\mathfrak{gl}}           

\newcommand\md{\textup{-mod}}

\newcommand\sC{\mathsf{C}}              
\newcommand\sD{\mathsf{D}}              
\newcommand\sR{\mathsf{R}}              
\newcommand\sT{\mathsf{T}}
\newcommand\sV{\mathsf{V}}              
\newcommand\sv{\mathsf{v}}              

\newcommand\qdim{\dim_q}

\newcommand\qbinom[2]{\begin{bmatrix} #1 \\ #2 \end{bmatrix}}   

\newcommand\AFcat{\mathpzc{AF}}         
\newcommand\cC{\mathpzc{C}}
\newcommand\cD{\mathpzc{D}}
\newcommand\cN{\mathpzc{N}}
\newcommand\cEnd{\mathpzc{End}}
\newcommand\Fcat{\mathpzc{F}}           
\newcommand\Mcat{\mathpzc{M}}           
\newcommand\Tcat{\mathpzc{T}}           

\newcommand\cI{\mathcal{I}}             

\newcommand\go{{\mathsf{I}}}            

\DeclareMathOperator{\End}{End}
\DeclareMathOperator{\ev}{ev}
\DeclareMathOperator{\flip}{flip}
\DeclareMathOperator{\Hom}{Hom}
\DeclareMathOperator{\Id}{Id}      

\DeclareMathOperator{\Kar}{Kar}

\DeclareMathOperator{\Mor}{Mor}
\DeclareMathOperator{\Ob}{Ob}
\DeclareMathOperator{\rank}{rank}

\DeclareMathOperator{\Rot}{Rot}

\DeclareMathOperator{\Tr}{Tr}
\DeclareMathOperator{\tr}{tr}

\usepackage{tikz}
\usetikzlibrary{cd}
\usetikzlibrary{arrows.meta}
\usetikzlibrary{decorations.markings}
\usetikzlibrary{calc}

\tikzset{anchorbase/.style={>=To,baseline={([yshift=-0.5ex]current bounding box.center)}}}
\tikzset{ 
    centerzero/.style={>=To,baseline={([yshift=-0.5ex](#1))}},
    centerzero/.default={0,0}
}
\tikzset{wipe/.style={white,line width=3pt}}

\newcommand\braidup{to[out=up,in=down]}
\newcommand\braiddown{to[out=down,in=up]}

\newcommand\posdot[1]{
    \draw[fill=white, draw=black] (#1) circle (0.13);
    \filldraw[black] (#1) -- +(0.09192,0.09192) arc(45:135:0.13) -- cycle;
    \filldraw[black] (#1) -- +(0.09192,-0.09192) arc(-45:-135:0.13) -- cycle
}
\newcommand\negdot[1]{
    \draw[fill=white, draw=black] (#1) circle (0.13);
    \filldraw[black] (#1) -- (0.09192,0.09192) arc(45:-45:0.13) -- cycle;
    \filldraw[black] (#1) -- (-0.09192,0.09192) arc(135:225:0.13) -- cycle
}
\newcommand\opendot[1]{
    \filldraw[fill=white, draw=black] (#1) circle (0.05)
}

\newcommand{\posaff}[1]{
    \filldraw[fill=white,draw=black] (#1)+(0.09,0) arc(0:360:0.09);
    \draw (#1)++(-0.09,0) to ++(0.18,0);
    \draw (#1)++(0,-0.09) -- ++(0,0.18)
}
\newcommand{\negaff}[1]{
    \filldraw[fill=white,draw=black] (#1)+(0.09,0) arc(0:360:0.09);
    \draw (#1)++(-0.09,0) to ++(0.18,0)
}
\newcommand\multdot[3]{
    \filldraw[fill=white, draw=black] (#1) circle (1.5pt) node[anchor=#2] {$\scriptstyle{#3}$}
}

\newcommand\bub[1]{
    \draw (#1)++(0,0.2) arc(90:-270:0.2)
}

\newcommand\bubble{
    \begin{tikzpicture}[centerzero]
        \bub{0,0};
    \end{tikzpicture}
}
\newcommand\idstrand[1][a]{
    \begin{tikzpicture}[centerzero]
        \draw (0,-0.2) -- (0,0.2);
    \end{tikzpicture}
}
\newcommand\crossmor{
    \begin{tikzpicture}[centerzero]
        \draw (0.2,-0.2) -- (-0.2,0.2);
        \draw (-0.2,-0.2) -- (0.2,0.2);
    \end{tikzpicture}
}
\newcommand\poscross{
    \begin{tikzpicture}[centerzero]
        \draw (0.2,-0.2) -- (-0.2,0.2);
        \draw[wipe] (-0.2,-0.2) -- (0.2,0.2);
        \draw (-0.2,-0.2) -- (0.2,0.2);
    \end{tikzpicture}
}
\newcommand\negcross{
    \begin{tikzpicture}[centerzero]
        \draw (-0.2,-0.2) -- (0.2,0.2);
        \draw[wipe] (0.2,-0.2) -- (-0.2,0.2);
        \draw (0.2,-0.2) -- (-0.2,0.2);
    \end{tikzpicture}
}
\newcommand{\cupmor}{
    \begin{tikzpicture}[anchorbase]
        \draw (-0.15,0.15) -- (-0.15,0) arc(180:360:0.15) -- (0.15,0.15);
    \end{tikzpicture}
}
\newcommand{\capmor}{
    \begin{tikzpicture}[anchorbase]
        \draw (-0.15,-0.15) -- (-0.15,0) arc(180:0:0.15) -- (0.15,-0.15);
    \end{tikzpicture}
}
\newcommand\mergemor{
    \begin{tikzpicture}[centerzero]
      \draw (-0.2,-0.2) to (0,0);
      \draw (0.2,-0.2) to (0,0);
      \draw (0,0) to (0,0.2);
    \end{tikzpicture}
}
\newcommand\splitmor{
    \begin{tikzpicture}[centerzero]
      \draw (-0.2,0.2) to (0,0);
      \draw (0.2,0.2) to (0,0);
      \draw (0,0) to (0,-0.2);
    \end{tikzpicture}
}
\newcommand\lolly{
    \begin{tikzpicture}[centerzero]
        \draw (0,-0.2) -- (0,0) to[out=40,in=down] (0.15,0.15) arc(0:180:0.15) to[out=down,in=130] (0,0);
    \end{tikzpicture}
}
\newcommand\lollydrop{
    \begin{tikzpicture}[centerzero]
        \draw (0,0.2) -- (0,0) to[out=-40,in=up] (0.15,-0.15) arc(0:-180:0.15) to[out=up,in=-130] (0,0);
    \end{tikzpicture}
}
\newcommand\posdotcross{
    \begin{tikzpicture}[centerzero]
        \draw (0.2,-0.2) -- (-0.2,0.2);
        \draw (-0.2,-0.2) -- (0.2,0.2);
        \posdot{0,0};
    \end{tikzpicture}
}
\newcommand\negdotcross{
    \begin{tikzpicture}[centerzero]
        \draw (0.2,-0.2) -- (-0.2,0.2);
        \draw (-0.2,-0.2) -- (0.2,0.2);
        \negdot{0,0};
    \end{tikzpicture}
}

\newcommand\triform{
    \begin{tikzpicture}[centerzero]
        \draw (-0.3,-0.2) -- (0,0.2);
        \draw (0,-0.2) -- (0,0.2);
        \draw (0.3,-0.2) -- (0,0.2);
    \end{tikzpicture}
}
\newcommand\explode{
    \begin{tikzpicture}[centerzero]
        \draw (-0.3,0.2) -- (0,-0.2);
        \draw (0,0.2) -- (0,-0.2);
        \draw (0.3,0.2) -- (0,-0.2);
    \end{tikzpicture}
}
\newcommand\trimor{
    \begin{tikzpicture}[anchorbase]
        \draw (0,0.15) -- (0.13,-0.075) -- (-0.13,-0.075) -- cycle;
        \draw (0,0.15) -- (0,0.3);
        \draw (0.13,-0.075) -- (0.24,-0.185);
        \draw (-0.13,-0.075) -- (-0.24,-0.185);
    \end{tikzpicture}
}
\newcommand\posstrand{
    \begin{tikzpicture}[centerzero]
        \draw (0,-0.2) -- (0,0.2);
        \posaff{0,0};
    \end{tikzpicture}
}
\newcommand\negstrand{
    \begin{tikzpicture}[centerzero]
        \draw (0,-0.2) -- (0,0.2);
        \negaff{0,0};
    \end{tikzpicture}
}

\newcommand\sqmor{
    \begin{tikzpicture}[centerzero]
        \draw (-0.15,-0.15) rectangle (0.15,0.15);
        \draw (-0.3,-0.3) -- (-0.15,-0.15);
        \draw (0.3,-0.3) -- (0.15,-0.15);
        \draw (-0.3,0.3) -- (-0.15,0.15);
        \draw (0.3,0.3) -- (0.15,0.15);
    \end{tikzpicture}
}
\newcommand\pentmor{
    \begin{tikzpicture}[centerzero]
        \draw (0,0.25) -- (-0.24,0.08) -- (-0.15,-0.20) -- (0.15,-0.20) -- (0.24,0.08) -- cycle;
        \draw (0,0.25) -- (0,0.4);
        \draw (-0.24,0.08) -- (-0.35,0.4);
        \draw (0.24,0.08) -- (0.35,0.4);
        \draw (-0.15,-0.2) -- (-0.35,-0.4);
        \draw (0.15,-0.2) -- (0.35,-0.4);
    \end{tikzpicture}
}
\newcommand\Hmor{
    \begin{tikzpicture}[centerzero]
        \draw (-0.3,-0.2) -- (-0.1,0) -- (-0.3,0.2);
        \draw (-0.1,0) -- (0.1,0);
        \draw (0.3,-0.2) -- (0.1,0) -- (0.3,0.2);
    \end{tikzpicture}
}
\newcommand\Imor{
    \begin{tikzpicture}[centerzero]
        \draw (-0.2,-0.3) -- (0,-0.1) -- (0.2,-0.3);
        \draw (-0.2,0.3) -- (0,0.1) -- (0.2,0.3);
        \draw (0,-0.1) -- (0,0.1);
    \end{tikzpicture}
}
\newcommand\hourglass{
    \begin{tikzpicture}[centerzero]
        \draw (-0.15,-0.3) -- (-0.15,-0.25) arc(180:0:0.15) -- (0.15,-0.3);
        \draw (-0.15,0.3) -- (-0.15,0.25) arc(180:360:0.15) -- (0.15,0.3);
    \end{tikzpicture}
}
\newcommand\jail{
    \begin{tikzpicture}[centerzero]
        \draw (-0.15,-0.3) -- (-0.15,0.3);
        \draw (0.15,-0.3) -- (0.15,0.3);
    \end{tikzpicture}
}

\newtheorem{theo}{Theorem}[section]

\newtheorem{prop}[theo]{Proposition}
\newtheorem{lem}[theo]{Lemma}
\newtheorem{cor}[theo]{Corollary}
\newtheorem{conj}[theo]{Conjecture}

\theoremstyle{definition}
\newtheorem{assumption}[theo]{Assumption}
\newtheorem{defin}[theo]{Definition}
\newtheorem{rem}[theo]{Remark}

\numberwithin{equation}{section}
\allowdisplaybreaks

\setenumerate[1]{label=(\alph*)}          

\newtoggle{comments}
\newtoggle{details}
\newtoggle{detailsnote}

\iftoggle{comments}{%
    \usepackage[notref,notcite]{showkeys}   
    \newcommand{\acomments}[1]{
        \ \\
        {\color{red} \textbf{AS:} #1}
        \ \\
    }
    \newcommand{\bcomments}[1]{
        \ \\
        {\color{purple} \textbf{BWW:} #1}
        \ \\
    }
}{%
    \newcommand{\acomments}[1]{}
    \newcommand{\bcomments}[1]{}
}

\iftoggle{details}{%
    \newcommand{\details}[1]{
        \ \\
        {\color{OliveGreen} \textbf{Details:} #1 }
        \ \\
    }
}{%
    \newcommand{\details}[1]{}
}

\begin{document}

\title{Quantum diagrammatics for $F_4$}

\author{Alistair Savage}
\address[A.S.]{
  Department of Mathematics and Statistics \\
  University of Ottawa \\
  Ottawa, ON K1N 6N5, Canada
}
\urladdr{\href{https://alistairsavage.ca}{alistairsavage.ca}, \textrm{\textit{ORCiD}:} \href{https://orcid.org/0000-0002-2859-0239}{orcid.org/0000-0002-2859-0239}}
\email{alistair.savage@uottawa.ca}

\author{Bruce W.\ Westbury}
\address[B.W.]{}
\urladdr{\textrm{\textit{ORCiD}:} \href{https://orcid.org/0000-0002-3793-3360}{orcid.org/0000-0002-3793-3360}}
\email{brucewestbury@gmail.com}

\begin{abstract}
    We introduce a graphical calculus for the representation theory of the quantized enveloping algebra of type $F_4$.  We do this by giving a diagrammatic description of the category of invariant tensors on the 26-dimensional fundamental representation.
\end{abstract}

\subjclass[2020]{18M15, 18M30, 17B25, 17B37}

\keywords{Monoidal category, string diagram, quantized enveloping algebra, exceptional Lie algebra, $F_4$, link invariant}

\ifboolexpr{togl{comments} or togl{details}}{%
  {\color{magenta}DETAILS OR COMMENTS ON}
}{%
}

\maketitle
\thispagestyle{empty}

\tableofcontents

\section{Introduction}

One of the main objectives in quantum invariant theory is to describe the \emph{category of invariant tensors on $V$}, where $V$ is a finite-dimensional module for the quantized enveloping algebra $U_q(\fg)$ of a reductive Lie algebra $\fg$.  This is the full monoidal subcategory of $U_q(\fg)$-mod on $V$ and its dual.  (Throughout this paper we work with finite-dimensional modules of type $1$.)  The basic problem is to give a finite presentation of this monoidal category.  The insight that arose from the construction of the Reshetikhin--Turaev link invariants from the representation theory of $U_q(\fg)$ is that the most effective approach to this problem is to use string diagrams.

The general strategy is that one defines a monoidal category $\cC$ with one generating object (and its dual if $V$ is not self-dual) and finitely many generating morphisms, depicted by simple string diagrams, subject to a finite number of relations on these generating morphisms.  One then defines a full monoidal functor $\bF \colon \cC \to U_q(\fg)\md$ sending the generating object to $V$.  If $V$ has the property that every module is a summand of some tensor product of copies of $V$ and its dual (this is true, for example, when the corresponding representation on $V$ of the associated compact, simply-connected Lie group is faithful),
\details{
    If the corresponding representation of the compact, simply-connected Lie group on $V$ is faithful, then every finite-dimensional module is a summand of $V^k \otimes (V^*)^l$ for some $k,l \in \N$.  See, for example, Theorem III.4.4 of T.\ Br\"ocker and T.\ tom Dieck, \emph{Representations of compact Lie groups}, Graduate Texts in Mathematics 98, Springer-Verlag, New York, 1985.
}
then we have an induced full and essentially surjective functor $\Kar(\bF) \colon \Kar(\cC) \to U_q(\fg)\md$, where $\Kar(\cC)$ is the idempotent completion of the additive envelope of $\cC$.  In this way, we obtain a diagrammatic calculus for the representation theory of $U_q(\fg)$, where arbitrary modules correspond to idempotent endomorphisms in the additive envelope of $\cC$.  Ideally one would also like to have a \emph{complete} set of relations in $\cC$, so that the above functors are also faithful, and hence $\Kar(\bF)$ is an equivalence of categories.

The first example of the above construction is the Temperley--Lieb category, which is the category of invariant tensors for the vector representation of $U_q(\mathfrak{sl}_2)$.  In general type $A$, the diagrammatic calculus is perhaps best encoded in the category associated to the HOMFLYPT link invariant.  This category, originally formulated by Turaev \cite{Tur89}, is a quotient of the category of oriented framed tangles.  The module $V$ in this case is the quantum analogue of the natural $\gl_n$-module.  There is a natural full functor to $U_q(\gl_n)$-mod, and thus we obtain a graphical calculus for that category.  Iwahori--Hecke algebras (more generally, quantum walled Brauer algebras) occur as endomorphism algebras in the diagrammatic category.

In types $BCD$, the graphical calculus is given by the Kauffman category, also introduced in \cite{Tur89}, which is a quotient of the category of unoriented framed tangles.  Again, $V$ is the natural module, and there is a full functor to $U_q(\fg)$-mod, where $\fg$ is a simple Lie algebra of type $BCD$.  In this case, the endomorphism algebras are Birman--Murukami--Wenzl (BMW) algebras.

In \cite{Kup94,Kup96}, Kuperberg gave complete presentations for the categories of invariant tensors for any fundamental representation when $\fg$ is a simple Lie algebra of rank two.  This was then done for arbitrary finite type $A$ in \cite{CKM14} and for arbitrary finite type $C$ in \cite{BERT21}.  Other papers giving partial descriptions include \cite{Wen03} on the minuscule representations in type $E_6$ or $E_7$, \cite{Wes03a} on adjoint representations, and \cite{Wes08} on the spin representation in type $B_3$.

A simpler concept, which we will refer to as \emph{classical} invariant theory, is where we take the enveloping algebra $U(\fg)$ in the place of its quantized version.  This was initially studied without the use of diagrams in Weyl's influential book \cite{Wey39}.  Perhaps the earliest use of diagrams in this context is the use of Brauer algebras in \cite{Bra37} to formulate a type-$BCD$ analogue of Schur--Weyl duality.  In \cite{LZ15}, Lehrer and Zhang defined the Brauer category, where the Brauer algebras appear as endomorphism algebras, and described the kernel of the full functor from the diagrammatic category to the representation category.  Inspired by the language of Feynman diagrams in quantum field theory, invariant tensors for semisimple Lie algebras have also been computed diagrammatically by Cvitanonivi\'c in \cite{Cvi08}.

In type $F_4$, the category of classical invariant tensors for the 26-dimensional simple module, which is the smallest nontrivial simple module, was recently studied in \cite{GSZ21}.  There the authors give a presentation of a diagrammatic monoidal category $\Fcat$, together with a full functor to $U(\fg)$-mod, which is also essentially surjective after passing to $\Kar(\Fcat)$.  The goal of the current paper is to develop a \emph{quantum} version of these diagrammatics.

Given a commutative ring $\kk$, we first define a strict monoidal $\kk$-linear category $\Tcat$ depending on parameters $\delta \in \kk$ and $\phi,\beta,\tau \in \kk^\times$.  This category is generated by a single object $\go$ and five morphisms
\[
    \mergemor \colon \go \otimes \go \to \go,\quad
    \poscross,\ \negcross \colon \go \otimes \go \to \go \otimes \go,\quad
    \cupmor \colon \one \to \go \otimes \go,\quad
    \capmor \colon \go \otimes \go \to \one,
\]
where $\one$ is the unit object.  We impose certain natural relations which imply, in particular, that $\Tcat$ is braided and strict pivotal, and that $\go$ has categorical dimension $\delta$.  We then specialize the ground ring to $\kk = \Q(q)$, fix certain values of the parameters $\delta,\phi,\beta,\tau$, and define a quotient $\Fcat_q$ of $\Tcat$ by a skein relation and three additional relations expressing the diagrams
\[
    \begin{tikzpicture}[centerzero]
        \draw (0.4,-0.4) -- (-0.2,0.2);
        \draw[wipe] (-0.2,-0.2) -- (0.4,0.4);
        \draw (-0.2,-0.2) -- (0.4,0.4);
        \draw (-0.2,-0.4) -- (-0.2,0.4);
    \end{tikzpicture}
    \ ,\qquad
    \sqmor \ , \qquad \pentmor\
\]
in terms of simpler ones.  The category $\Fcat_q$ should be thought of as a quantum version of the category $\Fcat$ introduced in \cite{GSZ21} in the sense that setting $q=1$ recovers $\Fcat$ as defined there.  On the other hand, $\Fcat_q$ should also be thought of as a type $F_4$ analogue of the framed HOMFLYPT skein category (type $A$), the Kauffman category (types $BCD$), and Kuperberg's $G_2$ category \cite{Kup94}.

One of the main contributions of the current paper is the method for \emph{deducing} the additional relations that define $\Fcat_q$ as a quotient of $\Tcat$.  We show, in \cref{sec:2point,sec:specialize}, that these relations are forced by certain assumptions on morphism spaces, namely that
\[
    \jail\ ,\quad \hourglass\ ,\quad \Hmor\ ,\quad \Imor\ ,\quad \poscross + \negcross
\]
form a basis for the space of endomorphisms of $\go^{\otimes 2}$, that the morphism space $\go^{\otimes 3} \to \go^{\otimes 2}$ has dimension at most $15$, and that the morphisms
\[
    \begin{tikzpicture}[centerzero]
        \draw (-0.15,-0.3) -- (-0.15,-0.23) arc(180:0:0.15) -- (0.15,-0.3);
        \draw (-0.3,0.3) -- (0,0.08) -- (0.3,0.3);
        \draw (0,0.3) -- (0,0.08);
    \end{tikzpicture}
    \ ,\quad
    \begin{tikzpicture}[centerzero]
        \draw (-0.2,-0.3) -- (-0.2,0.3);
        \draw (0,0.3) -- (0.15,0) -- (0.3,0.3);
        \draw (0.15,0) -- (0.15,-0.3);
    \end{tikzpicture}
    \ ,\quad
    \begin{tikzpicture}[centerzero]
        \draw (0.2,-0.3) -- (0.2,0.3);
        \draw (0,0.3) -- (-0.15,0) -- (-0.3,0.3);
        \draw (-0.15,0) -- (-0.15,-0.3);
    \end{tikzpicture}
    \ ,\quad
    \begin{tikzpicture}[centerzero]
        \draw (-0.3,0.3) -- (-0.3,0.23) arc(180:360:0.15) -- (0,0.3);
        \draw (0.3,0.3) -- (0.15,0) -- (-0.2,-0.3);
        \draw (0.2,-0.3) -- (0.15,0);
    \end{tikzpicture}
    \ ,\quad
    \begin{tikzpicture}[centerzero]
        \draw (0.3,0.3) -- (0.3,0.23) arc(360:180:0.15) -- (0,0.3);
        \draw (-0.3,0.3) -- (-0.15,0) -- (0.2,-0.3);
        \draw (-0.2,-0.3) -- (-0.15,0);
    \end{tikzpicture}
    \ ,\quad
    \begin{tikzpicture}[centerzero]
        \draw (0,0.3) -- (0,-0.15) -- (-0.15,-0.3);
        \draw (0,-0.15) -- (0.15,-0.3);
        \draw[wipe] (-0.2,0.3) -- (-0.2,0.25) arc(180:360:0.2) -- (0.2,0.3);
        \draw (-0.2,0.3) -- (-0.2,0.25) arc(180:360:0.2) -- (0.2,0.3);
    \end{tikzpicture}
    \ ,\quad
    \begin{tikzpicture}[centerzero]
        \draw (-0.3,0.3) -- (0,0) -- (0.3,0.3);
        \draw (0,0) -- (-0.15,-0.3);
        \draw[wipe] (0,0.3) to[out=-45,in=70] (0.15,-0.3);
        \draw (0,0.3) to[out=-45,in=70] (0.15,-0.3);
    \end{tikzpicture}
    \ ,\quad
    \begin{tikzpicture}[centerzero]
        \draw (0.3,0.3) -- (0,0) -- (-0.3,0.3);
        \draw (0,0) -- (0.15,-0.3);
        \draw[wipe] (0,0.3) to[out=225,in=110] (-0.15,-0.3);
        \draw (0,0.3) to[out=225,in=110] (-0.15,-0.3);
    \end{tikzpicture}
    \ ,\quad
    \begin{tikzpicture}[centerzero]
        \draw (-0.3,0.3) -- (-0.15,0.15) -- (0,0.3);
        \draw (-0.15,0.15) -- (0.15,-0.3);
        \draw[wipe] (0.3,0.3) -- (-0.15,-0.3);
        \draw (0.3,0.3) -- (-0.15,-0.3);
    \end{tikzpicture}
    \ ,\quad
    \begin{tikzpicture}[centerzero]
        \draw (0.3,0.3) -- (0.15,0.15) -- (0,0.3);
        \draw (0.15,0.15) -- (-0.15,-0.3);
        \draw[wipe] (-0.3,0.3) -- (0.15,-0.3);
        \draw (-0.3,0.3) -- (0.15,-0.3);
    \end{tikzpicture}
\]
are linearly independent.  In addition to giving a method for deducing the relations, this approach makes the proof of our main result quite easy.  We show (\cref{lightning}) that there is a full monoidal functor
\[
    \bF_q \colon \Fcat_q \to \Mcat_q,
\]
where $\Mcat_q$ denotes the category of finite-dimensional $U_q(\fg)$-modules of type $1$, and $\fg$ is the simple Lie algebra of type $F_4$.  The generating object $\go$ is sent to the 26-dimensional simple $U_q(\fg)$-module $\sV_q$.  We first define this functor on $\Tcat$, where its existence follows from general representation theory considerations.  Then, since the above-mentioned assumptions are satisfied in $U_q(\fg)$-mod, it follows immediately that the functor factors through $\Fcat_q$.  We expect that this approach can also be used to develop quantum diagrammatics in other exceptional types.

In the classical setting, the 26-dimensional simple $U(\fg)$-module $V$ can be naturally identified with the traceless part of the Albert algebra which, over the complex numbers, is the unique exceptional Jordan algebra.  The generating morphisms of the diagrammatic category $\Fcat$ defined in \cite{GSZ21} are given in terms of the multiplication and trace on the Albert algebra, so that $\Fcat$ can be viewed as a diagrammatic calculus for this algebra.  The category $\Fcat_q$ of the current paper can then be thought of as a graphical calculus for a quantum analogue of the Albert algebra.  To the best of our knowledge, this quantum analogue has not appeared in the literature.

Since the category $\Mcat_q$ is idempotent complete, we have an induced functor
\[
    \Kar(\bF_q) \colon \Kar(\Fcat_q) \to \Mcat_q.
\]
This functor is essentially surjective, which implies that every finite-dimensional $U_q(\fg)$-module of type $1$ is the image of $\bF_q(e)$ for some idempotent endomorphism $e$ in $\Fcat_q$.  We illustrate this phenomenon in \cref{sec:idempotents} by giving an explicit decomposition of the identity morphism of $\go^{\otimes 2}$ as a sum of orthogonal idempotents corresponding to the decomposition of $\sV_q^{\otimes 2}$ as a sum of simple modules.

It follows from the fullness of $\bF_q$ that we have surjective algebra homomorphisms
\[
    \End_{\Fcat_q}(\go^{\otimes k}) \twoheadrightarrow \End_{U_q(\fg)}(\sV_q^{\otimes k}),\qquad k \in \N.
\]
In this sense, the endomorphism algebras in $\Fcat_q$ should be thought of as $F_4$ analogues of quantum walled Brauer algebras and BMW algebras.  In addition, the diagrammatic calculus developed in the current paper is a first step towards a set of tools for computing the type $F_4$ Reshetikhin--Turaev invariant.

It is natural to ask if the functor $\bF_q$ is faithful.  Unfortunately, we do not know the answer to this question.  A weaker assertion would be that the kernel of $\bF_q$ consists of negligible morphisms, in which case $\bF_q$ would induce an equivalence of categories between $\Mcat_q$ and the semisimplification of $\Fcat_q$, which is the quotient of $\Fcat_q$ by the tensor ideal of negligible morphisms.  This would follow if one can show that any closed diagram in $\Fcat_q$ can be reduced to a scalar multiple of the empty diagram; see \cref{acai}.  Answering these questions in an important direction for future research.

As a final application of the results of the current paper, we define, in \cref{sec:affine}, an affine version, $\AFcat_q$, of the category $\Fcat_q$.  The category $\AFcat_q$ corresponds to considering string diagrams on a cylinder, as opposed to in the plane.  It acts naturally on the category of $U_q(\fg)$-modules and gives a way to construct elements in the center of $U_q(\fg)$.  We stress that the definition of the affine category $\AFcat_q$, which follows a general affinization procedure developed in \cite{MS21}, \emph{requires} that we work in the quantum setting; it is not available to us using the non-quantum diagrammatic category developed in \cite{GSZ21}.  The affine category suggests several other directions of further research, including connections to annular Reshetikhin--Turaev invariants and an $F_4$ skein algebra of the torus (see \cite{MS17}).

\subsection*{SageMath computations}

Throughout the paper, many direct computations are performed using the open-source mathematics software system SageMath \cite{sagemath}.  For the reader interested in verifying these, all such computations are included in a SageMath notebook that can be found on the \href{https://arxiv.org/abs/2204.11976}{arXiv} as an ancillary file.

\subsection*{Acknowledgements}

The research of A.S.\ was supported by NSERC Discovery Grant RGPIN-2017-03854.

\section{The diagrammatic categories\label{sec:Tdef}}

In this section we introduce our diagrammatic categories of interest and deduce some relations that follow from the definitions.  We fix a commutative ring $\kk$.  All categories are $\kk$-linear and all algebras and tensor products are over $\kk$ unless otherwise specified.  We let $\one$ denote the unit object of a monoidal category.  For objects $X$ and $Y$ in a category $\cC$, we denote by $\cC(X,Y)$ the $\kk$-module of morphisms from $X$ to $Y$.

\begin{defin} \label{Tdef}
    Fix $\delta \in \kk$ and $\phi,\beta,\tau \in \kk^\times$.  Let $\Tcat$ be the strict $\kk$-linear monoidal category generated by the object $\go$ and generating morphisms
    \begin{equation} \label{lego}
        \mergemor \colon \go \otimes \go \to \go,\quad
        \poscross,\ \negcross \colon \go \otimes \go \to \go \otimes \go,\quad
        \cupmor \colon \one \to \go \otimes \go,\quad
        \capmor \colon \go \otimes \go \to \one,
    \end{equation}
    subject to the following relations:
    \begin{gather} \label{vortex}
        \begin{tikzpicture}[centerzero]
            \draw (-0.3,-0.4) -- (-0.3,0) arc(180:0:0.15) arc(180:360:0.15) -- (0.3,0.4);
        \end{tikzpicture}
        =
        \begin{tikzpicture}[centerzero]
            \draw (0,-0.4) -- (0,0.4);
        \end{tikzpicture}
        =
        \begin{tikzpicture}[centerzero]
            \draw (-0.3,0.4) -- (-0.3,0) arc(180:360:0.15) arc(180:0:0.15) -- (0.3,-0.4);
        \end{tikzpicture}
        \ ,\quad
        \splitmor
        :=
        \begin{tikzpicture}[anchorbase]
            \draw (-0.4,0.2) to[out=down,in=180] (-0.2,-0.2) to[out=0,in=225] (0,0);
            \draw (0,0) -- (0,0.2);
            \draw (0.3,-0.3) -- (0,0);
        \end{tikzpicture}
        =
        \begin{tikzpicture}[anchorbase]
            \draw (0.4,0.2) to[out=down,in=0] (0.2,-0.2) to[out=180,in=-45] (0,0);
            \draw (0,0) -- (0,0.2);
            \draw (-0.3,-0.3) -- (0,0);
        \end{tikzpicture}
        \ ,\quad
        \begin{tikzpicture}[centerzero]
            \draw (-0.2,-0.3) -- (-0.2,-0.1) arc(180:0:0.2) -- (0.2,-0.3);
            \draw[wipe] (-0.3,0.3) \braiddown (0,-0.3);
            \draw (-0.3,0.3) \braiddown (0,-0.3);
        \end{tikzpicture}
        =
        \begin{tikzpicture}[centerzero]
            \draw (-0.2,-0.3) -- (-0.2,-0.1) arc(180:0:0.2) -- (0.2,-0.3);
            \draw[wipe] (0.3,0.3) \braiddown (0,-0.3);
            \draw (0.3,0.3) \braiddown (0,-0.3);
        \end{tikzpicture}
        \ ,\quad
        \begin{tikzpicture}[centerzero]
            \draw (-0.3,0.3) \braiddown (0,-0.3);
            \draw[wipe] (-0.2,-0.3) -- (-0.2,-0.1) arc(180:0:0.2) -- (0.2,-0.3);
            \draw (-0.2,-0.3) -- (-0.2,-0.1) arc(180:0:0.2) -- (0.2,-0.3);
        \end{tikzpicture}
        =
        \begin{tikzpicture}[centerzero]
            \draw (0.3,0.3) \braiddown (0,-0.3);
            \draw[wipe] (-0.2,-0.3) -- (-0.2,-0.1) arc(180:0:0.2) -- (0.2,-0.3);
            \draw (-0.2,-0.3) -- (-0.2,-0.1) arc(180:0:0.2) -- (0.2,-0.3);
        \end{tikzpicture}
        \ ,
        \\ \label{venom}
        \begin{tikzpicture}[centerzero]
            \draw (0.2,-0.4) to[out=135,in=down] (-0.15,0) to[out=up,in=-135] (0.2,0.4);
            \draw[wipe] (-0.2,-0.4) to[out=45,in=down] (0.15,0) to[out=up,in=-45] (-0.2,0.4);
            \draw (-0.2,-0.4) to[out=45,in=down] (0.15,0) to[out=up,in=-45] (-0.2,0.4);
        \end{tikzpicture}
        =
        \begin{tikzpicture}[centerzero]
            \draw (-0.15,-0.4) -- (-0.15,0.4);
            \draw (0.15,-0.4) -- (0.15,0.4);
        \end{tikzpicture}
        =
        \begin{tikzpicture}[centerzero]
            \draw (-0.2,-0.4) to[out=45,in=down] (0.15,0) to[out=up,in=-45] (-0.2,0.4);
            \draw[wipe] (0.2,-0.4) to[out=135,in=down] (-0.15,0) to[out=up,in=-135] (0.2,0.4);
            \draw (0.2,-0.4) to[out=135,in=down] (-0.15,0) to[out=up,in=-135] (0.2,0.4);
        \end{tikzpicture}
        \ ,\quad
        \begin{tikzpicture}[centerzero]
            \draw (0.3,-0.4) -- (-0.3,0.4);
            \draw[wipe] (0,-0.4) to[out=135,in=down] (-0.25,0) to[out=up,in=-135] (0,0.4);
            \draw (0,-0.4) to[out=135,in=down] (-0.25,0) to[out=up,in=-135] (0,0.4);
            \draw[wipe] (-0.3,-0.4) -- (0.3,0.4);
            \draw (-0.3,-0.4) -- (0.3,0.4);
        \end{tikzpicture}
        =
        \begin{tikzpicture}[centerzero]
            \draw (0.3,-0.4) -- (-0.3,0.4);
            \draw[wipe] (0,-0.4) to[out=45,in=down] (0.25,0) to[out=up,in=-45] (0,0.4);
            \draw (0,-0.4) to[out=45,in=down] (0.25,0) to[out=up,in=-45] (0,0.4);
            \draw[wipe] (-0.3,-0.4) -- (0.3,0.4);
            \draw (-0.3,-0.4) -- (0.3,0.4);
        \end{tikzpicture}
        \ ,\quad
        \begin{tikzpicture}[anchorbase,scale=0.8]
            \draw (-0.4,0.6) -- (-0.4,0.1) -- (0.4,-0.5);
            \draw[wipe] (-0.4,-0.5) -- (0.2,0.3) -- (0.4,0.1) -- (0,-0.5);
            \draw[wipe] (0.2,0.3) -- (0.2,0.6);
            \draw (-0.4,-0.5) -- (0.2,0.3) -- (0.4,0.1) -- (0,-0.5);
            \draw (0.2,0.3) -- (0.2,0.6);
        \end{tikzpicture}
        =
        \begin{tikzpicture}[anchorbase]
            \draw (0.2,-0.4) -- (0.2,0) -- (-0.2,0.4);
            \draw[wipe] (-0.4,-0.4) -- (-0.2,-0.2) -- (-0.2,0) -- (0.2,0.4);
            \draw (-0.4,-0.4) -- (-0.2,-0.2) -- (-0.2,0) -- (0.2,0.4);
            \draw (0,-0.4) -- (-0.2,-0.2);
        \end{tikzpicture}
        \ ,\quad
        \begin{tikzpicture}[anchorbase,scale=0.8]
            \draw (-0.4,0.5) -- (0.2,-0.3) -- (0.4,-0.1) -- (0,0.5);
            \draw (0.2,-0.3) -- (0.2,-0.6);
            \draw[wipe] (-0.4,-0.6) -- (-0.4,-0.1) -- (0.4,0.5);
            \draw (-0.4,-0.6) -- (-0.4,-0.1) -- (0.4,0.5);
        \end{tikzpicture}
        =
        \begin{tikzpicture}[anchorbase]
            \draw (-0.4,0.4) -- (-0.2,0.2) -- (-0.2,0) -- (0.2,-0.4);
            \draw (0,0.4) -- (-0.2,0.2);
            \draw[wipe] (0.2,0.4) -- (0.2,0) -- (-0.2,-0.4);
            \draw (0.2,0.4) -- (0.2,0) -- (-0.2,-0.4);
        \end{tikzpicture}
        \ ,
        \\ \label{chess}
        \begin{tikzpicture}[anchorbase]
            \draw (-0.15,0) to[out=down,in=135] (0.15,-0.4);
            \draw[wipe] (-0.15,-0.4) to[out=45,in=down] (0.15,0) arc(0:180:0.15);
            \draw (-0.15,-0.4) to[out=45,in=down] (0.15,0) arc(0:180:0.15);
        \end{tikzpicture}
        = \beta\ \capmor
        \ ,\quad
        \begin{tikzpicture}[anchorbase]
            \draw (0.2,-0.5) to [out=135,in=down] (-0.15,-0.2) to[out=up,in=-135] (0,0);
            \draw[wipe] (-0.2,-0.5) to[out=45,in=down] (0.15,-0.2);
            \draw (-0.2,-0.5) to[out=45,in=down] (0.15,-0.2) to[out=up,in=-45] (0,0) -- (0,0.2);
        \end{tikzpicture}
        = \tau \, \mergemor
        \ ,\quad
        \begin{tikzpicture}[centerzero]
            \draw  (0,-0.4) -- (0,-0.2) to[out=45,in=down] (0.15,0) to[out=up,in=-45] (0,0.2) -- (0,0.4);
            \draw (0,-0.2) to[out=135,in=down] (-0.15,0) to[out=up,in=-135] (0,0.2);
        \end{tikzpicture}
        = \phi\
        \begin{tikzpicture}[centerzero]
            \draw(0,-0.4) -- (0,0.4);
        \end{tikzpicture}
        \ ,\quad
        \bubble = \delta 1_\one,
        \quad
        \lollydrop = 0.
    \end{gather}
\end{defin}

\begin{rem} \label{degenerate}
    When $\beta = \tau = 1$, the category $\Tcat_{\phi,\delta}$ of \cite[Def.~2.1]{GSZ21} is the quotient of the category from \cref{Tdef} by the relation $\poscross = \negcross$.  (In \cite{GSZ21}, the parameter $\phi$ is denoted $\alpha$.)  In this sense, the category from \cref{Tdef} should be thought of as a quantum analogue of the one in \cite[Def.~2.1]{GSZ21}.
\end{rem}

\begin{rem}
    The relations in \cref{Tdef} all have conceptual category-theoretic meanings:
    \begin{itemize}
        \item The relations \cref{venom} and the last two equalities in \cref{vortex} correspond to the fact that the crossings equip $\Tcat$ with the structure of a braided monoidal category.
        \item The first two equalities in \cref{vortex} assert that the generating object $\go$ is self-dual.
        \item In \cref{sec:functor}, we will define a functor from $\Tcat$ to the category of modules over the quantized enveloping algebra of type $F_4$.  This functor sends $\go$ to a simple module $\sV_q$.  By Schur's lemma, the endomorphism algebra of $\sV_q$ consists of scalar multiples of the identity morphism, which implies that first and third relations in \cref{chess} will be satisfied in the image for some $\phi,\beta$.  Similarly, we know that $\Hom(\sV_q^{\otimes 2}, \sV_q)$ is one-dimensional and so the second relation in \cref{chess} will also be satisfied for some $\tau$.  The fifth relation in \cref{chess} corresponds to the fact that we want no nonzero morphisms from $\one$ to $\go$
        \item The fourth relation in \cref{chess} states that $\go$ has categorical dimension $\delta$.  This will correspond to the quantum dimension of the module $\sV_q$.
    \end{itemize}
    For a similar discussion in the case $\beta = \tau=1$ (see \cref{degenerate}), we refer the reader to \cite[Rem.~2.3]{GSZ21}.
\end{rem}

\begin{prop} \label{windy}
    The following relations hold in $\Tcat$:
    \begin{gather} \label{topsy}
        \begin{tikzpicture}[anchorbase]
            \draw (-0.4,-0.2) to[out=up,in=180] (-0.2,0.2) to[out=0,in=135] (0,0);
            \draw (0,0) -- (0,-0.2);
            \draw (0.3,0.3) -- (0,0);
        \end{tikzpicture}
        =
        \mergemor
        =
        \begin{tikzpicture}[anchorbase]
            \draw (0.4,-0.2) to[out=up,in=0] (0.2,0.2) to[out=180,in=45] (0,0);
            \draw (0,0) -- (0,-0.2);
            \draw (-0.3,0.3) -- (0,0);
        \end{tikzpicture}
        \ ,\quad
        \triform
        :=
        \begin{tikzpicture}[centerzero]
          \draw (-0.2,-0.2) to (0,0);
          \draw (0.2,-0.2) to (0,0);
          \draw (0,0) arc(0:180:0.2) -- (-0.4,-0.2);
        \end{tikzpicture}
        =
        \begin{tikzpicture}[centerzero]
          \draw (-0.2,-0.2) to (0,0);
          \draw (0.2,-0.2) to (0,0);
          \draw (0,0) arc(180:0:0.2) -- (0.4,-0.2);
        \end{tikzpicture}
        \ ,\quad
        \explode
        :=
        \begin{tikzpicture}[centerzero]
          \draw (-0.2,0.2) to (0,0);
          \draw (0.2,0.2) to (0,0);
          \draw (0,0) arc(360:180:0.2) to (-0.4,0.2);
        \end{tikzpicture}
        =
        \begin{tikzpicture}[centerzero]
          \draw (-0.2,0.2) to (0,0);
          \draw (0.2,0.2) to (0,0);
          \draw (0,0) arc(180:360:0.2) to (0.4,0.2);
        \end{tikzpicture}
        \ ,
        \\ \label{turvy}
        \begin{tikzpicture}[centerzero]
            \draw (-0.2,0.3) -- (-0.2,0.1) arc(180:360:0.2) -- (0.2,0.3);
            \draw[wipe] (-0.3,-0.3) to[out=up,in=down] (0,0.3);
            \draw (-0.3,-0.3) to[out=up,in=down] (0,0.3);
        \end{tikzpicture}
        =
        \begin{tikzpicture}[centerzero]
            \draw (-0.2,0.3) -- (-0.2,0.1) arc(180:360:0.2) -- (0.2,0.3);
            \draw[wipe] (0.3,-0.3) to[out=up,in=down] (0,0.3);
            \draw (0.3,-0.3) to[out=up,in=down] (0,0.3);
        \end{tikzpicture}
        \ ,\quad
        \begin{tikzpicture}[centerzero]
            \draw (-0.3,-0.3) to[out=up,in=down] (0,0.3);
            \draw[wipe] (-0.2,0.3) -- (-0.2,0.1) arc(180:360:0.2) -- (0.2,0.3);
            \draw (-0.2,0.3) -- (-0.2,0.1) arc(180:360:0.2) -- (0.2,0.3);
        \end{tikzpicture}
        =
        \begin{tikzpicture}[centerzero]
            \draw (0.3,-0.3) to[out=up,in=down] (0,0.3);
            \draw[wipe] (-0.2,0.3) -- (-0.2,0.1) arc(180:360:0.2) -- (0.2,0.3);
            \draw (-0.2,0.3) -- (-0.2,0.1) arc(180:360:0.2) -- (0.2,0.3);
        \end{tikzpicture}
        \ ,\quad
        \begin{tikzpicture}[anchorbase]
            \draw (-0.15,-0.4) to[out=45,in=down] (0.15,0) arc(0:180:0.15);
            \draw[wipe] (-0.15,0) to[out=down,in=135] (0.15,-0.4);
            \draw (-0.15,0) to[out=down,in=135] (0.15,-0.4);
        \end{tikzpicture}
        = \beta^{-1}\ \capmor
        \ ,\quad
        \begin{tikzpicture}[anchorbase]
            \draw (-0.2,-0.5) to[out=45,in=down] (0.15,-0.2) to[out=up,in=-45] (0,0) -- (0,0.2);
            \draw[wipe] (0.2,-0.5) to [out=135,in=down] (-0.15,-0.2) to[out=up,in=-135] (0,0);
            \draw (0.2,-0.5) to [out=135,in=down] (-0.15,-0.2) to[out=up,in=-135] (0,0);
        \end{tikzpicture}
        = \tau^{-1}\, \mergemor
        \ ,
        \\ \label{pokey}
        \begin{tikzpicture}[anchorbase,scale=0.8]
            \draw (0.4,0.5) -- (-0.2,-0.3) -- (-0.4,-0.1) -- (0,0.5);
            \draw (-0.2,-0.3) -- (-0.2,-0.6);
            \draw[wipe] (0.4,-0.6) -- (0.4,-0.1) -- (-0.4,0.5);
            \draw (0.4,-0.6) -- (0.4,-0.1) -- (-0.4,0.5);
        \end{tikzpicture}
        =
        \begin{tikzpicture}[anchorbase]
            \draw (0.4,0.4) -- (0.2,0.2) -- (0.2,0) -- (-0.2,-0.4);
            \draw (0,0.4) -- (0.2,0.2);
            \draw[wipe] (-0.2,0.4) -- (-0.2,0) -- (0.2,-0.4);
            \draw (-0.2,0.4) -- (-0.2,0) -- (0.2,-0.4);
        \end{tikzpicture}
        ,\quad
        \begin{tikzpicture}[anchorbase,scale=0.8]
            \draw (0.4,0.6) -- (0.4,0.1) -- (-0.4,-0.5);
            \draw[wipe] (0.4,-0.5) -- (-0.2,0.3) -- (-0.4,0.1) -- (0,-0.5);
            \draw[wipe] (-0.2,0.3) -- (-0.2,0.6);
            \draw (0.4,-0.5) -- (-0.2,0.3) -- (-0.4,0.1) -- (0,-0.5);
            \draw (-0.2,0.3) -- (-0.2,0.6);
        \end{tikzpicture}
        =
        \begin{tikzpicture}[anchorbase]
            \draw (-0.2,-0.4) -- (-0.2,0) -- (0.2,0.4);
            \draw[wipe] (0.4,-0.4) -- (0.2,-0.2) -- (0.2,0) -- (-0.2,0.4);
            \draw (0.4,-0.4) -- (0.2,-0.2) -- (0.2,0) -- (-0.2,0.4);
            \draw (0,-0.4) -- (0.2,-0.2);
        \end{tikzpicture}
        ,
        \\ \label{wobbly}
        \begin{tikzpicture}[anchorbase]
            \draw (-0.2,0.2) -- (0.2,-0.2);
            \draw[wipe] (-0.4,0.2) to[out=down,in=225,looseness=2] (0,0) to[out=45,in=up,looseness=2] (0.4,-0.2);
            \draw (-0.4,0.2) to[out=down,in=225,looseness=2] (0,0) to[out=45,in=up,looseness=2] (0.4,-0.2);
        \end{tikzpicture}
        =
        \negcross
        =
        \begin{tikzpicture}[anchorbase]
            \draw (0.4,0.2) to[out=down,in=-45,looseness=2] (0,0) to[out=135,in=up,looseness=2] (-0.4,-0.2);
            \draw[wipe] (0.2,0.2) -- (-0.2,-0.2);
            \draw (0.2,0.2) -- (-0.2,-0.2);
        \end{tikzpicture}
        \ ,\quad
        \begin{tikzpicture}[anchorbase]
            \draw (-0.4,0.2) to[out=down,in=225,looseness=2] (0,0) to[out=45,in=up,looseness=2] (0.4,-0.2);
            \draw[wipe] (-0.2,0.2) -- (0.2,-0.2);
            \draw (-0.2,0.2) -- (0.2,-0.2);
        \end{tikzpicture}
        =
        \poscross
        =
        \begin{tikzpicture}[anchorbase]
            \draw (0.2,0.2) -- (-0.2,-0.2);
            \draw[wipe] (0.4,0.2) to[out=down,in=-45,looseness=2] (0,0) to[out=135,in=up,looseness=2] (-0.4,-0.2);
            \draw (0.4,0.2) to[out=down,in=-45,looseness=2] (0,0) to[out=135,in=up,looseness=2] (-0.4,-0.2);
        \end{tikzpicture}
        \ .
    \end{gather}
\end{prop}

\begin{proof}
    The first and second equalities in \cref{topsy} follow immediately from the first four equalities in \cref{vortex}.  Then, using the first and second equalities in \cref{topsy}, we have
    \[
        \begin{tikzpicture}[centerzero]
          \draw (-0.2,-0.2) to (0,0);
          \draw (0.2,-0.2) to (0,0);
          \draw (0,0) arc(0:180:0.2) -- (-0.4,-0.2);
        \end{tikzpicture}
        =
        \begin{tikzpicture}[centerzero]
            \draw (0,-0.2) to (0,0) to[out=45,in=up,looseness=2] (0.3,-0.2);
            \draw (0,0) to[out=135,in=up,looseness=2] (-0.3,-0.2);
        \end{tikzpicture}
        =
        \begin{tikzpicture}[centerzero]
          \draw (-0.2,-0.2) to (0,0);
          \draw (0.2,-0.2) to (0,0);
          \draw (0,0) arc(180:0:0.2) -- (0.4,-0.2);
        \end{tikzpicture}
        \ ,
    \]
    proving the fourth equality in \cref{topsy}.  The proof of the sixth equality in \cref{topsy} is analogous.

    To prove the first equality in \cref{turvy}, we use \cref{vortex} to compute
    \[
        \begin{tikzpicture}[centerzero]
            \draw (-0.2,0.3) -- (-0.2,0.1) arc(180:360:0.2) -- (0.2,0.3);
            \draw[wipe] (-0.3,-0.3) \braidup (0,0.3);
            \draw (-0.3,-0.3) \braidup (0,0.3);
        \end{tikzpicture}
        =
        \begin{tikzpicture}[anchorbase]
            \draw (-1,0.5) -- (-1,0.2) arc(180:360:0.2) arc(180:0:0.2) arc(180:360:0.2) -- (0.2,0.5);
            \draw[wipe] (-0.3,-0.3) \braidup (0,0.5);
            \draw (-0.3,-0.3) \braidup (0,0.5);
        \end{tikzpicture}
        =
        \begin{tikzpicture}[anchorbase]
            \draw (-1,0.5) -- (-1,0.2) arc(180:360:0.2) arc(180:0:0.2) arc(180:360:0.2) -- (0.2,0.5);
            \draw[wipe] (-0.5,-0.3) \braidup (-0.8,0.5);
            \draw (-0.5,-0.3) \braidup (-0.8,0.5);
        \end{tikzpicture}
        =
        \begin{tikzpicture}[centerzero]
            \draw (-0.2,0.3) -- (-0.2,0.1) arc(180:360:0.2) -- (0.2,0.3);
            \draw[wipe] (0.3,-0.3) to[out=up,in=down] (0,0.3);
            \draw (0.3,-0.3) to[out=up,in=down] (0,0.3);
        \end{tikzpicture}
        \ .
    \]
    The proof of the second equality in \cref{turvy} is analogous.  The third and fourth equalities in \cref{turvy} follow from the first and second equalities in \cref{chess}, respectively, by composing on the bottom with $\negcross$ and using the second equality in \cref{venom}.

    To prove the first equation in \cref{pokey}, we compose the last equality in \cref{venom} on the bottom with $\negcross$ and on the top with
    $
        \begin{tikzpicture}[centerzero]
            \draw (-0.2,-0.2) -- (0,0.2);
            \draw (0,-0.2) -- (0.2,0.2);
            \draw[wipe] (0.2,-0.2) -- (-0.2,0.2);
            \draw (0.2,-0.2) -- (-0.2,0.2);
        \end{tikzpicture}
    $
    , and then use the first two equalities in \cref{venom}.  The proof of the second equation in \cref{pokey} is analogous.

    Next we have
    \[
        \begin{tikzpicture}[anchorbase]
            \draw (-0.4,-0.3) -- (-0.4,0) arc(180:0:0.2) arc(180:360:0.2) -- (0.4,0.3);
            \draw[wipe] (-0.2,-0.3) \braidup (0.2,0.3);
            \draw (-0.2,-0.3) \braidup (0.2,0.3);
        \end{tikzpicture}
        \overset{\cref{vortex}}{=}
        \begin{tikzpicture}[anchorbase]
            \draw (-0.4,-0.3) -- (-0.4,0) arc(180:0:0.2) arc(180:360:0.2) -- (0.4,0.3);
            \draw[wipe] (-0.2,-0.3) \braidup (-0.6,0.3);
            \draw (-0.2,-0.3) \braidup (-0.6,0.3);
        \end{tikzpicture}
        \overset{\cref{vortex}}{=}
        \negcross
        \qquad \text{and} \qquad
        \begin{tikzpicture}[anchorbase]
            \draw (0.4,-0.3) -- (0.4,0) arc(0:180:0.2) arc(360:180:0.2) -- (-0.4,0.3);
            \draw[wipe] (0.2,-0.3) \braidup (-0.2,0.3);
            \draw (0.2,-0.3) \braidup (-0.2,0.3);
        \end{tikzpicture}
        \overset{\cref{vortex}}{=}
        \begin{tikzpicture}[anchorbase]
            \draw (0.4,-0.3) -- (0.4,0) arc(0:180:0.2) arc(360:180:0.2) -- (-0.4,0.3);
            \draw[wipe] (0.2,-0.3) \braidup (0.6,0.3);
            \draw (0.2,-0.3) \braidup (0.6,0.3);
        \end{tikzpicture}
        \overset{\cref{vortex}}{=}
        \poscross,
    \]
    proving the second and third equalities in \cref{wobbly}.  The proofs of the first and fourth equalities in \cref{wobbly} are analogous.
\end{proof}

It follows from \cref{vortex}, \cref{topsy}, and the first two equalities in \cref{turvy} that the cups and caps equip $\Tcat$ with the structure of a \emph{strict pivotal} category.  Intuitively, this means that morphisms are invariant under ambient isotopy fixing the boundary.  Thus, for example, it makes sense to allow horizontal strands in diagrams:
\begin{equation} \label{honeycomb}
    \Hmor
    :=
    \begin{tikzpicture}[anchorbase]
        \draw (-0.4,-0.4) -- (-0.4,0) -- (-0.2,0.2) -- (0.2,-0.2) -- (0.4,0) -- (0.4,0.4);
        \draw (-0.2,0.2) -- (-0.2,0.4);
        \draw (0.2,-0.2) -- (0.2,-0.4);
    \end{tikzpicture}
    =
    \begin{tikzpicture}[anchorbase]
        \draw (0.4,-0.4) -- (0.4,0) -- (0.2,0.2) -- (-0.2,-0.2) -- (-0.4,0) -- (-0.4,0.4);
        \draw (0.2,0.2) -- (0.2,0.4);
        \draw (-0.2,-0.2) -- (-0.2,-0.4);
    \end{tikzpicture}
    \ .
\end{equation}
Since $\Tcat$ is also braided monoidal, the first relation in \cref{chess} then implies that it is a \emph{ribbon category}.  In addition, since the object $\go$ is self-dual, the cups and caps yield natural isomorphisms
\begin{equation} \label{twirl}
    \Tcat(\go^{\otimes m}, \go^{\otimes n})
    \cong \Tcat(\go^{\otimes (m+n)},\one),\qquad
    n,m \in \N.
\end{equation}
In what follows, we will sometimes use an equation number to refer to a rotated version of that equation.  For example, we will use the equation number \cref{chess} to refer to the relations
\[
    \begin{tikzpicture}[anchorbase,rotate=180]
        \draw (-0.15,0) to[out=down,in=135] (0.15,-0.4);
        \draw[wipe] (-0.15,-0.4) to[out=45,in=down] (0.15,0) arc(0:180:0.15);
        \draw (-0.15,-0.4) to[out=45,in=down] (0.15,0) arc(0:180:0.15);
    \end{tikzpicture}
    = \beta\ \cupmor
    \ ,\quad
    \begin{tikzpicture}[anchorbase,rotate=180]
        \draw (0.2,-0.5) to [out=135,in=down] (-0.15,-0.2) to[out=up,in=-135] (0,0);
        \draw[wipe] (-0.2,-0.5) to[out=45,in=down] (0.15,-0.2);
        \draw (-0.2,-0.5) to[out=45,in=down] (0.15,-0.2) to[out=up,in=-45] (0,0) -- (0,0.2);
    \end{tikzpicture}
    = \tau \, \splitmor
    \ ,\quad
    \lolly = 0.
\]

We say that a $\kk$-linear category is \emph{trivial} if all morphisms are equal to zero.  The category $\Tcat$ is trivial if and only if $1_\go = 0$.

\begin{lem} \label{forked}
    The category $\Tcat$ is trivial unless $\beta = \tau^2$.
\end{lem}

\begin{proof}
    We have
    \[
        \beta\, \mergemor
        \overset{\cref{chess}}{=}
        \begin{tikzpicture}[anchorbase]
        	\draw (0,0.45) to (0,0.225);
        	\draw (0.225,-0.15) to [out=0,in=-90](0.375,0);
        	\draw (0.375,0) to [out=90,in=0](0.225,0.15);
        	\draw (0.225,0.15) to [out=180,in=90](0,-0.225);
        	\draw[wipe] (0,0.225) to [out=-90,in=180] (0.225,-0.15);
        	\draw (0,0.225) to [out=-90,in=180] (0.225,-0.15);
        	\draw (0,-0.225) to (0,-0.3) to (0.2,-0.5);
            \draw (0,-0.3) to (-0.2,-0.5);
        \end{tikzpicture}
        =
        \begin{tikzpicture}[anchorbase]
            \draw (0,0) to (0.2,-0.4);
            \draw (0,0) to (-0.2,-0.4);
            \draw[wipe] (-0.3,0) to[out=down,in=down,looseness=1.5] (0.3,0);
            \draw (0,0.5) to[out=down,in=up] (-0.3,0) to[out=down,in=down,looseness=1.5] (0.3,0) to[out=up,in=up,looseness=1.5] (0,0);
        \end{tikzpicture}
        =
        \begin{tikzpicture}[anchorbase]
            \draw (0.5,0) to[out=45,in=0] (0.4,0.3) to[out=180,in=up] (0,-0.5);
            \draw (0.5,0) to[out=135,in=up,looseness=2] (0.2,-0.5);
            \draw[wipe] (0,0.5) to[out=down,in=down,looseness=2] (0.5,0);
            \draw (0,0.5) to[out=down,in=down,looseness=2] (0.5,0);
        \end{tikzpicture}
        =
        \begin{tikzpicture}[anchorbase]
            \draw (0,0.5) to (0,0.3);
            \draw (0.4,-0.1) to[out=0,in=0,looseness=1.5] (0.4,0.1) to[out=180,in=90] (0.2,-0.5);
            \draw (0.4,0.3) to[out=180,in=90] (0,-0.5);
            \draw[wipe] (0,0.3) to[out=-45,in=180] (0.4,-0.1);
            \draw (0,0.3) to[out=-45,in=180] (0.4,-0.1);
            \draw[wipe] (0,0.3) to (-0.1,0.2) to[out=-135,in=180] (0.4,-0.3) to[out=0,in=0,looseness=1.5] (0.4,0.3);
            \draw (0,0.3) to (-0.1,0.2) to[out=-135,in=180] (0.4,-0.3) to[out=0,in=0,looseness=1.5] (0.4,0.3);
        \end{tikzpicture}
        \overset{\cref{chess}}{=} \beta\,
        \begin{tikzpicture}[anchorbase]
            \draw (0,0.5) to (0,0.3);
            \draw (0.5,0) to[out=90,in=0] (0.3,0.2) to[out=180,in=90] (-0.1,-0.5);
            \draw[wipe] (0,0.3) to[out=-45,in=90] (0.2,-0.5);
            \draw (0,0.3) to[out=-45,in=90] (0.2,-0.5);
            \draw[wipe] (0,0.3) to[out=-135,in=90] (-0.1,0.1) to[out=-90,in=-90,looseness=1.5] (0.5,0);
            \draw (0,0.3) to[out=-135,in=90] (-0.1,0.1) to[out=-90,in=-90,looseness=1.5] (0.5,0);
        \end{tikzpicture}
        \ .
    \]
    Multiplying by $\beta^{-1}$ and composing on the bottom with $\poscross$ then gives
    \[
        \begin{tikzpicture}[anchorbase]
            \draw (0.2,-0.5) to [out=135,in=down] (-0.15,-0.2) to[out=up,in=-135] (0,0);
            \draw[wipe] (-0.2,-0.5) to[out=45,in=down] (0.15,-0.2);
            \draw (-0.2,-0.5) to[out=45,in=down] (0.15,-0.2) to[out=up,in=-45] (0,0) -- (0,0.2);
        \end{tikzpicture}
        =
        \begin{tikzpicture}[anchorbase]
            \draw (0,0.4) to (0,0.2) to[out=-45,in=45,looseness=1.3] (-0.2,-0.4);
            \draw (0.4,0) to[out=90,in=0] (0.3,0.1) to[out=180,in=90] (0.25,-0.4);
            \draw[wipe] (0,0.2) to[out=-135,in=90] (-0.1,0) to[out=-90,in=-90,looseness=1.2] (0.4,0);
            \draw (0,0.2) to[out=-135,in=90] (-0.1,0) to[out=-90,in=-90,looseness=1.2] (0.4,0);
        \end{tikzpicture}
        \overset{\cref{chess}}{\implies}
        \tau \mergemor
        =
        \beta\,
        \begin{tikzpicture}[anchorbase]
            \draw (-0.2,-0.5) to[out=45,in=down] (0.15,-0.2) to[out=up,in=-45] (0,0) -- (0,0.2);
            \draw[wipe] (0.2,-0.5) to [out=135,in=down] (-0.15,-0.2) to[out=up,in=-135] (0,0);
            \draw (0.2,-0.5) to [out=135,in=down] (-0.15,-0.2) to[out=up,in=-135] (0,0);
        \end{tikzpicture}
        \overset{\cref{turvy}}{=}
        \beta \tau^{-1}\, \mergemor
        \implies
        (\beta - \tau^2) \mergemor = 0.
    \]
    If $\beta \ne \tau^2$, this implies that $\mergemor = 0$.  But then the third relation in \cref{chess} implies that $1_\go = 0$, and hence $\Tcat$ is trivial.
\end{proof}

We now specialize to $\kk = \Q(q)$ and
\begin{gather} \label{carpool}
    \tau = q^{12},\qquad
    \beta = q^{24},
    \\ \label{cherry}
    \delta = \frac{[3][8][13][18]}{[4][6][9]},
    \qquad \text{and} \qquad
    \phi = \frac{[2][7][12]}{[4][6]},
    \qquad \text{where} \quad
    [n] = \frac{q^n - q^{-n}}{q - q^{-1}}.
\end{gather}
Explicit computation using SageMath shows that $\delta, \phi \in \N[q,q^{-1}]$.  The significance of these particular choices will become apparent in \cref{sec:functor}.  In particular, the above choice of $\delta$ corresponds to the quantum dimension of the 26-dimensional simple module over the quantized enveloping algebra of type $F_4$.

Define
\begin{equation} \label{radioactive}
    \posdotcross
    := q^6\
    \begin{tikzpicture}[centerzero]
        \draw (0.4,-0.4) -- (-0.2,0.2);
        \draw[wipe] (-0.2,-0.2) -- (0.4,0.4);
        \draw (-0.2,-0.2) -- (0.4,0.4);
        \draw (-0.2,-0.4) -- (-0.2,0.4);
    \end{tikzpicture}
    \qquad \text{and} \qquad
    \negdotcross
    := q^{-6}\
    \begin{tikzpicture}[centerzero]
        \draw (-0.2,-0.2) -- (0.4,0.4);
        \draw[wipe] (0.4,-0.4) -- (-0.2,0.2);
        \draw (0.4,-0.4) -- (-0.2,0.2);
        \draw (-0.2,-0.4) -- (-0.2,0.4);
    \end{tikzpicture}
    \ .
\end{equation}

\begin{lem}
    We have
    \begin{equation} \label{tangly}
        \begin{tikzpicture}[anchorbase]
            \draw (0.2,0.2) -- (-0.2,-0.2);
            \draw (0.4,0.2) to[out=down,in=-45,looseness=2] (0,0) to[out=135,in=up,looseness=2] (-0.4,-0.2);
            \posdot{0,0};
        \end{tikzpicture}
        = \negdotcross
        =
        \begin{tikzpicture}[anchorbase]
            \draw (-0.2,0.2) -- (0.2,-0.2);
            \draw (-0.4,0.2) to[out=down,in=225,looseness=2] (0,0) to[out=45,in=up,looseness=2] (0.4,-0.2);
            \posdot{0,0};
        \end{tikzpicture}
        \ ,\quad
        \begin{tikzpicture}[anchorbase]
            \draw (0.2,0.2) -- (-0.2,-0.2);
            \draw (0.4,0.2) to[out=down,in=-45,looseness=2] (0,0) to[out=135,in=up,looseness=2] (-0.4,-0.2);
            \negdot{0,0};
        \end{tikzpicture}
        = \posdotcross
        =
        \begin{tikzpicture}[anchorbase]
            \draw (-0.2,0.2) -- (0.2,-0.2);
            \draw (-0.4,0.2) to[out=down,in=225,looseness=2] (0,0) to[out=45,in=up,looseness=2] (0.4,-0.2);
            \negdot{0,0};
        \end{tikzpicture}
        \ .
    \end{equation}
\end{lem}

\begin{proof}
    Using the third and fourth equalities in \cref{vortex} and the first two equalities in \cref{topsy} to rotate trivalent vertices, we have
    \[
        \begin{tikzpicture}[anchorbase]
            \draw (-0.2,0.2) -- (0.2,-0.2);
            \draw (-0.4,0.2) to[out=down,in=225,looseness=2] (0,0) to[out=45,in=up,looseness=2] (0.4,-0.2);
            \posdot{0,0};
        \end{tikzpicture}
        = q^6\
        \begin{tikzpicture}[anchorbase]
            \draw (-0.2,0.4) -- (-0.2,0.2) -- (0,0) -- (0.2,0.2) -- (0.2,0.4);
            \draw (-0.2,0.2) -- (-0.4,0);
            \draw (0,0) -- (-0.4,-0.4);
            \draw[wipe] (-0.3,-0.1) -- (0,-0.4);
            \draw (-0.4,0) -- (0,-0.4);
        \end{tikzpicture}
        \ \overset{\cref{pokey}}{\underset{\cref{turvy}}{=}}
        q^{-6}\
        \begin{tikzpicture}[centerzero]
            \draw (-0.2,-0.2) -- (0.4,0.4);
            \draw[wipe] (0.4,-0.4) -- (-0.2,0.2);
            \draw (0.4,-0.4) -- (-0.2,0.2);
            \draw (-0.2,-0.4) -- (-0.2,0.4);
        \end{tikzpicture}
        =
        \negdotcross\ ,
    \]
    completing the verification of the second equality in \cref{tangly}.  The proofs of the remaining equalities are analogous.
\end{proof}

\begin{cor}
    We have
    \begin{equation} \label{nuclear}
        \posdotcross
        = q^6
        \begin{tikzpicture}[centerzero]
            \draw (0.2,-0.2) -- (-0.4,0.4);
            \draw[wipe] (-0.4,-0.4) -- (0.2,0.2);
            \draw (-0.4,-0.4) -- (0.2,0.2);
            \draw (0.2,-0.4) -- (0.2,0.4);
        \end{tikzpicture}
        = q^{-6}\
        \begin{tikzpicture}[centerzero]
            \draw (0.2,-0.2) -- (-0.2,0);
            \draw[wipe] (-0.2,-0.2) -- (0.2,0);
            \draw (-0.2,-0.2) -- (0.2,0);
            \draw (-0.2,0.2) -- (-0.2,0) -- (0.2,0) -- (0.2,0.2);
        \end{tikzpicture}
        = q^{-6}\
        \begin{tikzpicture}[centerzero]
            \draw (-0.2,0.2) -- (0.2,0);
            \draw[wipe] (0.2,0.2) -- (-0.2,0);
            \draw (0.2,0.2) -- (-0.2,0);
            \draw (-0.2,-0.2) -- (-0.2,0) -- (0.2,0) -- (0.2,-0.2);
        \end{tikzpicture}
        \ ,\qquad
        \negdotcross
        = q^{-6}
        \begin{tikzpicture}[centerzero]
            \draw (-0.4,-0.4) -- (0.2,0.2);
            \draw[wipe] (0.2,-0.2) -- (-0.4,0.4);
            \draw (0.2,-0.2) -- (-0.4,0.4);
            \draw (0.2,-0.4) -- (0.2,0.4);
        \end{tikzpicture}
        = q^6\
        \begin{tikzpicture}[centerzero]
            \draw (-0.2,-0.2) -- (0.2,0);
            \draw[wipe] (0.2,-0.2) -- (-0.2,0);
            \draw (0.2,-0.2) -- (-0.2,0);
            \draw (-0.2,0.2) -- (-0.2,0) -- (0.2,0) -- (0.2,0.2);
        \end{tikzpicture}
        = q^6\
        \begin{tikzpicture}[centerzero]
            \draw (0.2,0.2) -- (-0.2,0);
            \draw[wipe] (-0.2,0.2) -- (0.2,0);
            \draw (-0.2,0.2) -- (0.2,0);
            \draw (-0.2,-0.2) -- (-0.2,0) -- (0.2,0) -- (0.2,-0.2);
        \end{tikzpicture}
        \ .
    \end{equation}
\end{cor}

\begin{proof}
    This follows immediately from \cref{tangly}.
\end{proof}

\begin{defin} \label{Fdef}
    Let $\Fcat_q$ be the $\Q(q)$-linear strict monoidal category generated by the object $\go$ and generating morphisms
    \begin{equation}
        \mergemor \colon \go \otimes \go \to \go,\quad
        \poscross,\ \negcross \colon \go \otimes \go \to \go \otimes \go,\quad
        \cupmor \colon \one \to \go \otimes \go,\quad
        \capmor \colon \go \otimes \go \to \one,
    \end{equation}
    subject to the relations \cref{vortex,venom} and the following additional relations:
    \begin{gather} \label{checkers}
        \begin{tikzpicture}[anchorbase]
            \draw (-0.15,0) to[out=down,in=135] (0.15,-0.4);
            \draw[wipe] (-0.15,-0.4) to[out=45,in=down] (0.15,0) arc(0:180:0.15);
            \draw (-0.15,-0.4) to[out=45,in=down] (0.15,0) arc(0:180:0.15);
        \end{tikzpicture}
        = q^{24}\ \capmor
        \ ,\quad
        \begin{tikzpicture}[anchorbase]
            \draw (0.2,-0.5) to [out=135,in=down] (-0.15,-0.2) to[out=up,in=-135] (0,0);
            \draw[wipe] (-0.2,-0.5) to[out=45,in=down] (0.15,-0.2);
            \draw (-0.2,-0.5) to[out=45,in=down] (0.15,-0.2) to[out=up,in=-45] (0,0) -- (0,0.2);
        \end{tikzpicture}
        = q^{12} \, \mergemor
        \ ,\quad
        \begin{tikzpicture}[centerzero]
            \draw  (0,-0.4) -- (0,-0.2) to[out=45,in=down] (0.15,0) to[out=up,in=-45] (0,0.2) -- (0,0.4);
            \draw (0,-0.2) to[out=135,in=down] (-0.15,0) to[out=up,in=-135] (0,0.2);
        \end{tikzpicture}
        = \frac{[2][7][12]}{[4][6]}\
        \begin{tikzpicture}[centerzero]
            \draw(0,-0.4) -- (0,0.4);
        \end{tikzpicture}
        \ ,\quad
        \bubble = \frac{[3][8][13][18]}{[4][6][9]} 1_\one,
        \quad
        \lollydrop = 0,
        \\ \label{symskein}
        \poscross - \negcross
        = \frac{q^{-4}-q^4}{[3]} \left( \jail - \hourglass + \Hmor - \Imor \right),
        \\ \label{qnuke}
        \posdotcross
        := q^6\
        \begin{tikzpicture}[centerzero]
            \draw (0.4,-0.4) -- (-0.2,0.2);
            \draw[wipe] (-0.2,-0.2) -- (0.4,0.4);
            \draw (-0.2,-0.2) -- (0.4,0.4);
            \draw (-0.2,-0.4) -- (-0.2,0.4);
        \end{tikzpicture}
        = \frac{[2]}{[4]} \left( q^{-8} \, \jail + q^8\, \hourglass + \poscross \right) - q^{-2}\, \Hmor - q^2 \Imor\, ,
        \\ \label{qsquare}
        \begin{multlined}
            \sqmor
            = \frac{q^{10} + q^8 + q^6 + q^4 + 1}{q^4 (q^4+1)^2} \ \jail
            + \frac{q^2(q^{10} + q^6 + q^4 + q^2 + 1)}{(q^4 + 1)^2} \ \hourglass
            - \frac{[2][3][6]}{[4]^2}\ \poscross
            \\
            \qquad \qquad
            + \frac{q^{10} + 2q^6 + q^2 + 1}{q^8 + q^4}\ \Hmor
            + \frac{q^{10} + q^8 + 2q^4 + 1}{q^6 + q^2}\ \Imor,
        \end{multlined}
        \\ \label{qpent}
        \begin{multlined}
            \pentmor =
            - \left(
                \begin{tikzpicture}[anchorbase]
                    \draw (-0.2,0) -- (0,0.25) -- (0.2,0);
                    \draw (0,0.25) -- (0,0.4);
                    \draw (-0.2,-0.25) -- (-0.2,0) -- (-0.3,0.4);
                    \draw (0.2,-0.25) -- (0.2,0) -- (0.3,0.4);
                \end{tikzpicture}
                +
                \begin{tikzpicture}[anchorbase]
                    \draw (-0.3,0.3) -- (0,0) -- (0.3,0.3);
                    \draw (0,0.3) -- (-0.15,0.15);
                    \draw (0,0) -- (0,-0.15) -- (-0.15,-0.3);
                    \draw (0,-0.15) -- (0.15,-0.3);
                \end{tikzpicture}
                +
                \begin{tikzpicture}[anchorbase]
                    \draw (0.3,0.3) -- (0,0) -- (-0.3,0.3);
                    \draw (0,0.3) -- (0.15,0.15);
                    \draw (0,0) -- (0,-0.15) -- (0.15,-0.3);
                    \draw (0,-0.15) -- (-0.15,-0.3);
                \end{tikzpicture}
                +
                \begin{tikzpicture}[anchorbase]
                    \draw (-0.3,-0.3) -- (0.3,0.3);
                    \draw (-0.3,0.3) -- (-0.15,-0.15);
                    \draw (0,0.3) -- (0.15,0.15);
                    \draw (0,0) -- (0.3,-0.3);
                \end{tikzpicture}
                +
                \begin{tikzpicture}[anchorbase]
                    \draw (0.3,-0.3) -- (-0.3,0.3);
                    \draw (0.3,0.3) -- (0.15,-0.15);
                    \draw (0,0.3) -- (-0.15,0.15);
                    \draw (0,0) -- (-0.3,-0.3);
                \end{tikzpicture}
            \right)
            - \frac{[7]}{[4]^2}
            \left(
                \begin{tikzpicture}[centerzero]
                    \draw (-0.15,-0.3) -- (-0.15,-0.23) arc(180:0:0.15) -- (0.15,-0.3);
                    \draw (-0.3,0.3) -- (0,0.08) -- (0.3,0.3);
                    \draw (0,0.3) -- (0,0.08);
                \end{tikzpicture}
                +
                \begin{tikzpicture}[centerzero]
                    \draw (-0.2,-0.3) -- (-0.2,0.3);
                    \draw (0,0.3) -- (0.15,0) -- (0.3,0.3);
                    \draw (0.15,0) -- (0.15,-0.3);
                \end{tikzpicture}
                +
                \begin{tikzpicture}[centerzero]
                    \draw (0.2,-0.3) -- (0.2,0.3);
                    \draw (0,0.3) -- (-0.15,0) -- (-0.3,0.3);
                    \draw (-0.15,0) -- (-0.15,-0.3);
                \end{tikzpicture}
                +
                \begin{tikzpicture}[centerzero]
                    \draw (-0.3,0.3) -- (-0.3,0.23) arc(180:360:0.15) -- (0,0.3);
                    \draw (0.3,0.3) -- (0.15,0) -- (-0.2,-0.3);
                    \draw (0.2,-0.3) -- (0.15,0);
                \end{tikzpicture}
                +
                \begin{tikzpicture}[centerzero]
                    \draw (0.3,0.3) -- (0.3,0.23) arc(360:180:0.15) -- (0,0.3);
                    \draw (-0.3,0.3) -- (-0.15,0) -- (0.2,-0.3);
                    \draw (-0.2,-0.3) -- (-0.15,0);
                \end{tikzpicture}
            \right)
            \\
            + \frac{[3]^2}{[4]^2}
            \left(
                \begin{tikzpicture}[centerzero]
                    \draw (0,0.3) -- (0,-0.15) -- (-0.15,-0.3);
                    \draw (0,-0.15) -- (0.15,-0.3);
                    \draw[wipe] (-0.2,0.3) -- (-0.2,0.25) arc(180:360:0.2) -- (0.2,0.3);
                    \draw (-0.2,0.3) -- (-0.2,0.25) arc(180:360:0.2) -- (0.2,0.3);
                \end{tikzpicture}
                +
                \begin{tikzpicture}[centerzero]
                    \draw (-0.3,0.3) -- (0,0) -- (0.3,0.3);
                    \draw (0,0) -- (-0.15,-0.3);
                    \draw[wipe] (0,0.3) to[out=-45,in=70] (0.15,-0.3);
                    \draw (0,0.3) to[out=-45,in=70] (0.15,-0.3);
                \end{tikzpicture}
                +
                \begin{tikzpicture}[centerzero]
                    \draw (0.3,0.3) -- (0,0) -- (-0.3,0.3);
                    \draw (0,0) -- (0.15,-0.3);
                    \draw[wipe] (0,0.3) to[out=225,in=110] (-0.15,-0.3);
                    \draw (0,0.3) to[out=225,in=110] (-0.15,-0.3);
                \end{tikzpicture}
                +
                \begin{tikzpicture}[centerzero]
                    \draw (-0.3,0.3) -- (-0.15,0.15) -- (0,0.3);
                    \draw (-0.15,0.15) -- (0.15,-0.3);
                    \draw[wipe] (0.3,0.3) -- (-0.15,-0.3);
                    \draw (0.3,0.3) -- (-0.15,-0.3);
                \end{tikzpicture}
                +
                \begin{tikzpicture}[centerzero]
                    \draw (0.3,0.3) -- (0.15,0.15) -- (0,0.3);
                    \draw (0.15,0.15) -- (-0.15,-0.3);
                    \draw[wipe] (-0.3,0.3) -- (0.15,-0.3);
                    \draw (-0.3,0.3) -- (0.15,-0.3);
                \end{tikzpicture}
            \right).
        \end{multlined}
    \end{gather}
\end{defin}

Note that the relations \cref{checkers} are simply the relations \cref{chess} with the parameters given by \cref{carpool,cherry}.  Hence, $\Fcat_q$ is a quotient of $\Tcat$ with these parameters.  Note also that $\Fcat_q$ inherits from $\Tcat$ the property of being a ribbon category.

\begin{rem}
    All of the coefficients appearing in \cref{Fdef} are regular at $q=1$.  Setting $q=1$ in these relations recovers the category $\Fcat_{7,26}$ of \cite[Def.~5.1]{GSZ21}.  More precisely, \cref{symskein} gives $\crossmor := \poscross = \negcross$.  Then relations \cref{vortex,venom,checkers,qnuke,qsquare,qpent} become (2.2), (2.3), (2.4), (2.10), (2.11), and (2.12) of \cite{GSZ21}, respectively.
\end{rem}

\begin{lem}
    The relation
    \begin{equation} \label{F4triangle}
        \trimor
        = - \frac{[9]}{[3]}\ \mergemor
    \end{equation}
    holds in $\Fcat_q$.
\end{lem}

\begin{proof}
    We have
    \begin{gather*}
        \left( q^{12}-q^{-12} \right) \mergemor
        \overset{\cref{checkers}}{\underset{\cref{turvy}}{=}}
        \begin{tikzpicture}[anchorbase]
            \draw (0.2,-0.5) to [out=135,in=down] (-0.15,-0.2) to[out=up,in=-135] (0,0);
            \draw[wipe] (-0.2,-0.5) to[out=45,in=down] (0.15,-0.2);
            \draw (-0.2,-0.5) to[out=45,in=down] (0.15,-0.2) to[out=up,in=-45] (0,0) -- (0,0.2);
        \end{tikzpicture}
        -
        \begin{tikzpicture}[anchorbase]
            \draw (-0.2,-0.5) to[out=45,in=down] (0.15,-0.2) to[out=up,in=-45] (0,0) -- (0,0.2);
            \draw[wipe] (0.2,-0.5) to [out=135,in=down] (-0.15,-0.2) to[out=up,in=-135] (0,0);
            \draw (0.2,-0.5) to [out=135,in=down] (-0.15,-0.2) to[out=up,in=-135] (0,0);
        \end{tikzpicture}
        \overset{\cref{symskein}}{\underset{\cref{checkers}}{=}}
        \frac{q^{-4} - q^4}{[3]} \left( \mergemor + \trimor - \frac{[2][7][12]}{[4][6]} \mergemor \right)
        \\ \implies
        \trimor
        = \left( [3] \frac{q^{12}-q^{-12}}{q^{-4}-q^4} + \frac{[2][7][12]}{[4][6]} - 1 \right) \mergemor
        = - \frac{[9]}{[3]} \mergemor,
    \end{gather*}
    where the final equality is verified using SageMath.
\end{proof}

Consider the $\Q$-linear map
\begin{equation} \label{Qbar}
    \bar{\ } \colon \Q(q) \to \Q(q),\quad \overline{q^{\pm 1}} = q^{\mp 1}.
\end{equation}
We say that a functor $\bF \colon \cC \to \mathcal{D}$ between $\Q(q)$-linear categories is \emph{antilinear} if $\bF(cf) = \bar{c}\, \bF(f)$ for all morphisms $f$ in $\cC$ and all $c \in \Q(q)$.

\begin{lem} \label{disco}
    There is a unique antilinear monoidal functor $\Xi \colon \Fcat_q \to \Fcat_q$ satisfying $\bar{\go} = \go$ and defined on the generating morphisms by
    \begin{equation} \label{club}
        \Xi \left( \mergemor \right) = \mergemor,\quad
        \Xi \left( \poscross \right) = \negcross,\quad
        \Xi \left( \negcross \right) = \poscross,\quad
        \Xi \left( \cupmor \right) = \cupmor,\quad
        \Xi \left( \capmor \right) = \capmor.
    \end{equation}
    The functor $\Xi$ squares to the identity and also satisfies
    \begin{equation} \label{clubnuke}
        \Xi \left( \posdotcross \right) = \negdotcross,
        \qquad \qquad
        \Xi \left( \negdotcross \right) = \posdotcross.
    \end{equation}
\end{lem}

\begin{proof}
    It is straightforward to verify that $\Xi$ respects the defining relations and hence is well defined.  It clearly squares to the identity.  Finally, \cref{clubnuke} follows immediately from \cref{radioactive}.
\end{proof}

Intuitively, the involution $\Xi$ is given by flipping all crossings and taking $q$ to $q^{-1}$.

\section{The 2-point endomorphism algebra\label{sec:2point}}

Throughout this section we work over the ground ring $\kk = \Q(q)$.  We assume the parameters $\tau$ and $\beta$ are given by \cref{carpool}, but leave the other parameters $\delta \in \kk$ and $\phi \in \kk^\times$ general.  We work in the quotient $\Tcat/\cI$ by a tensor ideal satisfying the following assumption.

\begin{assumption} \label{expanse}
    Assume that
    \begin{equation} \label{lager}
        \delta \notin \{ q^4 + q^{-4}, - q^{16} - q^{-16} \}
    \end{equation}
    and that $\cI$ is a tensor ideal of $\Tcat$ such that
    \begin{equation} \label{IPA}
        \jail\ ,\quad \hourglass\ ,\quad \Hmor\ ,\quad \Imor\ ,\quad \poscross + \negcross
        \quad \text{form a basis for } (\Tcat/\cI)(\go^{\otimes 2}, \go^{\otimes 2}).
    \end{equation}
\end{assumption}

The motivation for this assumption is that the images of the morphisms \cref{IPA} under the functor to be defined in \cref{sec:functor} are a basis in the target category.  Hence this functor factors through a quotient of $\Tcat$ by an ideal $\cI$ satisfying \cref{expanse}.  Our goal in this section is to deduce several relations in the \emph{2-point endomorphism algebra} $(\Tcat/\cI)(\go^{\otimes 2}, \go^{\otimes 2})$ that follow from this assumption.  These will be used to show that our functor actually factors through $\Fcat_q$.  Throughout this section we continue to use the same diagrams to denote their image in $\Tcat/\cI$.

For $m,n \ge 1$, consider the linear operator
\begin{equation} \label{Rot}
    \Rot \colon \Tcat(\go^{\otimes m}, \go^{\otimes n}) \to \Tcat(\go^{\otimes m}, \go^{\otimes n}),
    \quad
    \Rot
    \left(
        \begin{tikzpicture}[anchorbase]
            \draw[line width=2] (-0.1,-0.4) -- (-0.1,0);
            \draw (0.1,-0.4) -- (0.1,0);
            \draw (-0.1,0) -- (-0.1,0.4);
            \draw[line width=2] (0.1,0) -- (0.1,0.4);
            \filldraw[fill=white,draw=black] (-0.25,0.2) rectangle (0.25,-0.2);
            \node at (0,0) {$\scriptstyle{f}$};
        \end{tikzpicture}
    \right)
    =
    \begin{tikzpicture}[anchorbase]
        \draw (-0.25,0.2) rectangle (0.25,-0.2);
        \node at (0,0) {$\scriptstyle{f}$};
        \draw (-0.4,-0.4) -- (-0.4,0.2) arc (180:0:0.15);
        \draw (0.4,0.4) -- (0.4,-0.2) arc(360:180:0.15);
        \draw[line width=2] (-0.1,-0.4) -- (-0.1,-0.2);
        \draw[line width=2] (0.1,0.4) -- (0.1,0.2);
    \end{tikzpicture}
    \ ,
\end{equation}
where the bottom and top thick strands represent $1_\go^{\otimes (m-1)}$ and $1_\go^{\otimes (n-1)}$, respectively.  In other words,
\[
    \Rot(f) = \left( \capmor \otimes 1_\go^{\otimes n} \right) \circ (1_\go \otimes f \otimes 1_\go) \circ (1_\go^{\otimes m} \otimes \cupmor).
\]
It follows easily from \cref{windy,IPA} that $\Rot^2$ acts as the identity on $(\Tcat/\cI)(\go^{\otimes 2}, \go^{\otimes 2})$.  We will often decompose $(\Tcat/\cI)(\go^{\otimes 2}, \go^{\otimes 2})$ into the $\pm 1$-eigenspaces for $\Rot$.

\begin{prop} \label{skeinexist}
    There exist $z_1,z_2 \in \kk$ such that
    \begin{equation} \label{skein}
        \poscross - \negcross
        = z_1 \left( \jail - \hourglass \right) + z_2 \left( \Hmor - \Imor \right).
    \end{equation}
\end{prop}

\begin{proof}
    It follows from \cref{IPA} that
    \[
        \poscross - \negcross
        = z_1 \left( \jail - \hourglass \right) + z_2 \left( \Hmor - \Imor \right) + z_3 \left( \jail + \hourglass \right) + z_4 \left( \Hmor +\Imor \right) + z_5 \left( \poscross + \negcross \right)
    \]
    for some $z_1,z_2,z_3,z_4,z_5 \in \kk$.  Applying the operator $\Rot$ of \cref{Rot} and multiplying by $-1$, we have
    \[
        \poscross - \negcross
        = z_1 \left( \jail - \hourglass \right) + z_2 \left( \Hmor - \Imor \right) - z_3 \left( \jail + \hourglass \right) - z_4 \left( \Hmor +\Imor \right) - z_5 \left( \poscross + \negcross \right).
    \]
    It then follows from \cref{IPA} that $z_3=z_4=z_5=0$.
\end{proof}

\begin{rem} \label{hog}
    We refer to relation \cref{skein} as the \emph{skein relation}.  As long as $z_1,z_2 \in \kk^\times$, we can rescale $\capmor$ by $\frac{z_2}{z_1}$ and $\cupmor$ by $\frac{z_1}{z_2}$.  This rescales $\splitmor$ by $\frac{z_1}{z_2}$, and so we obtain a skein relation of the form
    \begin{equation} \label{skeiner}
        \poscross - \negcross
        = z \left( \jail - \hourglass + \Hmor - \Imor \right).
    \end{equation}
    At some point we will perform this rescaling; see \cref{waxing}.  However, working with the more general relation \cref{skein} gives us some additional flexibility.
\end{rem}

\begin{cor} \label{stout}
    The morphisms
    \begin{equation} \label{slamdunk}
        \jail\ ,\quad \hourglass\ ,\quad \Hmor\ ,\quad \Imor\ ,\quad \poscross
    \end{equation}
    form a basis for $(\Tcat/\cI)(\go^{\otimes 2}, \go^{\otimes 2})$.
\end{cor}

\begin{proof}
    This follows immediately from \cref{IPA,skein}.
\end{proof}

\begin{cor}
    We have
    \[
        \begin{tikzpicture}[centerzero]
            \draw (0.2,-0.4) to[out=135,in=down] (-0.15,0);
            \draw[wipe] (-0.2,-0.4) to[out=45,in=down] (0.15,0);
            \draw (-0.2,-0.4) to[out=45,in=down] (0.15,0) to[out=up,in=-45] (-0.2,0.4);
            \draw[wipe] (-0.15,0) to[out=up,in=-135] (0.2,0.4);
            \draw (-0.15,0) to[out=up,in=-135] (0.2,0.4);
        \end{tikzpicture}
        =
        \jail + z_1 \left( \poscross - q^{24}\ \hourglass\, \right) + z_2 \left( q^6\ \posdotcross - q^{12}\ \Imor \right).
    \]
\end{cor}

\begin{proof}
    This follows from adding $\poscross$ to the top of \cref{skein}, then using \cref{chess,nuclear}.
\end{proof}

\begin{rem} \label{degenerate2}
    When $z_1=z_2=0$, relation \cref{skein} implies that $\poscross = \negcross$.  See \cref{degenerate}.  We refer to this as the \emph{degenerate} case.   (This is referred to as the \emph{classical} case in the introduction.)
\end{rem}

\begin{lem}
    We have
    \begin{equation}  \label{spoon}
        (1-\delta) z_1 + \phi z_2 = q^{24} - q^{-24}.
    \end{equation}
\end{lem}

\begin{proof}
    We have
    \[
        (q^{24} - q^{-24})\ \capmor
        \overset{\cref{chess}}{\underset{\cref{turvy}}{=}}
        \begin{tikzpicture}[anchorbase]
            \draw (-0.15,0) to[out=down,in=135] (0.15,-0.4);
            \draw[wipe] (-0.15,-0.4) to[out=45,in=down] (0.15,0) arc(0:180:0.15);
            \draw (-0.15,-0.4) to[out=45,in=down] (0.15,0) arc(0:180:0.15);
        \end{tikzpicture}
        -
        \begin{tikzpicture}[anchorbase]
            \draw (-0.15,-0.4) to[out=45,in=down] (0.15,0) arc(0:180:0.15);
            \draw[wipe] (-0.15,0) to[out=down,in=135] (0.15,-0.4);
            \draw (-0.15,0) to[out=down,in=135] (0.15,-0.4);
        \end{tikzpicture}
        \overset{\cref{skein}}{\underset{\cref{chess}}{=}}
        \left( (1-\delta) z_1 + \phi z_2 \right)\ \capmor\, .
    \]
    If $(1-\delta) z_1 + \phi z_2 \ne q^{24}-q^{-24}$, then $\capmor = 0$.  Then, by the first equality in \cref{vortex}, we have $1_\go = 0$, and so the category is trivial, contradicting \cref{IPA}.
\end{proof}

\begin{lem}
    We have
    \begin{equation} \label{triangle}
        z_2 \trimor
        = \left( q^{12} - q^{-12} + \phi z_2 - z_1 \right) \mergemor.
    \end{equation}
\end{lem}

\begin{proof}
    We have
    \[
        \left( q^{12}-q^{-12} \right) \mergemor
        \overset{\cref{chess}}{\underset{\cref{turvy}}{=}}
        \begin{tikzpicture}[anchorbase]
            \draw (0.2,-0.5) to [out=135,in=down] (-0.15,-0.2) to[out=up,in=-135] (0,0);
            \draw[wipe] (-0.2,-0.5) to[out=45,in=down] (0.15,-0.2);
            \draw (-0.2,-0.5) to[out=45,in=down] (0.15,-0.2) to[out=up,in=-45] (0,0) -- (0,0.2);
        \end{tikzpicture}
        -
        \begin{tikzpicture}[anchorbase]
            \draw (-0.2,-0.5) to[out=45,in=down] (0.15,-0.2) to[out=up,in=-45] (0,0) -- (0,0.2);
            \draw[wipe] (0.2,-0.5) to [out=135,in=down] (-0.15,-0.2) to[out=up,in=-135] (0,0);
            \draw (0.2,-0.5) to [out=135,in=down] (-0.15,-0.2) to[out=up,in=-135] (0,0);
        \end{tikzpicture}
        \overset{\cref{skein}}{\underset{\cref{chess}}{=}} (z_1 - \phi z_2) \mergemor + z_2 \trimor.
    \]
    Solving for $z_2 \trimor$ then yields \cref{triangle}.
\end{proof}

\begin{rem}
    Note that, when $z_2 \in \kk^\times$, \cref{triangle} allows us to write $\trimor$ as a multiple of the trivalent vertex $\mergemor$.  In the degenerate case (see \cref{degenerate,degenerate2}), one does not deduce such a relation until further relations are imposed; see \cite[Lem.~3.4]{GSZ21}.  Later, when we specialize the parameters $\delta$ and $\phi$, \cref{triangle} will reduce to \cref{F4triangle}; see \cref{bagel}.
\end{rem}

\begin{lem}
    We have
    \begin{equation} \label{pirate}
        \begin{aligned}
            (q^6+q^{-6}) \left( \posdotcross - \negdotcross \right)
            &= \phi z_1 \left(\, \jail - \hourglass \right) + \left( q^{12} - q^{-12} + \phi z_2 \right) \left( \Hmor - \Imor \right)
            \\
            &= \phi \left( \poscross - \negcross \right) + (q^{12}-q^{-12}) \left( \Hmor - \Imor \right)
        \end{aligned}
    \end{equation}
    and
    \begin{equation} \label{crate}
        2 z_2 \sqmor
        = (q^6-q^{-6}) \left( \posdotcross + \negdotcross \right)
        + \phi z_1 \left( \jail + \hourglass \right)
        + \big( q^{12} - q^{-12} + \phi z_2 - 2 z_1 \big) \left( \Hmor + \Imor \right).
    \end{equation}
\end{lem}

\begin{proof}
    Adding $\Hmor$ to the top of \cref{skein}, and using \cref{nuclear,chess,triangle}, gives
    \begin{equation} \label{pie1}
        q^6\, \posdotcross - q^{-6}\, \negdotcross
        = z_1 \, \Hmor - \phi z_1\ \hourglass\, + z_2\, \sqmor - \left(q^{12} - q^{-12} + \phi z_2 - z_1 \right) \Imor.
    \end{equation}
    Applying the rotation operator $\Rot$ then gives
    \begin{equation} \label{pie2}
        q^6\, \negdotcross - q^{-6}\, \posdotcross
        = z_1\, \Imor - \phi z_1\ \jail\, + z_2\, \sqmor - \left(q^{12} - q^{-12} + \phi z_2 - z_1 \right) \Hmor.
    \end{equation}
    Subtracting \cref{pie2} from \cref{pie1} gives
    \begin{equation} \label{eyepatch}
        \left( q^6 + q^{-6} \right) \left( \posdotcross - \negdotcross \right)
        = \phi z_1 \left(\, \jail - \hourglass \right) + \left( q^{12} - q^{-12} + \phi z_2 \right) \left( \Hmor - \Imor \right),
    \end{equation}
    proving the first equality in \cref{pirate}.  The second equality in \cref{pirate} follows from \cref{skein}.

    On the other hand, adding \cref{pie1,pie2} gives
    \[
        \left( q^6 - q^{-6} \right) \left( \posdotcross + \negdotcross \right)
        = - \phi z_1 \left(\, \jail + \hourglass \right) + 2z_2\, \sqmor - \left( q^{12} - q^{-12} + \phi z_2 - 2 z_1 \right) \left( \Hmor + \Imor \right).
    \]
    Solving for the square gives \cref{crate}.
\end{proof}

\begin{prop} \label{cow}
    We have
    \begin{align} \label{cow+}
        \posdotcross &= \frac{\phi(q^{10}+q^{-10})}{\delta + q^{16} + q^{-16}} \left( q^{-8} \, \jail + q^8\, \hourglass + \poscross \right) - q^{-2}\, \Hmor- q^2 \Imor\, ,
        \\ \label{cow-}
        \negdotcross &= \frac{\phi(q^{10}+q^{-10})}{\delta + q^{16} + q^{-16}} \left(q^8\, \jail + q^{-8}\, \hourglass + \negcross \right) - q^2\, \Hmor - q^{-2} \Imor\, .
    \end{align}
\end{prop}

\begin{proof}
    First note that \cref{cow-} follows from \cref{cow+} by applying either $\Rot$ or the involution $\Xi$ of \cref{disco}.  Thus, it suffices to prove \cref{cow+}.

    By \cref{stout}, we have a relation of the form
    \begin{equation} \label{cowseye1}
        \posdotcross = a\ \jail\, + b\, \hourglass + c\, \Hmor\, + d\, \Imor\, + e\, \poscross\, .
    \end{equation}
    Composing on the bottom with $\negcross$, and using \cref{turvy,nuclear}, gives
    \[
        q^{-6}\, \Hmor = a\, \negcross + b q^{-24}\ \hourglass + c q^{-6}\, \negdotcross\, + d q^{-12}\, \Imor\, + e\ \jail\, .
    \]
    Solving for $\negdotcross$ gives
    \[
        c\, \negdotcross
        =
        -eq^6\ \jail - bq^{-18}\ \hourglass + \Hmor - dq^{-6}\, \Imor\, - aq^6\, \negcross.
    \]
    Rotating then gives
    \[
        c\, \posdotcross
        =
        -bq^{-18}\ \jail - eq^6\ \hourglass - dq^{-6}\, \Hmor+ \Imor - aq^6\, \poscross.
    \]
    Comparing to \cref{cowseye1} and using \cref{stout}, we get
    \[
        ca = - bq^{-18},\quad
        cb = - eq^6,\quad
        c^2 = - dq^{-6},\quad
        cd = 1,\quad
        ce = -aq^6.
    \]
    From the third and fourth equations we get
    \[
        d^{-2} = -dq^{-6}
        \implies d = -q^2
        \implies c = -q^{-2}.
    \]
    Then we have
    \begin{equation} \label{dart}
        b = -q^{18} ca = q^{16} a,\quad
        e = -q^6 c^{-1}a = q^8 a.
    \end{equation}
    Substituting into \cref{cowseye1} gives
    \[
        \posdotcross = a\ \jail\, + q^{16} a\, \hourglass - q^{-2}\, \Hmor\, - q^2\, \Imor\, + q^8 a\, \poscross\, .
    \]
    Adding $\capmor$ to the top and using \cref{nuclear,chess,turvy} gives
    \[
        q^{18} \phi\ \capmor =
        \left( a + q^{16} a \delta - q^{-2} \phi + q^{32} a \right) \capmor.
    \]
    Thus
    \[
        q^{18} \phi = a \left( 1 + q^{16} \delta + q^{32} \right) - q^{-2} \phi
        \implies a = \frac{\phi(q^{18}+q^{-2})}{q^{32} + q^{16} \delta + 1}
        = \frac{\phi(q^2 + q^{-18})}{\delta+q^{16}+q^{-16}}.
    \]
    The relation \cref{cow+} now follows from \cref{dart}.
\end{proof}

\begin{lem} \label{lemmaz}
    We have
    \begin{align}
        z_1
        &= \frac{(q^{10} +q^{-10}) (q^6 + q^{-6}) (q^8-q^{-8})}{q^4 + q^{-4} - \delta},
        \\
        z_2
        &= \frac{(\delta+q^{16}+q^{-16})(q^6+q^{-6})(q^4+q^{-4})(q^2-q^{-2})}{\phi (q^4+q^{-4}-\delta)}.
    \end{align}
\end{lem}

\begin{proof}
    From \cref{cow+,cow-}, we have
    \begin{align*}
        \posdotcross - \negdotcross
        &= \frac{\phi(q^{10} + q^{-10})}{\delta + q^{16} + q^{-16}} \left( (q^{-8}-q^8) \left( \, \jail - \, \hourglass\, \right) + \poscross - \negcross \right) + (q^2-q^{-2}) \left( \Hmor - \Imor \right)
        \\
        &\overset{\mathclap{\cref{skein}}}{=}\
        \frac{\phi(q^{10}+q^{-10})}{\delta+q^{16}+q^{-16}} (q^{-8}-q^8+z_1) \left( \, \jail - \, \hourglass\, \right) + \left( \frac{\phi z_2 (q^{10}+q^{-10})}{\delta+q^{16}+q^{-16}} + q^2 - q^{-2} \right) \left( \Hmor - \Imor \right).
    \end{align*}
    Comparing to \cref{pirate}, and using \cref{IPA}, we have
    \begin{align*}
        z_1
        &= \frac{(q^6+q^{-6})(q^{10}+q^{-10})}{\delta+q^{16}+q^{-16}} (q^{-8}-q^8+z_1)
        \\ \implies
        z_1(\delta + q^{16} + q^{-16})
        &= (q^{10} + q^{-10}) (q^6 + q^{-6}) (q^{-8}-q^8+z_1)
        \\ \implies
        z_1(\delta - q^4 - q^{-4})
        &= (q^{10}+q^{-10}) (q^6+q^{-6}) (q^{-8}-q^8)
        \\ \implies
        z_1
        &= \frac{(q^{10}+q^{-10}) (q^6+q^{-6}) (q^{-8}-q^8)}{\delta - q^4 - q^{-4}}
    \end{align*}
    and
    \begin{align*}
        q^{12} - q^{-12} + \phi z_2
        &= (q^6+q^{-6}) \left( \frac{\phi z_2 (q^{10}+q^{-10})}{\delta+q^{16}+q^{-16}} + q^2 - q^{-2} \right)
        \\ \implies
        q^{12} - q^{-12} - (q^6+q^{-6})(q^2 - q^{-2})
        &= \frac{\phi z_2 ((q^{10} + q^{-10})(q^6+q^{-6}) - \delta - q^{16} - q^{-16})}{\delta+q^{16}+q^{-16}}
        \\ \implies
        (q^6+q^{-6})(q^4+q^{-4})(q^2-q^{-2})
        &= \frac{\phi z_2 (q^4 + q^{-4} - \delta)}{\delta+q^{16}+q^{-16}}
        \\ \implies
        z_2
        &= \frac{(\delta+q^{16}+q^{-16})(q^6+q^{-6})(q^4+q^{-4})(q^2-q^{-2})}{\phi (q^4+q^{-4}-\delta)}.
        \qedhere
    \end{align*}
\end{proof}

\begin{cor} \label{waxing}
    We have $z_1=z_2$ if and only if
    \begin{equation} \label{moon}
        \phi
        = \frac{\delta + q^{16} + q^{-16}}{(q^{10}+q^{-10})(q^2+q^{-2})}.
    \end{equation}
\end{cor}

\begin{proof}
    By \cref{lemmaz}, we have $z_1=z_2$ if and only if
    \begin{align*}
        \frac{(q^{10}+q^{-10}) (q^6+q^{-6}) (q^8-q^{-8})}{q^4 + q^{-4} - \delta}
        &= \frac{(\delta+q^{16}+q^{-16})(q^6+q^{-6})(q^4+q^{-4})(q^2-q^{-2})}{\phi (q^4+q^{-4}-\delta)}
        \\ \iff
        \phi &= \frac{\delta+q^{16}+q^{-16}}{(q^{10}+q^{-10})(q^2+q^{-2})}.
        \qedhere
    \end{align*}
\end{proof}

\begin{lem} \label{bull}
    If $z_1=z_2$, then
    \begin{align} \label{bull+}
        \posdotcross &= \frac{[2]}{[4]} \left( q^{-8} \, \jail + q^8\, \hourglass + \poscross \right) - q^{-2}\, \Hmor - q^2 \Imor\, ,
        \\ \label{bull-}
        \negdotcross &= \frac{[2]}{[4]} \left( q^8\, \jail + q^{-8}\, \hourglass + \negcross \right) - q^2\, \Hmor - q^{-2} \Imor\, .
    \end{align}
\end{lem}

\begin{proof}
    This follows from \cref{cow,waxing}.
\end{proof}

Note that \cref{bull+} is precisely \cref{qnuke}, while \cref{bull-} is its image under the rotation operator $\Rot$ of \cref{Rot}.

\section{Specializing the quantum dimension\label{sec:specialize}}

Throughout this section, we work over the ground ring $\kk = \Q(q)$ and make the choices \cref{carpool,cherry} for the parameters $\phi, \beta, \tau, \delta$.  We also continue to make \cref{expanse} and work in the quotient category $\Tcat/\cI$.

\begin{lem} \label{zbalance}
    We have
    \[
        z := z_1 = z_2 = \frac{q^{-4}-q^{4}}{[3]}.
    \]
\end{lem}

\begin{proof}
    Direct computation using SageMath verifies that
    \[
        \frac{\delta+q^{16}+q^{-16}}{(q^{10}+q^{-10})(q^2+q^{-2})}
        = \frac{[2][7][12]}{[4][6]}
        = \phi.
    \]
    Thus $z_1 = z_2$ by \cref{waxing}.  Then \cref{spoon} implies that $(1-\delta+\phi)z = q^{24} - q^{-24}$.  Further computation using SageMath then verifies that
    \[
        z = \frac{q^{24}-q^{-24}}{1-\delta+\phi}
        = \frac{q^{-4}-q^4}{[3]}.
        \qedhere
    \]
\end{proof}

\begin{prop} \label{bagel}
    Relations \cref{F4triangle,qsquare} are satisfied.
\end{prop}

\begin{proof}
    Direct computation using SageMath shows that
    \[
        \frac{q^{12}-q^{-12}}{z} + \phi - 1
        = - \frac{[9]}{[3]}.
    \]
    Then \cref{F4triangle} follows from \cref{triangle}.  The relation \cref{qsquare} also follows from \cref{crate,bull+,bull-} using SageMath.
\end{proof}

\begin{lem} \label{boots}
    The morphisms
    \begin{equation} \label{purple}
        \begin{tikzpicture}[centerzero]
            \draw (0,-0.15) -- (0.15,-0.3);
            \draw (-0.2,0.3) -- (-0.2,0.25) arc(180:360:0.2) -- (0.2,0.3);
            \draw[wipe] (0,0.3) -- (0,-0.15);
            \draw (0,0.3) -- (0,-0.15) -- (-0.15,-0.3);
        \end{tikzpicture}
        \ ,\quad
        \begin{tikzpicture}[centerzero]
            \draw (0,0) -- (-0.15,-0.3);
            \draw (0,0.3) to[out=-45,in=70] (0.15,-0.3);
            \draw[wipe] (0,0) -- (0.3,0.3);
            \draw (-0.3,0.3) -- (0,0) -- (0.3,0.3);
        \end{tikzpicture}
        \ ,\quad
        \begin{tikzpicture}[centerzero]
            \draw (0,0) -- (0.15,-0.3);
            \draw (0,0.3) to[out=225,in=110] (-0.15,-0.3);
            \draw[wipe] (0,0) -- (-0.3,0.3);
            \draw (0.3,0.3) -- (0,0) -- (-0.3,0.3);
        \end{tikzpicture}
        \ ,\quad
        \begin{tikzpicture}[centerzero]
            \draw (-0.3,0.3) -- (-0.15,0.15) -- (0,0.3);
            \draw (0.3,0.3) -- (-0.15,-0.3);
            \draw[wipe] (-0.15,0.15) -- (0.15,-0.3);
            \draw (-0.15,0.15) -- (0.15,-0.3);
        \end{tikzpicture}
        \ ,\quad
        \begin{tikzpicture}[centerzero]
            \draw (0.3,0.3) -- (0.15,0.15) -- (0,0.3);
            \draw (-0.3,0.3) -- (0.15,-0.3);
            \draw[wipe] (0.15,0.15) -- (-0.15,-0.3);
            \draw (0.15,0.15) -- (-0.15,-0.3);
        \end{tikzpicture}
        \ ,\quad
        \begin{tikzpicture}[anchorbase]
            \draw (-0.2,0) -- (0,0.25) -- (0.2,0);
            \draw (0,0.25) -- (0,0.4);
            \draw (0.2,-0.25) -- (-0.2,0) -- (-0.3,0.4);
            \draw[wipe] (-0.2,-0.25) -- (0.2,0);
            \draw (-0.2,-0.25) -- (0.2,0);
            \draw (0.2,0) -- (0.3,0.4);
        \end{tikzpicture}
        \ , \quad \text{and} \quad
        \begin{tikzpicture}[anchorbase]
            \draw (-0.2,0) -- (0,0.25) -- (0.2,0);
            \draw (0,0.25) -- (0,0.4);
            \draw (-0.2,-0.25) -- (0.2,0) -- (0.3,0.4);
            \draw[wipe] (0.2,-0.25) -- (-0.2,0);
            \draw (0.2,-0.25) -- (-0.2,0);
            \draw (-0.2,0) -- (-0.3,0.4);
        \end{tikzpicture}
    \end{equation}
    lie in the span of the morphisms
    \begin{equation} \label{5pointers}
        \begin{tikzpicture}[anchorbase]
            \draw (-0.2,0) -- (0,0.25) -- (0.2,0);
            \draw (0,0.25) -- (0,0.4);
            \draw (-0.2,-0.25) -- (-0.2,0) -- (-0.3,0.4);
            \draw (0.2,-0.25) -- (0.2,0) -- (0.3,0.4);
        \end{tikzpicture}
        \, ,\
        \begin{tikzpicture}[anchorbase]
            \draw (-0.3,0.3) -- (0,0) -- (0.3,0.3);
            \draw (0,0.3) -- (-0.15,0.15);
            \draw (0,0) -- (0,-0.15) -- (-0.15,-0.3);
            \draw (0,-0.15) -- (0.15,-0.3);
        \end{tikzpicture}
        \, ,\
        \begin{tikzpicture}[anchorbase]
            \draw (0.3,0.3) -- (0,0) -- (-0.3,0.3);
            \draw (0,0.3) -- (0.15,0.15);
            \draw (0,0) -- (0,-0.15) -- (0.15,-0.3);
            \draw (0,-0.15) -- (-0.15,-0.3);
        \end{tikzpicture}
        \, ,\
        \begin{tikzpicture}[anchorbase]
            \draw (-0.3,-0.3) -- (0.3,0.3);
            \draw (-0.3,0.3) -- (-0.15,-0.15);
            \draw (0,0.3) -- (0.15,0.15);
            \draw (0,0) -- (0.3,-0.3);
        \end{tikzpicture}
        \, ,\
        \begin{tikzpicture}[anchorbase]
            \draw (0.3,-0.3) -- (-0.3,0.3);
            \draw (0.3,0.3) -- (0.15,-0.15);
            \draw (0,0.3) -- (-0.15,0.15);
            \draw (0,0) -- (-0.3,-0.3);
        \end{tikzpicture}
        \, ,\
        \begin{tikzpicture}[centerzero]
            \draw (-0.15,-0.3) -- (-0.15,-0.23) arc(180:0:0.15) -- (0.15,-0.3);
            \draw (-0.3,0.3) -- (0,0.08) -- (0.3,0.3);
            \draw (0,0.3) -- (0,0.08);
        \end{tikzpicture}
        \, ,\
        \begin{tikzpicture}[centerzero]
            \draw (-0.2,-0.3) -- (-0.2,0.3);
            \draw (0,0.3) -- (0.15,0) -- (0.3,0.3);
            \draw (0.15,0) -- (0.15,-0.3);
        \end{tikzpicture}
        \, ,\
        \begin{tikzpicture}[centerzero]
            \draw (0.2,-0.3) -- (0.2,0.3);
            \draw (0,0.3) -- (-0.15,0) -- (-0.3,0.3);
            \draw (-0.15,0) -- (-0.15,-0.3);
        \end{tikzpicture}
        \, ,\
        \begin{tikzpicture}[centerzero]
            \draw (-0.3,0.3) -- (-0.3,0.23) arc(180:360:0.15) -- (0,0.3);
            \draw (0.3,0.3) -- (0.15,0) -- (-0.2,-0.3);
            \draw (0.2,-0.3) -- (0.15,0);
        \end{tikzpicture}
        \, ,\
        \begin{tikzpicture}[centerzero]
            \draw (0.3,0.3) -- (0.3,0.23) arc(360:180:0.15) -- (0,0.3);
            \draw (-0.3,0.3) -- (-0.15,0) -- (0.2,-0.3);
            \draw (-0.2,-0.3) -- (-0.15,0);
        \end{tikzpicture}
        \, ,\
        \begin{tikzpicture}[centerzero]
            \draw (0,0.3) -- (0,-0.15) -- (-0.15,-0.3);
            \draw (0,-0.15) -- (0.15,-0.3);
            \draw[wipe] (-0.2,0.3) -- (-0.2,0.25) arc(180:360:0.2) -- (0.2,0.3);
            \draw (-0.2,0.3) -- (-0.2,0.25) arc(180:360:0.2) -- (0.2,0.3);
        \end{tikzpicture}
        \, ,\
        \begin{tikzpicture}[centerzero]
            \draw (-0.3,0.3) -- (0,0) -- (0.3,0.3);
            \draw (0,0) -- (-0.15,-0.3);
            \draw[wipe] (0,0.3) to[out=-45,in=70] (0.15,-0.3);
            \draw (0,0.3) to[out=-45,in=70] (0.15,-0.3);
        \end{tikzpicture}
        \, ,\
        \begin{tikzpicture}[centerzero]
            \draw (0.3,0.3) -- (0,0) -- (-0.3,0.3);
            \draw (0,0) -- (0.15,-0.3);
            \draw[wipe] (0,0.3) to[out=225,in=110] (-0.15,-0.3);
            \draw (0,0.3) to[out=225,in=110] (-0.15,-0.3);
        \end{tikzpicture}
        \, ,\
        \begin{tikzpicture}[centerzero]
            \draw (-0.3,0.3) -- (-0.15,0.15) -- (0,0.3);
            \draw (-0.15,0.15) -- (0.15,-0.3);
            \draw[wipe] (0.3,0.3) -- (-0.15,-0.3);
            \draw (0.3,0.3) -- (-0.15,-0.3);
        \end{tikzpicture}
        \, ,\
        \begin{tikzpicture}[centerzero]
            \draw (0.3,0.3) -- (0.15,0.15) -- (0,0.3);
            \draw (0.15,0.15) -- (-0.15,-0.3);
            \draw[wipe] (-0.3,0.3) -- (0.15,-0.3);
            \draw (-0.3,0.3) -- (0.15,-0.3);
        \end{tikzpicture}
        .
    \end{equation}
\end{lem}

\begin{proof}
    Let $X$ denote the span of the morphisms \cref{5pointers}.  The fact that the first five morphisms in \cref{purple} lie in $X$ follows immediately from \cref{skein}.  Next, note that
    \begin{multline*}
        \begin{tikzpicture}[anchorbase]
            \draw (-0.2,0) -- (0,0.25) -- (0.2,0);
            \draw (0,0.25) -- (0,0.4);
            \draw (0.2,-0.25) -- (-0.2,0) -- (-0.3,0.4);
            \draw[wipe] (-0.2,-0.25) -- (0.2,0);
            \draw (-0.2,-0.25) -- (0.2,0);
            \draw (0.2,0) -- (0.3,0.4);
        \end{tikzpicture}
        \overset{\cref{venom}}{=}
        \begin{tikzpicture}[anchorbase]
            \draw (-0.2,0) to[out=-45,in=180] (0.3,0.25) to[out=0,in=45] (0.2,-0.25);
            \draw[wipe] (0,0.25) -- (0.2,0) -- (0.3,0.4);
            \draw (0,0.25) -- (0.2,0) -- (0.3,0.4);
            \draw (0,0.25) -- (0,0.4);
            \draw (-0.2,0) -- (0,0.25);
            \draw (-0.2,0) -- (-0.3,0.4);
            \draw (-0.2,-0.25) -- (0.2,0);
        \end{tikzpicture}
        \overset{\cref{nuclear}}{=} q^6
        \begin{tikzpicture}[anchorbase]
            \draw (-0.2,0.2) -- (-0.4,0.4);
            \draw (-0.2,0.2) -- (0,0.4);
            \draw (-0.2,0.2) to[out=-45,in=up] (0.4,-0.4);
            \draw[wipe] (-0.1,-0.2) -- (0.4,0.4);
            \draw (-0.1,-0.2) -- (0.4,0.4);
            \draw (-0.29,0.11) to[out=225,in=135] (-0.1,-0.2) -- (-0.4,-0.4);
            \posdot{-0.2,0.2};
        \end{tikzpicture}
        \overset{\cref{bull+}}{=} \frac{q^6}{q^2+q^{-2}}
        \left(
            q^{-8}\,
            \begin{tikzpicture}[centerzero]
                \draw (0,0) -- (-0.15,-0.3);
                \draw (0,0.3) to[out=-45,in=70] (0.15,-0.3);
                \draw[wipe] (0,0) -- (0.3,0.3);
                \draw (-0.3,0.3) -- (0,0) -- (0.3,0.3);
            \end{tikzpicture}
            + q^8\,
            \begin{tikzpicture}[centerzero,scale=0.75]
                \draw (0,0.4) arc (360:180:0.2);
                \draw (0,0.1) to[out=0,in=110] (0.4,-0.4);
                \draw[wipe] (-0.1,-0.2) -- (0.4,0.4);
                \draw (-0.1,-0.2) -- (0.4,0.4);
                \draw (-0.1,-0.2) to[out=110,in=180] (0,0.1);
                \draw (-0.4,-0.4) -- (-0.1,-0.2);
            \end{tikzpicture}
            +
            \begin{tikzpicture}[centerzero]
                \draw (0.3,0.3) -- (0.15,0.15) -- (0,0.3);
                \draw (-0.3,0.3) -- (0.15,-0.3);
                \draw[wipe] (0.15,0.15) -- (-0.15,-0.3);
                \draw (0.15,0.15) -- (-0.15,-0.3);
            \end{tikzpicture}
        \right)
        - q^4\,
        \begin{tikzpicture}[anchorbase]
            \draw (-0.3,0.4) -- (-0.2,0.2) -- (-0.05,0.2) -- (0,0.4);
            \draw (-0.05,0.2) -- (0.3,-0.25);
            \draw[wipe] (-0.2,-0.1) to[out=20,in=-100] (0.3,0.4);
            \draw (-0.2,-0.1) to[out=20,in=-100] (0.3,0.4);
            \draw (-0.2,0.2) -- (-0.2,-0.1) -- (-0.3,-0.25);
        \end{tikzpicture}
        - q^8\,
        \begin{tikzpicture}[anchorbase,scale=0.75]
            \draw (-0.4,0.4) -- (-0.2,0.2) -- (0,0.4);
            \draw (-0.2,0.2) -- (-0.2,0) to[out=0,in=100] (0.4,-0.4);
            \draw[wipe] (-0.1,-0.2) -- (0.4,0.4);
            \draw (-0.1,-0.2) -- (0.4,0.4);
            \draw (-0.2,0) -- (-0.1,-0.2) -- (-0.4,-0.4);
        \end{tikzpicture}
        \\
        \overset{\cref{turvy}}{\underset{\cref{radioactive}}{=}}
        \frac{q^6}{q^2+q^{-2}}
        \left(
            q^{-8}\,
            \begin{tikzpicture}[centerzero]
                \draw (0,0) -- (-0.15,-0.3);
                \draw (0,0.3) to[out=-45,in=70] (0.15,-0.3);
                \draw[wipe] (0,0) -- (0.3,0.3);
                \draw (-0.3,0.3) -- (0,0) -- (0.3,0.3);
            \end{tikzpicture}
            + q^{-4}\,
            \begin{tikzpicture}[centerzero]
                \draw (-0.3,0.3) -- (-0.3,0.23) arc(180:360:0.15) -- (0,0.3);
                \draw (0.3,0.3) -- (0.15,0) -- (-0.2,-0.3);
                \draw (0.2,-0.3) -- (0.15,0);
            \end{tikzpicture}
            +
            \begin{tikzpicture}[centerzero]
                \draw (0.3,0.3) -- (0.15,0.15) -- (0,0.3);
                \draw (-0.3,0.3) -- (0.15,-0.3);
                \draw[wipe] (0.15,0.15) -- (-0.15,-0.3);
                \draw (0.15,0.15) -- (-0.15,-0.3);
            \end{tikzpicture}
        \right)
        - q^4\,
        \begin{tikzpicture}[anchorbase]
            \draw (-0.3,0.4) -- (-0.2,0.2) -- (-0.05,0.2) -- (0,0.4);
            \draw (-0.05,0.2) -- (0.3,-0.25);
            \draw[wipe] (-0.2,-0.1) to[out=20,in=-100] (0.3,0.4);
            \draw (-0.2,-0.1) to[out=20,in=-100] (0.3,0.4);
            \draw (-0.2,0.2) -- (-0.2,-0.1) -- (-0.3,-0.25);
        \end{tikzpicture}
        - q^2\,
        \begin{tikzpicture}[anchorbase]
            \draw (0.14,0.04) -- (0.3,0.3);
            \draw (0.14,-0.14) -- (0.3,-0.3);
            \draw (-0.04,-0.14) -- (-0.3,-0.3);
            \draw (-0.04,0.04) -- (-0.15,0.15) -- (-0.3,0.3);
            \draw (-0.15,0.15) -- (0,0.3);
            \posdot{0.05,-0.05};
        \end{tikzpicture}.
    \end{multline*}
    Using \cref{bull+}, we see that the last morphism appearing above is in $X$.  Thus we have
    \[
        \begin{tikzpicture}[anchorbase]
            \draw (-0.2,0) -- (0,0.25) -- (0.2,0);
            \draw (0,0.25) -- (0,0.4);
            \draw (0.2,-0.25) -- (-0.2,0) -- (-0.3,0.4);
            \draw[wipe] (-0.2,-0.25) -- (0.2,0);
            \draw (-0.2,-0.25) -- (0.2,0);
            \draw (0.2,0) -- (0.3,0.4);
        \end{tikzpicture}
        + q^4\,
        \begin{tikzpicture}[anchorbase]
            \draw (-0.3,0.4) -- (-0.2,0.2) -- (-0.05,0.2) -- (0,0.4);
            \draw (-0.05,0.2) -- (0.3,-0.25);
            \draw[wipe] (-0.2,-0.1) to[out=20,in=-100] (0.3,0.4);
            \draw (-0.2,-0.1) to[out=20,in=-100] (0.3,0.4);
            \draw (-0.2,0.2) -- (-0.2,-0.1) -- (-0.3,-0.25);
        \end{tikzpicture}
        \in X.
    \]
    Note that the rightmost diagram above is $\Xi \Rot$ applied to the leftmost diagram.  Since $X$ is invariant under the action of $\Rot$ and $\Xi$ (using the skein relation \cref{skein}), we have
    \[
        \begin{tikzpicture}[anchorbase]
            \draw (-0.2,0) -- (0,0.25) -- (0.2,0);
            \draw (0,0.25) -- (0,0.4);
            \draw (0.2,-0.25) -- (-0.2,0) -- (-0.3,0.4);
            \draw[wipe] (-0.2,-0.25) -- (0.2,0);
            \draw (-0.2,-0.25) -- (0.2,0);
            \draw (0.2,0) -- (0.3,0.4);
        \end{tikzpicture}
        + q^{20}\,
        \begin{tikzpicture}[anchorbase]
            \draw (-0.2,0) -- (0,0.25) -- (0.2,0);
            \draw (0,0.25) -- (0,0.4);
            \draw (-0.2,-0.25) -- (0.2,0) -- (0.3,0.4);
            \draw[wipe] (0.2,-0.25) -- (-0.2,0);
            \draw (0.2,-0.25) -- (-0.2,0);
            \draw (-0.2,0) -- (-0.3,0.4);
        \end{tikzpicture}
        = \left( 1 - q^4 \Xi \Rot + q^8 \Xi^2 \Rot^2 - q^{12} \Xi^3 \Rot^3 + q^{16} \Xi^4 \Rot^4 \right)
        \left(
            \begin{tikzpicture}[anchorbase]
                \draw (-0.2,0) -- (0,0.25) -- (0.2,0);
                \draw (0,0.25) -- (0,0.4);
                \draw (0.2,-0.25) -- (-0.2,0) -- (-0.3,0.4);
                \draw[wipe] (-0.2,-0.25) -- (0.2,0);
                \draw (-0.2,-0.25) -- (0.2,0);
                \draw (0.2,0) -- (0.3,0.4);
            \end{tikzpicture}
            + q^4\,
            \begin{tikzpicture}[anchorbase]
                \draw (-0.3,0.4) -- (-0.2,0.2) -- (-0.05,0.2) -- (0,0.4);
                \draw (-0.05,0.2) -- (0.3,-0.25);
                \draw[wipe] (-0.2,-0.1) to[out=20,in=-100] (0.3,0.4);
                \draw (-0.2,-0.1) to[out=20,in=-100] (0.3,0.4);
                \draw (-0.2,0.2) -- (-0.2,-0.1) -- (-0.3,-0.25);
            \end{tikzpicture}
        \right)
        \in X.
    \]
    Applying $\Xi$ shows that
    \[
        \begin{tikzpicture}[anchorbase]
            \draw (-0.2,0) -- (0,0.25) -- (0.2,0);
            \draw (0,0.25) -- (0,0.4);
            \draw (-0.2,-0.25) -- (0.2,0) -- (0.3,0.4);
            \draw[wipe] (0.2,-0.25) -- (-0.2,0);
            \draw (0.2,-0.25) -- (-0.2,0);
            \draw (-0.2,0) -- (-0.3,0.4);
        \end{tikzpicture}
        +
        q^{-20}\,
        \begin{tikzpicture}[anchorbase]
            \draw (-0.2,0) -- (0,0.25) -- (0.2,0);
            \draw (0,0.25) -- (0,0.4);
            \draw (0.2,-0.25) -- (-0.2,0) -- (-0.3,0.4);
            \draw[wipe] (-0.2,-0.25) -- (0.2,0);
            \draw (-0.2,-0.25) -- (0.2,0);
            \draw (0.2,0) -- (0.3,0.4);
        \end{tikzpicture}
        \in X.
    \]
    Thus the last two morphisms in \cref{purple} lie in $X$.
\end{proof}

The following proposition, whose hypotheses may seem ad hoc at first, will be used in the proof of \cref{lightning}; see \cref{window}.

\begin{prop} \label{protomolecule}
    Suppose that
    \begin{equation} \label{5dim15}
        \dim \left( (\Tcat/\cI)(\go^{\otimes 3}, \go^{\otimes 2}) \right) \le 15,
    \end{equation}
    and that the last 10 morphisms in \cref{5pointers} are linearly independent.  Then relation \cref{qpent} holds in $\Tcat/\cI$.
\end{prop}

\begin{proof}
    We first show that we have a relation of the form
    \begin{equation} \label{robot}
        \begin{multlined}
            \pentmor =
            \gamma_1
            \left(
                \begin{tikzpicture}[anchorbase]
                    \draw (-0.2,0) -- (0,0.25) -- (0.2,0);
                    \draw (0,0.25) -- (0,0.4);
                    \draw (-0.2,-0.25) -- (-0.2,0) -- (-0.3,0.4);
                    \draw (0.2,-0.25) -- (0.2,0) -- (0.3,0.4);
                \end{tikzpicture}
                +
                \begin{tikzpicture}[anchorbase]
                    \draw (-0.3,0.3) -- (0,0) -- (0.3,0.3);
                    \draw (0,0.3) -- (-0.15,0.15);
                    \draw (0,0) -- (0,-0.15) -- (-0.15,-0.3);
                    \draw (0,-0.15) -- (0.15,-0.3);
                \end{tikzpicture}
                +
                \begin{tikzpicture}[anchorbase]
                    \draw (0.3,0.3) -- (0,0) -- (-0.3,0.3);
                    \draw (0,0.3) -- (0.15,0.15);
                    \draw (0,0) -- (0,-0.15) -- (0.15,-0.3);
                    \draw (0,-0.15) -- (-0.15,-0.3);
                \end{tikzpicture}
                +
                \begin{tikzpicture}[anchorbase]
                    \draw (-0.3,-0.3) -- (0.3,0.3);
                    \draw (-0.3,0.3) -- (-0.15,-0.15);
                    \draw (0,0.3) -- (0.15,0.15);
                    \draw (0,0) -- (0.3,-0.3);
                \end{tikzpicture}
                +
                \begin{tikzpicture}[anchorbase]
                    \draw (0.3,-0.3) -- (-0.3,0.3);
                    \draw (0.3,0.3) -- (0.15,-0.15);
                    \draw (0,0.3) -- (-0.15,0.15);
                    \draw (0,0) -- (-0.3,-0.3);
                \end{tikzpicture}
            \right)
            + \gamma_2
            \left(
                \begin{tikzpicture}[centerzero]
                    \draw (-0.15,-0.3) -- (-0.15,-0.23) arc(180:0:0.15) -- (0.15,-0.3);
                    \draw (-0.3,0.3) -- (0,0.08) -- (0.3,0.3);
                    \draw (0,0.3) -- (0,0.08);
                \end{tikzpicture}
                +
                \begin{tikzpicture}[centerzero]
                    \draw (-0.2,-0.3) -- (-0.2,0.3);
                    \draw (0,0.3) -- (0.15,0) -- (0.3,0.3);
                    \draw (0.15,0) -- (0.15,-0.3);
                \end{tikzpicture}
                +
                \begin{tikzpicture}[centerzero]
                    \draw (0.2,-0.3) -- (0.2,0.3);
                    \draw (0,0.3) -- (-0.15,0) -- (-0.3,0.3);
                    \draw (-0.15,0) -- (-0.15,-0.3);
                \end{tikzpicture}
                +
                \begin{tikzpicture}[centerzero]
                    \draw (-0.3,0.3) -- (-0.3,0.23) arc(180:360:0.15) -- (0,0.3);
                    \draw (0.3,0.3) -- (0.15,0) -- (-0.2,-0.3);
                    \draw (0.2,-0.3) -- (0.15,0);
                \end{tikzpicture}
                +
                \begin{tikzpicture}[centerzero]
                    \draw (0.3,0.3) -- (0.3,0.23) arc(360:180:0.15) -- (0,0.3);
                    \draw (-0.3,0.3) -- (-0.15,0) -- (0.2,-0.3);
                    \draw (-0.2,-0.3) -- (-0.15,0);
                \end{tikzpicture}
            \right)
            \\
            + \gamma_3
            \left(
                \begin{tikzpicture}[centerzero]
                    \draw (0,0.3) -- (0,-0.15) -- (-0.15,-0.3);
                    \draw (0,-0.15) -- (0.15,-0.3);
                    \draw[wipe] (-0.2,0.3) -- (-0.2,0.25) arc(180:360:0.2) -- (0.2,0.3);
                    \draw (-0.2,0.3) -- (-0.2,0.25) arc(180:360:0.2) -- (0.2,0.3);
                \end{tikzpicture}
                +
                \begin{tikzpicture}[centerzero]
                    \draw (-0.3,0.3) -- (0,0) -- (0.3,0.3);
                    \draw (0,0) -- (-0.15,-0.3);
                    \draw[wipe] (0,0.3) to[out=-45,in=70] (0.15,-0.3);
                    \draw (0,0.3) to[out=-45,in=70] (0.15,-0.3);
                \end{tikzpicture}
                +
                \begin{tikzpicture}[centerzero]
                    \draw (0.3,0.3) -- (0,0) -- (-0.3,0.3);
                    \draw (0,0) -- (0.15,-0.3);
                    \draw[wipe] (0,0.3) to[out=225,in=110] (-0.15,-0.3);
                    \draw (0,0.3) to[out=225,in=110] (-0.15,-0.3);
                \end{tikzpicture}
                +
                \begin{tikzpicture}[centerzero]
                    \draw (-0.3,0.3) -- (-0.15,0.15) -- (0,0.3);
                    \draw (-0.15,0.15) -- (0.15,-0.3);
                    \draw[wipe] (0.3,0.3) -- (-0.15,-0.3);
                    \draw (0.3,0.3) -- (-0.15,-0.3);
                \end{tikzpicture}
                +
                \begin{tikzpicture}[centerzero]
                    \draw (0.3,0.3) -- (0.15,0.15) -- (0,0.3);
                    \draw (0.15,0.15) -- (-0.15,-0.3);
                    \draw[wipe] (-0.3,0.3) -- (0.15,-0.3);
                    \draw (-0.3,0.3) -- (0.15,-0.3);
                \end{tikzpicture}
            \right).
        \end{multlined}
    \end{equation}
    If the morphisms \cref{5pointers} are linearly independent, then this follows immediately from \cref{5dim15} and the fact that the pentagon is invariant under rotation.  So we suppose there is a linear dependence relation involving the morphisms \cref{5pointers}.  Since the last 10 of these morphisms are linearly independent by assumption, it must be the case that the coefficient of one of the first five morphisms is nonzero.  Then we can rotate to assume that the coefficient of the first morphism is nonzero.  We then compose on the bottom with $\Hmor$ to obtain a pentagon from the first term.  Using \cref{chess,F4triangle,qsquare,cow+,boots}, the other terms can be written as linear combinations of the morphisms in \cref{5pointers}.  Summing over all rotations and solving for the pentagon then yields a relation of the form \cref{robot}.

    So now we have a relation of the form \cref{robot}.  Composing \cref{robot} on the top with $\mergemor\ \idstrand$ and using \cref{chess,F4triangle} gives
    \begin{multline*}
        - \frac{[9]}{[3]} \sqmor
        = \gamma_1 \left( \left( \phi - \frac{[9]}{[3]} \right) \left( \Hmor + \Imor \right) + \sqmor \right)
        + \gamma_2 \left( \phi \left(\, \jail + \hourglass\, \right) + \Hmor + \Imor \right)
        \\
        + \gamma_3 \left( q^{12}\, \Hmor + q^{-12}\, \Imor + \phi \poscross + (q^6 + q^{-6})\, \negdotcross \right).
    \end{multline*}
    Using \cref{qsquare,bull-} to write everything in terms of the morphisms \cref{slamdunk}, the coefficient of each of these morphisms must be zero by \cref{stout}.  This yields five equations in $\gamma_1, \gamma_2, \gamma_3$.  Using SageMath to solve these equations recovers the coefficients appearing in \cref{qpent}.
\end{proof}

\begin{rem}
    While it may appear as though we have made a choice of crossings in the last five terms of \cref{qpent}, a straightforward computation using only the skein relation \cref{skein} shows that
    \[
        \begin{tikzpicture}[centerzero]
            \draw (0,0.3) -- (0,-0.15) -- (-0.15,-0.3);
            \draw (0,-0.15) -- (0.15,-0.3);
            \draw[wipe] (-0.2,0.3) -- (-0.2,0.25) arc(180:360:0.2) -- (0.2,0.3);
            \draw (-0.2,0.3) -- (-0.2,0.25) arc(180:360:0.2) -- (0.2,0.3);
        \end{tikzpicture}
        +
        \begin{tikzpicture}[centerzero]
            \draw (-0.3,0.3) -- (0,0) -- (0.3,0.3);
            \draw (0,0) -- (-0.15,-0.3);
            \draw[wipe] (0,0.3) to[out=-45,in=70] (0.15,-0.3);
            \draw (0,0.3) to[out=-45,in=70] (0.15,-0.3);
        \end{tikzpicture}
        +
        \begin{tikzpicture}[centerzero]
            \draw (0.3,0.3) -- (0,0) -- (-0.3,0.3);
            \draw (0,0) -- (0.15,-0.3);
            \draw[wipe] (0,0.3) to[out=225,in=110] (-0.15,-0.3);
            \draw (0,0.3) to[out=225,in=110] (-0.15,-0.3);
        \end{tikzpicture}
        +
        \begin{tikzpicture}[centerzero]
            \draw (-0.3,0.3) -- (-0.15,0.15) -- (0,0.3);
            \draw (-0.15,0.15) -- (0.15,-0.3);
            \draw[wipe] (0.3,0.3) -- (-0.15,-0.3);
            \draw (0.3,0.3) -- (-0.15,-0.3);
        \end{tikzpicture}
        +
        \begin{tikzpicture}[centerzero]
            \draw (0.3,0.3) -- (0.15,0.15) -- (0,0.3);
            \draw (0.15,0.15) -- (-0.15,-0.3);
            \draw[wipe] (-0.3,0.3) -- (0.15,-0.3);
            \draw (-0.3,0.3) -- (0.15,-0.3);
        \end{tikzpicture}
        =
        \begin{tikzpicture}[centerzero]
            \draw (0,-0.15) -- (0.15,-0.3);
            \draw (-0.2,0.3) -- (-0.2,0.25) arc(180:360:0.2) -- (0.2,0.3);
            \draw[wipe] (0,0.3) -- (0,-0.15);
            \draw (0,0.3) -- (0,-0.15) -- (-0.15,-0.3);
        \end{tikzpicture}
        +
        \begin{tikzpicture}[centerzero]
            \draw (0,0) -- (-0.15,-0.3);
            \draw (0,0.3) to[out=-45,in=70] (0.15,-0.3);
            \draw[wipe] (0,0) -- (0.3,0.3);
            \draw (-0.3,0.3) -- (0,0) -- (0.3,0.3);
        \end{tikzpicture}
        +
        \begin{tikzpicture}[centerzero]
            \draw (0,0) -- (0.15,-0.3);
            \draw (0,0.3) to[out=225,in=110] (-0.15,-0.3);
            \draw[wipe] (0,0) -- (-0.3,0.3);
            \draw (0.3,0.3) -- (0,0) -- (-0.3,0.3);
        \end{tikzpicture}
        +
        \begin{tikzpicture}[centerzero]
            \draw (-0.3,0.3) -- (-0.15,0.15) -- (0,0.3);
            \draw (0.3,0.3) -- (-0.15,-0.3);
            \draw[wipe] (-0.15,0.15) -- (0.15,-0.3);
            \draw (-0.15,0.15) -- (0.15,-0.3);
        \end{tikzpicture}
        +
        \begin{tikzpicture}[centerzero]
            \draw (0.3,0.3) -- (0.15,0.15) -- (0,0.3);
            \draw (-0.3,0.3) -- (0.15,-0.3);
            \draw[wipe] (0.15,0.15) -- (-0.15,-0.3);
            \draw (0.15,0.15) -- (-0.15,-0.3);
        \end{tikzpicture}
        .
    \]
\end{rem}

\section{The quantized enveloping algebra of type $F_4$\label{sec:F4}}

In this section we collect some important facts about the quantized enveloping algebra of type $F_4$ that we will use to define a functor from $\Fcat_q$ to the category of finite-dimensional modules over this algebra.  For basic definitions and facts, we follow the presentations of \cite[Ch.~4, 5]{Jan96} and \cite[Ch.~9, 10]{CP95}.  While none of the material in this section is new, we provide proofs of a few results for which we had difficulty finding good references in the literature.

Consider the following labeling of the nodes of the Dynkin diagram of type $F_4$:
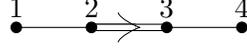
\[
    \begin{tikzpicture}[centerzero]
        \draw (0,0) -- (1,0);
        \draw (2,0) -- (3,0);
        \draw[style=double,double distance=2pt] (1,0) -- (2,0);
        \draw[style=double,double distance=2pt,-{Classical TikZ Rightarrow[length=3mm,width=4mm]}] (1,0) -- (1.65,0);
        \filldraw (0,0) circle (2pt) node[anchor=south] {$1$};
        \filldraw (1,0) circle (2pt) node[anchor=south] {$2$};
        \filldraw (2,0) circle (2pt) node[anchor=south] {$3$};
        \filldraw (3,0) circle (2pt) node[anchor=south] {$4$};
    \end{tikzpicture}
\]
Let $\alpha_i$, $i \in I := \{1,2,3,4\}$, be the corresponding simple roots.  There is a unique inner product $(\ ,\ )$ on the real vector space $\R \Phi$ generated by the root system $\Phi$ such that $(\alpha,\alpha)=2$ for all short roots.  Thus, we have
\[
    (\alpha_1,\alpha_1) = (\alpha_2,\alpha_2) = 4,\qquad
    (\alpha_3,\alpha_3) = (\alpha_4,\alpha_4) = 2.
\]
For $i,j \in I$, we have
\begin{equation} \label{Cartan}
    \frac{2(\alpha_i, \alpha_j)}{(\alpha_i, \alpha_i)} = a_{ij},\quad \text{where }
    [a_{ij}] =
    \begin{bmatrix}
        2 & -1 & 0 & 0 \\
        -1 & 2 & -2 & 0 \\
        0 & -1 & 2 & -1 \\
        0 & 0 & -1 & 2
    \end{bmatrix}
\end{equation}
is the corresponding Cartan matrix.  In what follows, we will write $(i,j)$ for $(\alpha_i, \alpha_j)$, $i,j \in I$.  For each $i \in I$, define
\[
    d_i = \frac{(i,i)}{2},\quad
    \text{so that } d_1 = d_2 = 2,\ d_3 = d_4 = 1.
\]
Let $\Phi_+$ denote the set of positive roots and let
\[
    \rho = \frac{1}{2} \sum_{\gamma \in \Phi_+} \gamma.
\]
Then we have
\[
    (2 \rho, \alpha_i) = (\alpha_i, \alpha_i), \quad i \in I.
\]

Let $J$ denote the maximal ideal of $\Q[q,q^{-1}]$ generated by $q-1$.  Equivalently, $J$ is the kernel of the $\Q$-algebra homomorphism $\Q[q,q^{-1}] \to \Q$, $q \mapsto 1$, which we use to endow $\C$ with the structure of an $A$-module, where
\begin{equation} \label{Adef}
    A := \left\{ \frac{a}{b} : a,b \in \Q[q,q^{-1}],\ b \notin J \right\} \subseteq \Q(q)
\end{equation}
is the localization of $\Q[q,q^{-1}]$ at the ideal $J$.  Since $J$ is also the maximal ideal of $\Q[q,q^{-1}]$ generated by $q^{-1}-1$, the map
\[
    \bar{\ } \colon \Q[q,q^{-1}] \to \Q[q,q^{-1}],\quad \overline{q^{\pm 1}} = q^{\mp 1},
\]
induces a $\Q$-algebra involution of $A$.  The group of units of $A$ is
\[
    A^\times = \left\{ \frac{a}{b} : a,b \in \Q[q,q^{-1}],\ a,b \notin J \right\},
\]
and $A$ is a local ring with unique maximal ideal $(q-1)A$.  Every nonzero element of $A$ is of the form $a (q-1)^n$ for $a \in A^\times$, $n \in \N$.  In particular, $A$ is a principal ideal domain.

\begin{lem} \label{plug}
    If $M$ is a finitely-generated $A$-module such that $\dim_{\Q(q)} M \otimes_A \Q(q) = \dim_\C M \otimes_A \C$, then $M$ is a free $A$-module and $\rank_A M = \dim_\C M \otimes_A \C$.
\end{lem}

\begin{proof}
    By the structure theorem for finitely-generated modules over a principal ideal domain, every finitely-generated $A$-module is a direct sum of ones of the form $A/(q-1)^n$, $n \in \N$.  Now, if $n \ge 1$, then $(A/(q-1)^n) \otimes_A \Q(q) = 0$ but $(A/(q-1)^n) \otimes_A \C \cong \C$.
\end{proof}

For $i \in I$, $a,n \in \Z$, with $n>0$, we define the following invertible elements of $A$:
\begin{gather*}
    q_i := q^{d_i},\qquad
    [a]_i := \frac{q_i^a - q_i^{-a}}{q_i - q_i^{-1}} = \frac{q^{ad_i} - q^{-ad_i}}{q^{d_i} - q^{-d_i}},
    \\
    [n]_i^! := [n]_i [n-1]_i \dotsb [1]_i,\qquad
    \qbinom{a}{n}_i := \frac{[a]_i [a-1]_i \dotsb [a-n+1]_i}{[n]_i^!},\qquad
    \qbinom{a}{0}_i := 1.
\end{gather*}

Let $\fg$ denote the simple Lie algebra of type $F_4$.  Its quantized enveloping algebra $U_q = U_q(\fg)$ is the $\Q(q)$-algebra with generators $E_i, F_i, H_i, K_i, K_i^{-1}$, $i \in I$, and relations (for $i,j \in I$)
\begin{align*}
    K_i K_i^{-1} &= 1 = K_i^{-1} K_i,&
    K_i K_j &= K_j K_i, \\
    K_i E_j K_i^{-1} &= q^{(i,j)} E_j,&
    K_i F_j K_i^{-1} &= q^{-(i,j)} F_j, \\
    E_i F_j - F_j E_i &= \delta_{ij} H_i,&
    (q_i-q_i^{-1})H_i &= K_i - K_i^{-1},
\end{align*}
where $\delta_{ij}$ is the Kronecker delta, and (for $i \ne j$),
\[
    \sum_{r=0}^{1-a_{ij}} (-1)^r \qbinom{1-a_{ij}}{r}_i E_i^{1-a_{ij}-r} E_j E_i^r = 0, \qquad
    \sum_{r=0}^{1-a_{ij}} (-1)^r \qbinom{1-a_{ij}}{r}_i F_i^{1-a_{ij}-r} F_j F_i^r = 0.
\]
For $\nu = \sum_{i \in I} n_i \alpha_i$, $n_i \in \Z$, we define $K_\nu = \prod_{i \in I} K_i^{n_i}$.  There is a unique Hopf algebra structure on $U_q$ such that, for all $i \in I$,
\begin{align} \label{HopfE}
    \Delta(E_i) &= E_i \otimes 1 + K_i \otimes E_i,&
    \varepsilon(E_i) &= 0,&
    S(E_i) &= - K_i^{-1} E_i,
    \\ \label{HopfF}
    \Delta(F_i) &= F_i \otimes K_i^{-1} + 1 \otimes F_i,&
    \varepsilon(F_i) &=0,&
    S(F_i) &= - F_i K_i,
    \\ \label{HopfK}
    \Delta(K_i) &= K_i \otimes K_i,&
    \varepsilon(K_i) &=1,&
    S(K_i) &= K_i^{-1}.
\end{align}
\details{
    There are different conventions for the Hopf algebra structure.  We have chosen the one from \cite[Prop.~4.11]{Jan96} and \cite[Lem.~3.1.4]{Lus10}.  The one in \cite[p.~281]{CP95} is the opposite of the one given above.  Note that the above equations imply, for example, that $\Delta(H_i) = H_i \otimes K_i^{-1} + K_i \otimes H_i$ for $i \in I$.
}

Now let $U_A$ be the $A$-subalgebra of $U_q$ generated by the elements
\[
    E_i,\quad F_i,\quad K_i^{\pm 1},\qquad
    i \in I.
\]
(This implies that $U_A$ also contains $H_i = E_i F_i - F_i E_i$, $i \in I$.)  Note that this is the same as the $A$-subalgebra of $U_q$ generated by the elements
\[
    E_i^{(r)} := \frac{E_i^r}{[r]_i^!},\quad F_i^{(r)}:= \frac{F_i^r}{[r]_i^!},\quad K_i^{\pm 1},\qquad
    i \in I,\ r \ge 1,
\]
since the $[r]_i^!$ are invertible in $A$.  Thus $U_A$ is an extension of both the restricted and non-restricted integral forms of $U_q$; see \cite[\S9.2, \S9.3]{CP95}.  Note that $U_A$ is a Hopf subalgebra of $U_q$ (viewing $U_q$ as a Hopf algebra over $A$).  Furthermore, we have isomorphisms of algebras
\begin{equation} \label{sedai}
    U_A \otimes_A \Q(q) \cong U_q,\qquad
    (U_A \otimes_A \C)/(K_i-1 : i \in I) \cong U_\C,
\end{equation}
where $U_\C = U(\fg)$ denotes the universal enveloping algebra of $\fg$ over $\C$, and we view $\C$ as an $A$-module via the map $q \mapsto 1$.  See, for example, \cite[Prop.~9.2.3]{CP95}.

The \emph{bar involution} is the antilinear ring involution
\[
    \bar{\ } \colon U_q \to U_q,\qquad
    \overline{E_i} = E_i,\quad
    \overline{F_i} = F_i,\quad
    \overline{K_i^{\pm 1}} = K_i^{\mp 1}.
\]
By \emph{antilinear}, we mean that it is $\Q$-linear and $\overline{q^{\pm 1} x} = q^{\mp 1} \overline{x}$ for $x \in U_q$.  The bar involution induces an antilinear involution of $U_A$, which we also call the bar involution.

Let $\Mcat_q$ denote the category of finite-dimensional $U_q$-modules of type $1$, and let $\Mcat_\C$ denote the category of finite-dimensional $U_\C$-modules.  The categories $\Mcat_q$ and $\Mcat_\C$ are both semisimple, with irreducible objects given by highest-weight representations with dominant integral highest weights.  Furthermore, tensor product multiplicities are the same in the two categories, as are dimensions of homomorphism spaces between corresponding modules.
\details{
    The classification of objects in $\Mcat_q$ can be found in \cite[\S10.1]{CP95} and \cite[Ch.~5]{Jan96}.  The fact that tensor product multiplicities are the same can be proved as in the proof of \cite[Prop.~10.1.16]{CP95}.  Then equality of dimensions of homomorphism spaces follows from Schur's lemma.
}

For each dominant integral weight $\lambda$, let $L_q(\lambda)$ denote the corresponding type $1$ simple $U_q$-module, and let $L_A(\lambda)$ denote its $A$-form.  Thus $L_A(\lambda)$ is the $U_A$-submodule of $L_q(\lambda)$ generated by a chosen highest weight vector $v_\lambda$.  Then, under the isomorphisms \cref{sedai},
\[
    L_A(\lambda) \otimes_A \Q(q) \cong L_q(\lambda),\qquad
    L_A(\lambda) \otimes_A \C \cong L_\C(\lambda),
\]
where $L_\C(\lambda)$ is the simple $U_\C$-module of highest weight $\lambda$.  Furthermore, we can choose an $A$-basis $\bB(\lambda)$ of $L_A(\lambda)$, and this induces a $\Q(q)$-basis of $L_q(\lambda)$ and a $\C$-basis of $L_\C(\lambda)$.  We will denote the image of the chosen highest weight vector $v_\lambda$ in these two tensor products again by $v_\lambda$.  Let $\Mcat_A$ denote the category of $U_A$-modules that are direct sums of $L_A(\lambda)$ for dominant integral weights $\lambda$.

Throughout this section and the next, we will be moving between the three categories $\Mcat_q$, $\Mcat_A$, and $\Mcat_\C$.  When we wish to make a statement that holds in each of these categories, we use a bullet $\bullet$.  For example, $U_\bullet$ can refer to $U_q$, $U_A$, or $U_\C$.  Then $\kk$ will denote $\Q(q)$, $A$ , or $\C$, respectively.

Since the longest element of the Weyl group of type $F_4$ acts on the weight lattice as multiplication by $-1$, it follows that $L_q(\lambda)^* \cong L_q(\lambda)$ for all dominant integral weights $\lambda$.  Thus we have nonzero $U_q$-module homomorphisms, unique up to multiplication by an element of $\Q(q)^\times$,
\[
    D_q^\lambda \colon L_q(\lambda) \otimes L_q(\lambda) \to \Q(q),
\]
where we view $\Q(q)$ as the trivial one-dimensional module.  Now fix an element $v_{-\lambda}$ spanning the $(-\lambda)$-weight space $L_A(\lambda)_{-\lambda}$ of $L_A(\lambda)$.  We normalize $D_q^\lambda$ so that
\[
    D_q^\lambda ( v_\lambda \otimes v_{-\lambda} ) = 1.
\]
Similarly, there is a unique $U_\C$-module homomorphism
\begin{equation} \label{skates}
    D_\C^\lambda \colon L_\C(\lambda) \otimes_\C L_\C(\lambda) \to \C
    \quad \text{such that} \quad
    D_\C^\lambda(v_\lambda \otimes v_{-\lambda}) = 1.
\end{equation}

\begin{lem}
    If $\lambda$ is a dominant integral weight, then
    \begin{equation} \label{ricochet}
        D_q^\lambda(x v \otimes w) = D_q^\lambda(v \otimes S(x)w),\quad
        v,w \in L_q(\lambda),\ x \in U_q.
    \end{equation}
\end{lem}

\begin{proof}
    It suffices to prove that \cref{ricochet} holds for $x \in \{E_i, F_i, K_i : i \in I\}$.  For $v,w \in L_q(\lambda)$ and $i \in I$, we have
    \[
        D_q^\lambda(v \otimes S(K_i) w)
        = D_q^\lambda(v \otimes K_i^{-1} w)
        = K_i D_q^\lambda(v \otimes K_i^{-1} w)
        = D_q^\lambda(K_i v \otimes w)
    \]
    and
    \begin{multline*}
        0 = E_i D_q^\lambda(v \otimes w)
        = D_q^\lambda(E_i v \otimes w) + D_q^\lambda(K_i v \otimes E_i w)
        \\
        = D_q^\lambda(E_i v \otimes w) + D_q^\lambda(v \otimes K_i^{-1} E_i w)
        = D_q^\lambda(E_i v \otimes w) - D_q^\lambda(v \otimes S(E_i) w)
    \end{multline*}
    and
    \[
        0 = F_i D_q^\lambda(v \otimes K_i w)
        = D_q^\lambda(F_i v \otimes K_i^{-1} K_i w) + D_q^\lambda(v \otimes F_i K_i w)
        = D_q^\lambda(F_i v \otimes w) - D_q^\lambda(v \otimes S(F_i) w).
        \qedhere
    \]
\end{proof}

Restriction of $D_q^\lambda$ gives a map $L_A(\lambda) \otimes L_A(\lambda) \to \Q(q)$ sending $v_\lambda \otimes v_{-\lambda}$ to $1$.  Now, by definition, every element of $L_A(\lambda)$ can be written as $x v_\lambda$ for some $x \in U_A$.  Since $L_\C(\lambda)$ is an irreducible $U_\C$-module, there exists a sequence of $i_1,\dotsc,i_r$ of elements of $I$ such that $E_{i_r} \dotsm E_{i_1} v_{-\lambda} \otimes_A \C$ is a nonzero element of $L_\C(\lambda)_\lambda$.  It follows that $E_{i_r} \dotsm E_{i_1} v_{-\lambda} \subseteq A^{\times} v_\lambda$, and so every element of $L_A(\lambda)$ can also be written as $y v_{-\lambda}$ for some $y \in U_A$.  For $x,y \in U_A$,
\[
    D_q^\lambda (x v_\lambda \otimes y v_{-\lambda})
    = D_q^\lambda ( v_\lambda \otimes S(x) y v_{-\lambda} )
    \subseteq D_q^\lambda ( v_\lambda \otimes U_A v_{-\lambda} )
    = D_q^\lambda (v_\lambda \otimes A v_{-\lambda})
    = A,
\]
where the second-to-last equality follows from the fact that $D_q^\lambda (v_\lambda \otimes w) = 0$ for $w \in L_A(\lambda)_\mu$, $\mu \ne -\lambda$.  It follows that the restriction of $D_q^\lambda$ yields a homomorphism of $U_A$-modules
\begin{equation}
    D_A^\lambda \colon L_A(\lambda) \otimes_A L_A(\lambda) \to A
    \quad \text{such that} \quad
    D_A^\lambda(v_\lambda \otimes v_{-\lambda}) = 1.
\end{equation}

Since $D^\lambda_q$ induces the isomorphism $L_q(\lambda)^* \cong L_q(\lambda)$, it is nondegenerate.  Let $\bB(\lambda)^\vee = \{b^\vee : b \in \bB(\lambda)\}$ be the basis of $L_q(\lambda)$ dual to $\bB(\lambda)$ with respect to $D_q^\lambda$, so that
\[
    D_q^\lambda(b^\vee \otimes c) = \delta_{b,c},\qquad b,c \in \bB(\lambda).
\]
The following lemma implies that $\bB(\lambda)^\vee$ is also the basis dual to $\bB(\lambda)$ with respect to $D_A^\lambda$ and $D_\C^\lambda$.

\begin{lem} \label{icebubble}
    We have $\bB(\lambda)^\vee \subseteq L_A(\lambda)$.
\end{lem}

\begin{proof}
    Let $n = \dim_\C L_\C(\lambda)$.  Taking coefficients in the basis $\bB(\lambda)$, we can identify $L_\bullet(\lambda)$ with $\kk^n$, whose elements we write as column vectors.  Then $D^\lambda_q$ corresponds to an invertible $n \times n$ matrix $M$ with entries in $\Q(q)$:
    \[
        D^\lambda_q(v \otimes w) = v^t M w,\qquad v,w \in L_q(\lambda),
    \]
    where $v^t$ denotes the transpose of an element of $\kk^n$.  Since restriction to $L_A(\lambda)$ yields  $D_A^\lambda$, the entries of $M$, which are all of the form $D^\lambda_q(b \otimes c)$, for $b,c \in \bB(\lambda)$, lie in $A$.  Furthermore, tensoring over $A$ with $\C$ yields $D_\C^\lambda$.  Since $D_\C^\lambda$ is nondegenerate, the matrix $M$ is invertible at $q=1$.  This implies that $\det M \in A^\times$, and hence $M$ is also invertible over $A$.  Therefore, $L_A(\lambda)$ has a basis $\bB'$ that is dual to $\bB(\lambda)$ with respect to $D_A^\lambda$.  But then $\bB'$ is also dual to $\bB(\lambda)$ with respect to $D_q^\lambda$.  Since dual bases are unique, this implies that $\bB(\lambda)^\vee = \bB' \subseteq L_A(\lambda)$.
\end{proof}

\begin{rem} \label{groove}
    Our choices of $v_\lambda$ and $v_{-\lambda}$ are unique up to multiplication by an element of $A^\times$.  It follows that the bilinear forms $D_\bullet^\lambda$ are also unique up to such a multiple.  We will use this fact later, in \cref{groundwork}, where we will scale the $D_\bullet^\lambda$ by an appropriate factor.
\end{rem}

\begin{lem}
    If $\lambda$ is a dominant integral weight, then
    \begin{equation} \label{trolloc}
        D_q^\lambda(v \otimes w)
        = D_q^\lambda(w \otimes K_{-2\rho}v)
        = D_q^\lambda \big( K_{2\rho}w \otimes v \big),
        \qquad v,w \in L_q(\lambda).
    \end{equation}
\end{lem}

\begin{proof}
    It suffices to prove the first equality, since the second then follows from \cref{ricochet}.  Let $V = L_q(\lambda)$.  First note that the map
    \[
        \chi \colon V \to V^*,\quad \chi(v)(w) = D_q(v \otimes w),
    \]
    is an isomorphism of $U_q$-modules.  Indeed, for $v,w \in V$ and $x \in U_q$, we have
    \[
        \chi(xv)(w)
        = D_q(xv \otimes w)
        \overset{\cref{ricochet}}{=} D_q(v \otimes S(x)w)
        = \chi(v)(S(x)w)
        = (x \chi(v))(w).
    \]
    Consider also the $U_q$-module homomorphisms
    \begin{align*}
        \ev &\colon V^* \otimes V \to \Q(q),& f \otimes v &\mapsto f(v), \\
        \ev' &\colon V \otimes V^* \to \Q(q),& v \otimes f &\mapsto f(K_{-2\rho}v).
    \end{align*}
    See, for example, \cite[3.9(3), 3.9(4)]{Jan96}.  (The statements there are for $\mathfrak{sl}_2$, but the generalization to arbitrary finite type is discussed in \cite[\S5.3]{Jan96}.)

    By the tensor-hom adjunction and the fact that $V$ is self-dual, we have
    \begin{equation} \label{arrow}
        \dim \Hom_{U_q}(V \otimes V, \Q(q)) = \dim \Hom_{U_q}(V,V) = 1,
    \end{equation}
    since $V$ is simple.  Thus, the compositions $\ev \circ (\chi \otimes 1_V)$ and $\ev' \circ (1_V \otimes \chi)$ differ by a scalar.  Since
    \[
        D_q^\lambda(v \otimes w) = \ev \circ (\chi \otimes 1_V)(v,w)
        \qquad \text{and} \qquad
        D_q^\lambda(w \otimes K_{-2\rho} v) = \ev' \circ (1_V \otimes \chi)(v,w),
    \]
    it suffices to show that \cref{trolloc} holds for two vectors $v$ and $w$ such that $D_q^\lambda(v \otimes w)$ is nonzero.

    Let $v \in V_\lambda$ be a nonzero highest-weight vector.  Since the determinant of the Cartan matrix \cref{Cartan} is one, the weight lattice of $\fg$ is equal to its root lattice, and so $\lambda$ is a sum of simple roots.  Therefore, since $V$ is irreducible and self-dual (in particular, its set of weights is invariant under the map $\nu \mapsto -\nu$), there exists a sequence $i_1,\dotsc,i_r$ of elements of $I$ such that $\sum_{j=1}^r \alpha_{i_j} = \lambda$ and, setting $y := F_{i_1} F_{i_2} \dotsm F_{i_r} F_{i_r} \dotsm F_{i_2} F_{i_1}$,  the vector $w:= yv$ is a nonzero element of the lowest weight space $V_{-\lambda}$.  Using \cite[4.13(4)]{Jan96}, we have
    \[
        S(y) = q^{2(\lambda,\rho-\lambda)} y K_{2\lambda}
        \quad \text{and so} \quad
        S^{-1}(y) = q^{2(\lambda,\lambda-\rho)} K_{2\lambda} y.
    \]
    \details{
        We use here that $\tau(y)=y$, where $\tau$ is given in \cite[Lem.~4.6(b)]{Jan96}, and the fact that, using \cite[4.13(5)]{Jan96} and the notation introduced there, we have
        \[
            m(2 \lambda) = \frac{1}{2}\left( (2\lambda,2\lambda) - \sum_{\alpha \in \Pi} m_\alpha (\alpha, \alpha) \right) = 2(\lambda, \lambda-\rho).
        \]
    }
    Thus
    \begin{multline*}
        D_q^\lambda (v \otimes w)
        = D_q^\lambda (v \otimes yv)
        \overset{\cref{ricochet}}{=} D_q^\lambda (S^{-1}(y) v \otimes v)
        = q^{2(\lambda - \rho,\lambda)} D_q^\lambda(K_{2\lambda} y v \otimes v)
        \\
        = q^{-(2\rho,\lambda)} D_q^\lambda(w \otimes v)
        = D_q^\lambda(w \otimes K_{-2\rho}v).
        \qedhere
    \end{multline*}
\end{proof}

We next recall the definition of the $R$-matrix following the approach of \cite[\S32.1]{Lus10}.  Let $\Theta$ be the \emph{quasi-$R$-matrix} from \cite[\S4.1]{Lus10}.  This is an infinite sum over the positive root lattice
\begin{equation} \label{Theta}
    \Theta = \sum_\nu \Theta_\nu,\quad
    \Theta_\nu \in U_A(\fg)^-_{-\nu} \otimes_A U_A(\fg)^+_\nu,\quad
    \Theta_0 = 1 \otimes 1.
\end{equation}
The fact that the components of $\Theta$ lie in the $A$-form follows from \cite[Cor.~24.1.6]{Lus10}.  We also have that $\Theta = 1 \otimes 1$ at $q=1$.  Indeed, by \cite[Cor.~4.1.3]{Lus10}, $\Theta \bar{\Theta} = \bar{\Theta} \Theta = 1 \otimes 1$.  Thus, at $q=1$, we have $\Theta^2 = 1 \otimes 1$.  Then a straightforward argument by induction on the height of $\nu$ shows that $\Theta_\nu = 0$ at $q=1$ for $\nu \ne 0$.

For finite-dimensional $U_q$-modules $V$ and $W$, let
\[
    \flip \colon V \otimes W \to W \otimes V,\quad
    \flip (v \otimes w) := w \otimes v,
\]
be the tensor flip.  We also define
\[
    \Pi \colon V \otimes W \to V \otimes W,\quad
    \Pi(v \otimes w) := q^{-(\lambda,\mu)} v \otimes w,
\]
for $v$ of weight $\lambda$ and $w$ of weight $\mu$.  Then we have the \emph{$R$-matrix}
\[
    R_{V,W} := \Theta \circ \flip \circ \Pi \colon V \otimes W \xrightarrow{\cong} W \otimes V.
\]
The action of $\Theta$ is well defined since all but finitely many of the components $\Theta_\nu$ act as zero.  The inverse $R_{V,W}^{-1} \colon W \otimes V \to V \otimes W$ is given by $R_{V,W}^{-1} = \Pi^{-1} \circ \flip^{-1} \circ \overline{\Theta}$, where $\overline{\Theta}$ is obtained from $\Theta$ by applying the bar involution to each tensor factor.  Note that our $R_{V,W}$ is Lusztig's ${}_f\mathcal{R}_{W,V}$, taking the function $f$ of \cite[\S32.1.3]{Lus10} to be $f(\lambda,\mu) = -(\lambda,\mu)$ and noting that our $q$ is denoted $v$ in \cite{Lus10}.  See also \cite[\S7.9]{Jan96} for a discussion of this choice, noting that we may take the $d$ there to be $1$ since the root lattice is equal to the weight lattice in our setting.  Since $\Theta = 1 \otimes 1$ at $q=1$, it follows that $R_{V,W} = \flip$ at $q=1$.

\begin{lem}
    If $\lambda$ is a dominant integral weight, then
    \begin{equation} \label{curly}
        D_q^\lambda \circ R_{L_q(\lambda),L_q(\lambda)} = q^{(\lambda,\lambda + 2\rho)} D_q^\lambda.
    \end{equation}
\end{lem}

\begin{proof}
    Let $V = L_q(\lambda)$.  By \cref{arrow}, there exists $c \in \Q(q)$ such that $D_q^\lambda \circ R_{V,V} = c D_q^\lambda$.  It remains to show that $c=q^{(\lambda,\lambda+2\rho)}$.  Let $v \in V_\lambda$ be a nonzero highest-weight vector and let $w \in V_{-\lambda}$ be a nonzero lowest-weight vector.  We compute
    \[
        D_q^\lambda \circ R_{V,V} (v \otimes w)
        = D_q^\lambda \circ \Theta \circ \flip \circ \Pi (v \otimes w)
        \overset{\cref{Theta}}{=} q^{(\lambda,\lambda)} D_q^\lambda (w \otimes v)
        \overset{\cref{trolloc}}{=}  q^{(\lambda+2\rho,\lambda)} D_q^\lambda (v \otimes w).
    \]
    Hence $c = q^{(\lambda,\lambda + 2\rho)}$, as desired.
\end{proof}

Recall that, for $V \in \Mcat_q$, the \emph{quantum trace} is the $\Q(q)$-linear map
\[
    \tr_q \colon \End_{\Q(q)}(V) \to \Q(q),\quad
    f \mapsto \tr(f \circ K_{-2\rho}),
\]
where $\tr$ denotes the usual trace of a linear map.  (See \cite[\S5.3]{Jan96}.)  The \emph{quantum dimension} of $V$ is
\[
    \qdim(V) = \tr_q(1_V) = \tr(K_{-2\rho}).
\]
The quantum dimension of the simple modules is given by the \emph{quantum Weyl dimension formula}.  Even though this formula is well known, we were unable to find a proof in the literature.  Hence we have included a proof for the sake of completeness.

\begin{lem}[Quantum Weyl dimension formula]
    We have
    \begin{equation} \label{qWeyl}
        \dim_q L_q(\lambda)
        = \prod_{\nu \in \Phi^+} \frac{[(\lambda + \rho, \nu)]}{[(\rho,\nu)]}.
    \end{equation}
\end{lem}

\begin{proof}
    It suffices to show that the equality \cref{qWeyl} holds when $q$ is specialized to an arbitrary positive integer and so we assume $q$ is a positive integer.  The Weyl character formula gives that, in the $U_q$-module $L_q(\lambda)$,
    \begin{equation} \label{sunny}
        \tr \left( e^{(\log q) \gamma^\vee} \right)
        = \frac{\sum_{w \in W} \det(w) q^{(w(\lambda+\rho),\gamma)}}{\prod_{\nu \in \Phi_+} \left( q^{(\nu,\gamma)/2} - q^{-(\nu,\gamma)/2} \right) },
        \qquad \gamma \in \R \Phi,
    \end{equation}
    where $W$ is the Weyl group, $\det(w)$ is the determinant of the action of $w$ on the Cartan subalgebra, and $\gamma^\vee$ is the element of the Cartan subalgebra satisfying $\nu(\gamma^\vee) = (\nu,\gamma)$ for all $\nu \in \R \Phi$.  The case $\lambda=0$ gives the Weyl denominator identity
    \[
        \sum_{w \in W} \det(w) q^{(w(\rho),\gamma)}
        = \prod_{\nu \in \Phi_+} \left( q^{(\nu,\gamma)/2} - q^{-(\nu,\gamma)/2} \right).
    \]
    Since the inner product $(\ ,\ )$ is invariant under the action of the Weyl group and $\det(w) = \det(w^{-1})$ for all $w \in W$ (because elements of $w$ are products of reflections), this implies that
    \begin{equation} \label{rays}
        \sum_{w \in W} \det(w) q^{(w(\lambda+\rho),-2\rho)}
        = \sum_{w \in W} \det(w) q^{-2(\lambda+\rho,w(\rho))}
        = \prod_{\nu \in \Phi_+} \left( q^{-(\lambda+\rho,\nu)} - q^{(\lambda+\rho,\nu)} \right).
    \end{equation}
    Now, for all $v \in L_q(\lambda)$ of weight $\mu$, we have
    \[
        e^{- 2 (\log q) \rho^\vee} v
        = q^{\mu(-2\rho^\vee)} v
        = q^{(\mu,-2\rho)} v.
    \]
    Hence
    \begin{multline*}
        \qdim(L_q(\lambda)) = \tr(K_{-2\rho})
        = \tr \left( e^{-2(\log q)\rho^\vee} \right)
        \overset{\cref{sunny}}{=}  \frac{\sum_{w \in W} \det(w) q^{(w(\lambda+\rho),-2\rho)}}{\prod_{\nu \in \Phi_+} \left( q^{-(\nu,\rho)} - q^{(\nu,\rho)} \right) }
        \\
        \overset{\cref{rays}}{=} \prod_{\nu \in \Phi_+} \frac{ q^{(\lambda+\rho,\nu)} - q^{-(\lambda+\rho,\nu)} }{ q^{(\rho,\nu)} - q^{-(\rho,\nu)} }
        = \prod_{\nu \in \Phi^+} \frac{[(\lambda + \rho, \nu)]}{[(\rho,\nu)]}.
        \qedhere
    \end{multline*}
\end{proof}

\section{The functor to the category of modules\label{sec:functor}}

The goal of this section is to define a full functor from the category $\Fcat_q$ of \cref{Fdef} to the category $\Mcat_q$ of finite-dimensional $U_q$-modules of type $1$.  We will first define this functor on $\Tcat$, and then use the results of \cref{sec:2point,sec:specialize,sec:F4} to deduce that this functor factors through the quotient category $\Fcat_q$.

Throughout this section, we set
\begin{gather}
    \tau = q^{12},\qquad \beta = q^{24},
    \\ \label{tree}
    \delta = \frac{[3][8][13][18]}{[4][6][9]},\qquad
    \phi = \frac{[2][7][12]}{[4][6]},\qquad
    z = \frac{q^{-4}-q^4}{[3]},
    \qquad \text{where} \quad
    [n] = \frac{q^n-q^{-n}}{q-q^{-1}}.
\end{gather}
We let $\Tcat_q$ and $\Tcat_A$ denote the corresponding categories from \cref{Tdef} with $\kk = \Q(q)$ and $\kk = A$, respectively.

\begin{rem} \label{degenerate3}
    Since
    \[
        \delta = 26,\quad
        \phi = 7,\quad
        z = 0,\qquad
        \text{when} \quad q = 1,
    \]
    we see that the quotient of
    \[
        \Tcat_\C := \Tcat_A \otimes_A \C,
    \]
    by the relation $\poscross = \negcross$ is isomorphic to the category $\Tcat_{7,26}$ of \cite[Def.~2.1]{GSZ21}, where these crossings are denoted by $\crossmor$.  (See \cref{degenerate,degenerate2}.)  As noted in \cite[Rem.~2.2]{GSZ21}, $\Tcat_{\phi,\delta}$ is actually independent of $\phi$ (which is denoted by $\alpha$ there).  Therefore, even though the choice $\phi = 7/3$ was used for most results in \cite{GSZ21}, these results remain true for $\phi=7$ after translating via the isomorphism from $\Tcat_{7/3,26}$ to $\Tcat_{7,26}$.
\end{rem}

Let $\omega_i$, $i \in I$, be the fundamental weights of $\fg$, and let $\sV_q := L_q(\omega_4)$.  Fix a highest weight vector $\sv \in \sV_q$ and let $\sV_A := U_A \sv$ be the corresponding $A$-form $L_A(\omega_4)$ of $\sV_q$.  Let $\sV_\C := \sV_A \otimes_A \C$ be the corresponding simple highest-weight $U_\C$-module $L_\C(\omega_4)$.  As noted in \cref{sec:F4}, we have
\begin{equation} \label{business}
    \dim_{\Q(q)} \Hom_{U_q}(\sV_q^{\otimes n}, \sV_q^{\otimes m})
    = \dim_\C \Hom_{U_\C}(\sV_\C^{\otimes n}, \sV_\C^{\otimes m})
    \qquad \text{for all } m,n \in \N.
\end{equation}

Note that
\begin{equation} \label{cider}
    (\omega_4, \omega_4) = 2,\quad
    (\omega_4,2\rho) = 22,\quad
    \text{and so}\quad
    (\omega_4, \omega_4 + 2\rho) = 24.
\end{equation}
Using SageMath to compute the quantum dimension of $\sV_q$ via the quantum Weyl dimension formula \cref{qWeyl}, we see that
\begin{equation}
    \qdim \sV_q = \frac{[3] [8] [13] [18]}{[4] [6] [9]} = \delta.
\end{equation}
In addition, computation in SageMath shows that we have a decomposition of $U_q$-modules
\begin{equation} \label{lizard}
    \sV_q^{\otimes 2} = L_q(0) \oplus L_q(\omega_1) \oplus L_q(\omega_3) \oplus \sV_q \oplus L_q(2 \omega_4).
\end{equation}

Choose an $A$-basis $\bB_\sV$ for $\sV_A$ consisting of weight vectors.  It follows that the images of $\bB_\sV$ under the maps $- \otimes_A \Q(q)$ and $- \otimes_A \C$, which we also denote by $\bB_\sV$, are a $\Q(q)$-basis of $\sV_q$ and a $\C$-basis of $\sV_\C$, respectively.  By \cref{lizard}, there is a one-dimensional subspace of $\sV_q^{\otimes 2}$ of weight $\omega_4$ that is annihilated by $E_i$ for $i \in I$.  Let $v_\sT$ be a nonzero vector in this subspace.  We may write $v_\sT$ as a $\Q(q)$-linear combination of the elements of the basis $\bB_\sV^{\otimes 2}$ for $\sV_q^{\otimes 2}$.  Multiplying by an appropriate element of $\Q(q)$ to clear denominators, we may assume that $v_\sT$ is actually a $\Q[q,q^{-1}]$-linear combination of elements of $\bB_\sV^{\otimes 2}$.  Then, dividing all coefficients by the highest power of $q-1$ dividing them all, we may assume that at least one of the coefficients is not divisible by $q-1$.  Hence the image of $v_\sT$ in $\sV_\C \otimes_\C \sV_\C$ is nonzero.

Now, by \cref{lizard}, there is a homomorphism of $U_q$-modules
\begin{equation} \label{lobster}
    \sT_q \colon \sV_q^{\otimes 2} \to \sV_q
\end{equation}
that is unique up to multiplication by an element of $\Q(q)$.  We fix this scalar by declaring that the vector $v_\sT$ defined above is mapped to $\sv$.  It follows that $\sT_q$ induces nonzero module homomorphisms
\[
    \sT_A \colon \sV_A^{\otimes 2} \to \sV_A
    \qquad \text{and} \qquad
    \sT_\C \colon \sV_\C^{\otimes 2} \to \sV_\C.
\]

Let $\sD_\bullet := D_\bullet^{\omega_4}$ and let $\bB_\sV^\vee$ be the basis dual to $\bB_\sV$ with respect to $\sD_\bullet$.  (Recall that it follows from \cref{icebubble} that this dual basis is the same for $\sD_q$, $\sD_A$, and $\sD_\C$.)  Note that, whenever we refer to $\sD_\bullet$, we are referring to a triple $(D_A,D_q,D_\C) = (D_A^{\omega_4}, D_q^{\omega_4}, D_\C^{\omega_4})$ of forms related as described in \cref{sec:F4}.  For any $w \in \sV_A$, we have
\begin{equation} \label{snail}
    w = \sum_{v \in \bB_\sV} \sD_A(v^\vee \otimes w) v
    = \sum_{v \in \bB_\sV} \sD_A(w \otimes v) v^\vee.
\end{equation}
\details{
    There exist $c_v \in A$, for $v \in \bB_\sV$, such that $w = \sum_{v \in \bB_\sV} c_v v$.  Then, for $v \in \bB_\sV$, we have
    \[
        \sD_A(v^\vee \otimes w)
        = \sum_{u \in \bB_\sV} c_u \sD_A(v^\vee \otimes u) = c_v.
    \]
    The proof of the second equality in \cref{snail} is similar.
}

\begin{lem}
    The linear map
    \[
        \sC_\bullet \colon \kk \to \sV_\bullet^{\otimes 2},\qquad
        \sC_\bullet(1) = \sum_{v \in \bB_\sV} v \otimes v^\vee,
    \]
    is a homomorphism of $U_\bullet$-modules.
\end{lem}

\begin{proof}
    It suffices to prove the result over $A$.  For $x \in U_A$, we have (using sumless Sweedler notation)
    \begin{multline*}
        x \sum_{v \in \bB_\sV} v \otimes v^\vee
        = \sum_{v \in \bB_\sV} x_{(1)} v \otimes x_{(2)} v^\vee
        \overset{\cref{snail}}{=} \sum_{v,w \in \bB_\sV} \sD_A(w^\vee \otimes x_{(1)} v) w \otimes x_{(2)} v^\vee
        \\
        \overset{\cref{ricochet}}{=} \sum_{v,w \in \bB_\sV} \sD_A(S^{-1}(x_{(1)}) w^\vee \otimes v) w \otimes x_{(2)} v^\vee
        \overset{\cref{snail}}{=} \sum_{w \in \bB_\sV} w \otimes x_{(2)} S^{-1}(x_{(1)}) w^\vee
        \\
        = \sum_{w \in \bB_\sV} w \otimes S^{-1}(x_{(1)} S(x_{(2)})) w^\vee
        = \varepsilon(x) \sum_{w \in \bB_\sV} w \otimes w^\vee.
        \qedhere
    \end{multline*}
\end{proof}

Let $\sR_\bullet := R_{\sV_\bullet, \sV_\bullet}$.  Recall from \cref{groove} that $\sD_\bullet$ is unique up to multiplication by an element of $A^\times$.  Note that scaling $\sD_\bullet$ by $a \in A^\times$ automatically scales $\sC_\bullet$ by $a^{-1}$.

\begin{prop} \label{groundwork}
    We can choose $\sD_\bullet$ so that there exist monoidal functors $\bFp_\bullet \colon \Tcat_\bullet \to \Mcat_\bullet$ given by $\go \mapsto \sV_\bullet$ and
    \[
        \mergemor \mapsto \sT_\bullet,\quad
        \poscross \mapsto \sR_\bullet,\quad
        \negcross \mapsto \sR_\bullet^{-1},\quad
        \cupmor \mapsto \sC_\bullet,\quad
        \capmor \mapsto \sD_\bullet.
    \]
    Furthermore, the diagram
    \begin{equation} \label{teatime}
        \begin{tikzcd}[column sep=large]
            \Tcat_q \arrow[d, "\bFp_q"'] & \Tcat_A \arrow[l, "- \otimes_A \Q(q)"'] \arrow[r, "- \otimes_A \C"] \arrow[d, "\bFp_A"] & \Tcat_\C \arrow[d, "\bFp_\C"] \\
            \Mcat_q & \Mcat_A \arrow[l, "- \otimes_A \Q(q)"'] \arrow[r, "- \otimes_A \C"] & \Mcat_\C
        \end{tikzcd}
    \end{equation}
    commutes.
\end{prop}

\begin{proof}
    For now, we fix any choice of $\sD_\bullet$.  We will then rescale it at the end of the proof.

    If the functors are well defined, it follows immediately from the definitions of $\sT_\bullet$, $\sR_\bullet$, $\sC_\bullet$, and $\sD_\bullet$ that the diagram \cref{teatime} commutes.  Thus we must verify that $\bFp_\bullet$ respects the relations \cref{vortex,venom,chess}.  Once we have shown that a relation is respected by $\bFp_A$, tensoring over $A$ with $\Q(q)$ or $\C$ shows that this relation is also respected by $\bFp_q$ and $\bFp_\C$, respectively.  So it suffices to consider $\bFp_A$.  Now, the morphism spaces in $\Mcat_A$ are free $A$-modules since, by definition, objects of $\Mcat_A$ are direct sums of $L_A(\lambda)$ for dominant integral weights $\lambda$, and $\Hom_{U_A}(L_A(\lambda), L_A(\mu))$ is a free $A$-module of rank one if $\lambda = \mu$ and is zero otherwise.  Thus, a morphism is zero in $\Mcat_A$ if and only if it image under $- \otimes_A \Q(q)$ is zero.  Thus it also suffices to show that a relation is respected by $\bFp_q$.

    It is well known that the $R$-matrix yields a braiding on the category $\Mcat_q$.  (See, for example, \cite[Ch.~32]{Lus10}.)  The relations \cref{venom} then follow immediately.  The last two equalities in \cref{vortex} also follow from this braided structure.  Indeed, it follows from the property of a braided monoidal category that
    \[
        \bFp_q \left(
            \begin{tikzpicture}[anchorbase]
                \draw (-0.3,-0.4) to[out=up,in=down] (0,0) arc(180:0:0.15) to[out=down,in=up] (0,-0.4);
                \draw[wipe] (0.3,-0.4) to[out=up,in=down] (-0.3,0.4);
                \draw (0.3,-0.4) to[out=up,in=down] (-0.3,0.4);
            \end{tikzpicture}
        \right)
        =
        \bFp_q \left(
            \begin{tikzpicture}[anchorbase]
                \draw (-0.3,-0.4) -- (-0.3,-0.2) arc(180:0:0.15) -- (0,-0.4);
                \draw (0.3,-0.4) -- (0.3,0.4);
            \end{tikzpicture}
        \right).
    \]
    Composing on the bottom with $\bFp_q \left( \idstrand\ \poscross \right)$ then shows that the penultimate equality in \cref{vortex} holds.  The verification of the last equality in \cref{vortex} is similar.  In fact, an analogous argument shows directly that the first two relations in \cref{turvy} are also respected by $\bFp_q$.

    The first relation in \cref{chess} follows from \cref{curly,cider}.  For the fourth relation in \cref{chess}, we compute
    \begin{multline*}
        \bFp_q\left( \bubble \right)
        = \sD_q \circ \sC_q \colon \Q(q) \to \Q(q),\quad
        1 \mapsto \sum_{v \in \bB_\sV} v \otimes v^\vee
        \mapsto \sum_{v \in \bB_\sV} \sD_q(v \otimes v^\vee)
        \\
        \overset{\cref{trolloc}}{=} \sum_{v \in \bB_\sV} \sD_q(v^\vee \otimes K_{-2\rho}(v))
        = \qdim \sV_q
        = \delta.
    \end{multline*}
    The fifth relation in \cref{chess} follows from the fact that there are no nonzero $U_q$-module homomorphisms from the trivial module to $\sV_q$ by Schur's lemma.

    Next we show that the second equality in \cref{chess} is preserved.  Since $\Hom_{U_q}(\sV_q^{\otimes 2}, \sV_q)$ is one-dimensional, we have
    \[
        \bFp_q
        \left(
            \begin{tikzpicture}[anchorbase]
                \draw (0.2,-0.5) to [out=135,in=down] (-0.15,-0.2) to[out=up,in=-135] (0,0);
                \draw[wipe] (-0.2,-0.5) to[out=45,in=down] (0.15,-0.2);
                \draw (-0.2,-0.5) to[out=45,in=down] (0.15,-0.2) to[out=up,in=-45] (0,0) -- (0,0.2);
            \end{tikzpicture}
        \right)
        = \tau\, \bFp_q \left( \mergemor \right)
    \]
    for some scalar $\tau$.  Then \cref{forked} implies that $\tau^2 = q^{24}$, and so $\tau = \pm q^{12}$.  Now, since $\Hom_{U_\C}(\sV_\C^{\otimes 2}, \sV_\C)$ is also one-dimensional, $\bFp_\C(\mergemor)$ must be a nonzero multiple of $\Phi(\mergemor)$ in the notation of \cite[Th.~5.1]{GSZ21}.  Since $\sR_q$ specializes to $\flip$ at $q=1$, it then follows from \cite[Th.~5.1]{GSZ21} that
    \[
        \bFp_\C
        \left(
            \begin{tikzpicture}[anchorbase]
                \draw (0.2,-0.5) to [out=135,in=down] (-0.15,-0.2) to[out=up,in=-135] (0,0);
                \draw[wipe] (-0.2,-0.5) to[out=45,in=down] (0.15,-0.2);
                \draw (-0.2,-0.5) to[out=45,in=down] (0.15,-0.2) to[out=up,in=-45] (0,0) -- (0,0.2);
            \end{tikzpicture}
        \right)
        = \bFp_\C \left( \mergemor \right).
    \]
    Hence $\tau = 1$ when $q=1$, and so $\tau = q^{12}$, as desired.

    The first two equalities in \cref{vortex} follow from \cref{snail}.  Next we will show that the fourth equality in \cref{vortex} holds.  Since we have already shown that the first two equalities in \cref{vortex} are respected by $\bFp_q$, we have
    \[
        \bFp_q \left(
            \begin{tikzpicture}[anchorbase]
                \draw (-0.7,-0.3) -- (-0.7,0) arc(180:0:0.15) to[out=down,in=180] (-0.2,-0.2) to[out=0,in=225] (0,0);
                \draw (0,0) -- (0,0.2);
                \draw (0.3,-0.3) -- (0,0);
            \end{tikzpicture}
        \right)
        =
        \bFp_q \left( \mergemor \right)
        =
        \bFp_q \left(
            \begin{tikzpicture}[anchorbase]
                \draw (0.7,-0.3) -- (0.7,0) arc(0:180:0.15) to[out=down,in=0] (0.2,-0.2) to[out=180,in=-45] (0,0);
                \draw (0,0) -- (0,0.2);
                \draw (-0.3,-0.3) -- (0,0);
            \end{tikzpicture}
        \right).
    \]
    It follows that
    \begin{equation} \label{mugs}
        \bFp_q \left(
            \begin{tikzpicture}[anchorbase]
                \draw (-0.4,0.2) to[out=down,in=180] (-0.2,-0.2) to[out=0,in=225] (0,0);
                \draw (0,0) -- (0,0.2);
                \draw (0.3,-0.3) -- (0,0);
            \end{tikzpicture}
        \right)
        \qquad \text{and} \qquad
        \bFp_q
        \left(
            \begin{tikzpicture}[anchorbase]
                \draw (0.4,0.2) to[out=down,in=0] (0.2,-0.2) to[out=180,in=-45] (0,0);
                \draw (0,0) -- (0,0.2);
                \draw (-0.3,-0.3) -- (0,0);
            \end{tikzpicture}
        \right)
    \end{equation}
    are both nonzero.  Thus, since $\Hom_{U_q}(\sV_q, \sV_q^{\otimes 2})$ is one-dimensional, there exists $c \in \Q(q)$ such that
    \[
        \bFp_q \left(
            \begin{tikzpicture}[anchorbase]
                \draw (-0.4,0.2) to[out=down,in=180] (-0.2,-0.2) to[out=0,in=225] (0,0);
                \draw (0,0) -- (0,0.2);
                \draw (0.3,-0.3) -- (0,0);
            \end{tikzpicture}
        \right)
        = c
        \bFp_q
        \left(
            \begin{tikzpicture}[anchorbase]
                \draw (0.4,0.2) to[out=down,in=0] (0.2,-0.2) to[out=180,in=-45] (0,0);
                \draw (0,0) -- (0,0.2);
                \draw (-0.3,-0.3) -- (0,0);
            \end{tikzpicture}
        \right).
    \]
    It remains to show that $c=1$.  Since we know that composition induces a nonzero linear map
    \[
        \Hom_{U_q}(\sV_q^{\otimes 2}, \sV_q) \otimes \Hom_{U_q}(\sV_q \otimes \sV_q^{\otimes 2})
        \xrightarrow{\cong} \End_{U_q}(\sV_q) \cong \Q(q),
    \]
    it follows that the composition of each of the maps in \cref{mugs} with $\sT$ is nonzero.  Hence it suffices to show that these two compositions are equal.  Now, a quick look at the proof of \cref{windy} shows that we have shown enough relations are respected by $\bFp_q$ to know that \cref{turvy,pokey} are respected.  Similarly, we have demonstrated enough relations (namely, the first, second, and last two equalities in \cref{vortex}, and all of the equalities in \cref{turvy}) to know that rotated versions of the first equality in \cref{chess} and the fourth equality in \cref{turvy} are respected by $\bFp_q$.   In particular, $\bFp_q$ respects all of the relations used in the following string of equalities in $\Fcat$:
    \[
        \begin{tikzpicture}[anchorbase]
            \draw (-0.2,0.5) -- (-0.2,0.3) -- (0,0) -- (0.2,-0.5);
            \draw (0,0) to[out=235,in=0] (-0.2,-0.3) to[out=180,in=225] (-0.2,0.3);
        \end{tikzpicture}
        \overset{\cref{chess}}{=} q^{-24}
        \begin{tikzpicture}[anchorbase]
            \draw (-0.3,-0.3) to[out=-90,in=-90,looseness=1.5] (-0.1,-0.3) to[out=90,in=225] (-0.2,0.3);
            \draw[wipe] (0,0) to[out=235,in=90] (-0.3,-0.3);
            \draw (0,0) to[out=235,in=90] (-0.3,-0.3);
            \draw (-0.2,0.5) -- (-0.2,0.3) -- (0,0) -- (0.2,-0.5);
        \end{tikzpicture}
        \overset{\cref{turvy}}{=} q^{-12}
        \begin{tikzpicture}[anchorbase]
            \draw (-0.3,-0.3) to[out=-90,in=-90,looseness=1.5] (-0.1,-0.3) to[out=90,in=-45] (-0.2,0.3);
            \draw[wipe] (0.1,0) to[out=235,in=90] (-0.3,-0.3);
            \draw[wipe] (0.1,0) to[out=150,in=225,looseness=1.5] (-0.2,0.3);
            \draw (0.1,0) to[out=235,in=90] (-0.3,-0.3);
            \draw (0.1,0) to[out=150,in=225,looseness=1.5] (-0.2,0.3);
            \draw (-0.2,0.3) -- (-0.2,0.5);
            \draw (0.1,0) -- (0.2,-0.5);
        \end{tikzpicture}
        \overset{\cref{pokey}}{=} q^{-12}\,
        \begin{tikzpicture}[anchorbase]
            \draw (-0.1,0) to[out=150,in=225,looseness=1.5] (-0.2,0.3);
            \draw (-0.2,0.3) -- (-0.2,0.5);
            \draw (-0.1,0) to[out=-135,in=180,looseness=1.5] (-0.1,-0.4) to[out=0,in=-90] (0.1,0) to[out=90,in=-45] (-0.2,0.3);
            \draw[wipe] (-0.025,-0.125) -- (0.2,-0.5);
            \draw (-0.1,0) -- (0.2,-0.5);
        \end{tikzpicture}
        \overset{\cref{turvy}}{\underset{\cref{chess}}{=}}
        \begin{tikzpicture}[anchorbase]
            \draw (0.2,0.5) -- (0.2,0.3) -- (0,0) -- (-0.2,-0.5);
            \draw (0,0) to[out=-55,in=180] (0.2,-0.3) to[out=0,in=-45] (0.2,0.3);
        \end{tikzpicture}
        \ ,
    \]
    where, in the third equality, we used the relation
    \[
        \begin{tikzpicture}[anchorbase]
            \draw (0.2,0.3) to[out=225,in=90] (-0.2,0) to[out=-90,in=135] (0,-0.3);
            \draw[wipe] (0,0.3) -- (0,0) -- (-0.3,-0.3);
            \draw (0,0.3) -- (0,0) -- (-0.3,-0.3);
            \draw (0,0) -- (0.3,-0.3);
        \end{tikzpicture}
        =
        \begin{tikzpicture}[anchorbase]
            \draw (0.2,0.3) to[out=-90,in=45] (0,-0.3);
            \draw (0,0.3) -- (0,0) -- (-0.3,-0.3);
            \draw[wipe] (0,0) -- (0.3,-0.3);
            \draw (0,0) -- (0.3,-0.3);
        \end{tikzpicture}
        \ ,
    \]
    which follows from composing both sides of the second relation in \cref{pokey} on the bottom with $\poscross\ \idstrand$\ .  Thus $c=1$ as desired.

    It remains to verify the third relation in \cref{chess}, and this is where we will fix the scaling of $\sD_\bullet$.  Since $\Hom_{U_A}(\sV_A,\sV_A)$ is free of rank one, spanned by $1_{\sV_q}$, we have
    \[
        \bFp_A
        \left(
            \begin{tikzpicture}[centerzero]
                \draw  (0,-0.4) -- (0,-0.2) to[out=45,in=down] (0.15,0) to[out=up,in=-45] (0,0.2) -- (0,0.4);
                \draw (0,-0.2) to[out=135,in=down] (-0.15,0) to[out=up,in=-135] (0,0.2);
            \end{tikzpicture}
        \right)
        = \phi' 1_{\sV_A}
        = \phi' \bFp_A \left(\, \idstrand\, \right)
    \]
    for some $\phi' \in A$.  Now, the composition
    \[
        \Hom_{U_\C}(\sV_\C, \sV_\C^{\otimes 2}) \otimes_\C \Hom_{U_\C}(\sV_\C^{\otimes 2}, \sV_\C)
        \to \Hom_{U_\C}(\sV_\C, \sV_\C)
    \]
    is an isomorphism, since all the hom-spaces involved are one-dimensional.  Thus, $\phi'$ must be nonzero at $q=1$, and so $\phi' \in A^\times$.  We now scale $\sD_\bullet$ by $\phi'/\phi$.  This scales $\sC_\bullet$ by $\phi/\phi'$, which, by the definition of $\splitmor$ in \cref{vortex}, also scales $\bFp_\bullet \left( \splitmor \right)$ by $\phi/\phi'$.  Then the third relation in \cref{chess} is satisfied.
\end{proof}

\begin{rem} \label{degenerate4}
    In light of \cref{degenerate3}, \cite[Th.~5.1]{GSZ21} gives a monoidal functor
    \[
        \Phi \colon \Tcat_\C \to \Mcat_\C.
    \]
    As noted in the proof of \cref{groundwork}, we have $\bFp_\C(\poscross) = \flip = \Phi(\crossmor)$.  Since
    \[
        \dim \Hom_{U_\C}(\sV_\C^{\otimes 2}, \sV_\C)
        = \dim \Hom_{U_\C}(\C,\sV_\C^{\otimes 2})
        = \dim \Hom_{U_\C}(\sV_\C^{\otimes 2}, \C)
        = 1,
    \]
    we also see that $\bFp_\C(\mergemor)$, $\bFp_\C(\cupmor)$, and $\bFp_\C(\capmor)$ are equal to nonzero scalar multiples of $\Phi(\mergemor)$, $\Phi(\cupmor)$, and $\Phi(\capmor)$, respectively.  It follows that the images of $\bFp_\C$ and $\Phi$ are equal.  In fact, we could scale the choice of $v_{-\lambda}$ made before \cref{skates} and the choice of $v_\sT$ made before \cref{lobster}
    by elements of $\C^\times$ to set these scalars equal to $1$, so that $\bFp_\C = \Phi$.
\end{rem}

\begin{prop} \label{window}
    The images under $\bFp_q$ of
    \begin{equation} \label{4pointers}
        \jail\ ,\quad \hourglass\ ,\quad \Imor\ ,\quad \Hmor\ ,\quad \poscross + \negcross
    \end{equation}
    form a basis of $\Hom_{U_q}(\sV_q^{\otimes 2}, \sV_q^{\otimes 2})$ and the images under $\bFp_q$ of the last 10 morphisms in \cref{5pointers} are linearly independent in $\Hom_{U_q}(\sV_q^{\otimes 2}, \sV_q^{\otimes 3})$.
\end{prop}

\begin{proof}
    We have
    \[
        \dim_{\Q(q)} \Hom_{U_q}(\sV_q^{\otimes 2}, \sV_q^{\otimes 2}) = 5
        \qquad \text{and} \qquad
        \dim_{\Q(q)} \Hom_{U_q}(\sV_q^{\otimes 2}, \sV_q^{\otimes 3}) = 15.
    \]
    (See, for example, \cite[(4.14)]{GSZ21} and use \cref{business}.)  Therefore it suffices to prove that the morphisms \cref{4pointers} are linearly independent, and that the last 10 morphisms in \cref{5pointers} are linearly independent.  Any linear dependence relation in $\Hom_{U_q}(\sV_q^{\otimes 2}, \sV_q^{\otimes 2})$ gives, after clearing denominators if necessary, a linear dependence relation in $\Hom_{U_A}(\sV_A^{\otimes 2}, \sV_A^{\otimes 2})$.  Hence, it is enough to prove that the images under $\bFp_A$ of the diagrams \cref{4pointers} are linearly independent.  But this follows immediately from the fact that they are linearly independent at $q=1$ (that is, after passing to $\Hom_{U_\C}(\sV_\C^{\otimes 2}, \sV_\C^{\otimes 2})$ by tensoring over $A$ with $\C$) by \cite[Prop.~3.2]{GSZ21}.  The argument for the last 10 morphisms in \cref{5pointers} is analogous, using the fact that they are also linearly independent at $q=1$, as shown in the proof of \cite[Lem.~3.7]{GSZ21}.
\end{proof}

\begin{prop} \label{fillerup}
    The functor $\bFp_q$ of \cref{groundwork} is full.
\end{prop}

\begin{proof}
    Fix $m,n \in \N$.  The functor $\bFp_\C$ is full by \cite[Prop.~5.2]{GSZ21} and \cref{degenerate4}.  Since the diagram \cref{teatime} commutes, this implies that there is a set $X$ of morphism in $\Tcat_A$ such that $\bFp_A(f) \otimes_A \C$, $f \in X$, are a basis for $\Mcat_\C(\sV_\C^{\otimes m}, \sV_\C^{\otimes n})$.  Thus the morphisms $\bFp_A(f)$, $f \in X$, are linearly independent in $\Mcat_A(\sV_A^{\otimes m}, \sV_A^{\otimes n})$.  Hence the morphisms $\bFp_q(f) = \bFp_A(f) \otimes_A \Q(q)$, $f \in X$, are linearly independent in $\Mcat_q(\sV_A^{\otimes m}, \sV_A^{\otimes n})$.  Thus, by \cref{business}, they are a basis.
\end{proof}

\begin{theo} \label{lightning}
    The functor $\bFp_q$ of \cref{groundwork} factors through $\Fcat_q$ to give a full functor
    \[
        \bF_q \colon \Fcat_q \to \Mcat_q.
    \]
\end{theo}

\begin{proof}
    If $\bFp_q$ factors through $\Fcat_q$, then the induced functor $\bF_q$ is full by \cref{fillerup}.  Let $\cI$ be the kernel of the functor $\bFp_q$ of \cref{groundwork}.  We must show that relations \cref{symskein,qnuke,qsquare,qpent} are satisfied in $\Tcat_q/\cI$.  It follows from \cref{window} that \cref{expanse} and the hypotheses of \cref{protomolecule} are satisfied.  Then
    \begin{itemize}
        \item \cref{symskein} holds in $\Tcat_q/\cI$ by \cref{skeinexist,zbalance};
        \item \cref{qnuke} holds by \cref{bull};
        \item \cref{qsquare} holds by \cref{bagel};
        \item \cref{qpent} holds by \cref{protomolecule}. \qedhere
    \end{itemize}
\end{proof}

Since the category $\Mcat_q$ is idempotent complete, $\bF_q$ induces a functor
\[
    \Kar (\bF_q) \colon \Kar(\Fcat_q) \to \Mcat_q,
\]
where $\Kar$ denotes the additive Karoubi envelope (that is, the idempotent completion of the additive envelope).

\begin{theo} \label{westboro}
    The functor $\Kar(\bF_q)$ is full and essentially surjective.
\end{theo}

\begin{proof}
    Fullness of $\Kar(\bF_q)$ follows from the fact that $\bF_q$ is full, so it remains to show that $\Kar(\bF_q)$ is essentially surjective.  This argument is the same as the one in the proof of \cite[Prop.~5.5]{GSZ21}.
\end{proof}

Unfortunately, we do not know if the functor $\bF_q$ is faithful.  As we now explain, a key next step in this direction would be to determine if the defining relations of $\Fcat_q$ allow us to reduce any closed diagram to a scalar multiple of the empty diagram.

We refer the reader to \cite{Sel11} for an overview of various properties of monoidal categories and their diagrammatic interpretations.  In particular, see \cite[\S 4.4.3]{Sel11} for the definition of spherical pivotal category.  Since $\Fcat_q$ is a ribbon category, it is spherical pivotal.  In any spherical pivotal category $\cC$, we have a trace map $\Tr \colon \bigoplus_{X \in \cC} \End_\cC(X) \to \End_\cC(\one)$.  In terms of string diagrams, this corresponds to closing a diagram off to the right or left:
\begin{equation} \label{natal}
    \Tr
    \left(
        \begin{tikzpicture}[centerzero]
            \draw[line width=2] (0,-0.5) -- (0,0.5);
            \filldraw[fill=white,draw=black] (-0.25,0.2) rectangle (0.25,-0.2);
            \node at (0,0) {$\scriptstyle{f}$};
        \end{tikzpicture}
    \right)
    =
    \begin{tikzpicture}[centerzero]
        \draw[line width=2] (0,0.2) arc(180:0:0.3) -- (0.6,-0.2) arc(360:180:0.3);
        \filldraw[fill=white,draw=black] (-0.25,0.2) rectangle (0.25,-0.2);
        \node at (0,0) {$\scriptstyle{f}$};
    \end{tikzpicture}
    =
    \begin{tikzpicture}[centerzero]
        \draw[line width=2] (0,0.2) arc(0:180:0.3) -- (-0.6,-0.2) arc(180:360:0.3);
        \filldraw[fill=white,draw=black] (-0.25,0.2) rectangle (0.25,-0.2);
        \node at (0,0) {$\scriptstyle{f}$};
    \end{tikzpicture}
    \ ,
\end{equation}
where the second equality follows from the axioms of a spherical category.  We say that a morphism $f \in \Hom_\cC(X,Y)$ is \emph{negligible} if $\Tr(f \circ g) = 0$ for all $g \in \Hom_\cC(Y,X)$.  The negligible morphisms form a two-sided tensor ideal $\cN$ of $\cC$, and the quotient $\cC/\cN$ is called the \emph{semisimplification} of $\cC$.

\begin{prop} \label{graviola}
    Suppose $\kk$ is a field and $\cC$ is a $\kk$-linear spherical pivotal category such that $\End_\cC(\one) = \kk 1_\one \ne 0$.
    \begin{enumerate}
        \item The negligible morphisms form the unique maximal tensor ideal of $\cC$.

        \item If $\cD$ is a nonzero semisimple $\kk$-linear monoidal category, then the kernel of any full monoidal functor $\cC \to \cD$ is equal to the tensor ideal of negligible morphisms.
    \end{enumerate}
\end{prop}

\begin{proof}
    \begin{enumerate}[wide]
        \item First note that a tensor ideal of $\cC$ is equal to $\cC$ if and only if it contains $1_\one$.  Furthermore, the tensor ideal of negligible morphisms is not equal to $\cC$ since $\Tr(1_\one) = 1_\one \ne 0$ by our assumption on $\cC$.

            Now suppose $f \colon X \to Y$ is a morphism in $\cC$ that is not negligible.  Then there exists a morphism $g \colon Y \to X$ in $\cC$ such that $\Tr(f \circ g) = c 1_\one$ for some $c \in \kk^\times$.  Then $1_\one = c^{-1} \Tr(f \circ g)$ is contained in the tensor ideal generated by $f$ and so this tensor ideal is all of $\cC$.  It follows that the negligible morphisms form the unique maximal tensor ideal of $\cC$.

        \item This follows from the fact that any full monoidal functor sends negligible morphisms to negligible morphisms, and that a semisimple category has no nonzero negligible morphisms. \qedhere
    \end{enumerate}
\end{proof}

\begin{cor} \label{acai}
    If $\End_{\Fcat_q}(\one) = \C(q) 1_\one$, then $\ker(\bF_q)$ is equal to the tensor ideal of negligible morphisms, and the functor $\bF_q$ induces an equivalence of categories from the semisimplification of $\cC$ to $\Mcat_q$.
\end{cor}

\begin{proof}
    Since $\Mcat_q$ is a semisimple category, this follows from \cref{westboro,graviola}.
\end{proof}

We conclude this section with a conjecture concerning a basis for the morphism spaces of $\Fcat_q$.  A string diagram built from the generating morphisms \cref{lego} via tensor product and composition can be viewed as a graph embedded in the cube $[0,1]^3$, whose univalent vertices occur at points in $[0,1] \times \{0,1\} \times \{1\}$.  Here we view $\mergemor$ as a trivalent vertex and $\poscross$, $\negcross$ as two edges, one of which passes under the other.  If such a graph corresponds to a morphism in $\Fcat_q(\go^{\otimes m},\go^{\otimes n})$, then it has $m$ univalent vertices at points in $[0,1] \times \{0\} \times \{1\}$ and $n$ univalent vertices at points in $[0,1] \times \{1\} \times \{1\}$.  We number the univalent vertices of the diagram clockwise, starting at the leftmost top vertex.  We then say that the diagram is \emph{normally-stacked component-planar} if it satisfies the following conditions:
\begin{itemize}
    \item All connected components have at least two univalent vertices.
    \item All connected components are planar; that is, they contain no crossings $\poscross$ or $\negcross$.
    \item Let $C_1,\dotsc,C_r$ be the connected components, ordered so that the lowest univalent vertex in $C_i$ is less than the lowest univalent vertex in $C_j$ for $i<j$.  Then, for $i<j$, all edges in $C_i$ pass under all edges in $C_j$.
\end{itemize}
For example,
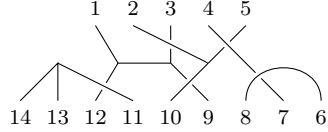
\[
    \begin{tikzpicture}[anchorbase]
        \draw (1,0) node[anchor=north] {$\scriptstyle 12$} -- (1.3,0.5) -- (1,1) node[anchor=south] {$\scriptstyle 1$};
        \draw (1.3,0.5) -- (2,0.5);
        \draw (2.5,0) node[anchor=north] {$\scriptstyle 9$} -- (2,0.5) -- (2,1) node[anchor=south] {$\scriptstyle 3$};
        \draw[wipe] (2,0) -- (2.5,0.5) -- (1.5,1);
        \draw (2,0) node[anchor=north] {$\scriptstyle 10$} -- (2.5,0.5) -- (1.5,1) node[anchor=south] {$\scriptstyle 2$};
        \draw[wipe] (2.5,0.5) -- (3,1);
        \draw (2.5,0.5) -- (3,1) node[anchor=south] {$\scriptstyle 5$};
        \draw[wipe] (3.5,0) -- (2.5,1);
        \draw (3.5,0) node[anchor=north] {$\scriptstyle 7$} -- (2.5,1) node[anchor=south] {$\scriptstyle 4$};
        \draw[wipe] (3,0) to[out=up,in=up,looseness=1.5] (4,0);
        \draw (3,0) node[anchor=north] {$\scriptstyle 8$} to[out=up,in=up,looseness=1.5] (4,0) node[anchor=north] {$\scriptstyle 6$};
        \draw[wipe] (0.5,0.5) -- (1.5,0);
        \draw (0.5,0.5) -- (1.5,0) node[anchor=north] {$\scriptstyle 11$};
        \draw (0,0) node[anchor=north] {$\scriptstyle 14$} -- (0.5,0.5) -- (0.5,0) node[anchor=north] {$\scriptstyle 13$};
    \end{tikzpicture}
\]
is a normally-stacked component-planar diagram, where we have indicated the numbering of the univalent vertices.

\begin{conj} \label{wish}
    For $m,n \in \N$, the normally-stacked component-planar graphs whose cycles are all of length at least six span $\Fcat_q(\go^{\otimes m}, \go^{\otimes n})$.
\end{conj}

\Cref{wish} is a quantum analogue of \cite[Conj.~5.9]{GSZ21}, which conjectures a spanning set in the degenerate (that is, $q=1$) setting.

\section{Idempotents in the 2-point endomorphism algebra\label{sec:idempotents}}

It follows from \cref{westboro} that every finite-dimensional $U_q$-module of type $1$ is the image of $\bF_q(e)$ for some idempotent endomorphism $e$ in $\Fcat_q$.  In this section we illustrate this by giving an explicit decomposition of $\sV_q^{\otimes 2}$.

Any decomposition $1_\go^{\otimes n} = \sum_{i=1}^m e_i$ of the identity of $\Fcat_q(\go^{\otimes n}, \go^{\otimes n})$ as a sum of orthogonal idempotents induces a decomposition
\[
    1_{\sV_q^{\otimes n}} = \sum_{i=1}^m \bF_q(e_i)
\]
of the identity endomorphism of $\sV_q^{\otimes n}$ as a sum of orthogonal idempotents.  This, in turn, yields a direct sum decomposition
\[
    \sV_q^{\otimes n} = \bigoplus_{i=1}^m \bF_q(e_i) \sV_q^{\otimes n}.
\]
The idempotent $\bF_q(e_i)$ is the identity map on the summand $\bF_q(e_i) \sV_q^{\otimes n}$, and so we have
\[
    \tr_q \left( \bF_q(e_i) \right) = \qdim \left( \bF_q(e_i) \sV_q^{\otimes n} \right).
\]
The following lemma, whose proof is a generalization of an argument found in the proof of \cref{groundwork}, implies that we can compute this quantum dimension by taking the trace, as in \cref{natal}, of the diagrams occurring in $e_i$.

\begin{lem} \label{lounge}
    For an endomorphism $f \in \Fcat_q(\go^{\otimes n},\go^{\otimes n})$, we have
    \begin{equation} \label{chaise}
        \bF_q \left(
            \begin{tikzpicture}[centerzero]
                \draw[line width=2] (0,0.2) arc(180:0:0.3) -- (0.6,-0.2) arc(360:180:0.3);
                \filldraw[fill=white,draw=black] (-0.25,0.2) rectangle (0.25,-0.2);
                \node at (0,0) {$\scriptstyle{f}$};
            \end{tikzpicture}
        \right)
        = \tr_q(\bF_q(f)),
    \end{equation}
    where the thick strand represents $1_\go^{\otimes n}$ and, on the right-hand side, we are identifying $a \in \Q(q)$ with the element of $\End_{U_q}(L_q(0))$ given by multiplication by $a$.
\end{lem}

\begin{proof}
    Define
    \[
        \sD_q^{(n)} \colon \sV_q^{\otimes 2n} \to \Q(q),\qquad
        v_1 \otimes \dotsb \otimes v_n \otimes w_n \otimes \dotsb \otimes w_1
        \mapsto \sD_q(v_1 \otimes w_1) \dotsm \sD_q(v_n \otimes w_n).
    \]
    Then the left-hand side of \cref{chaise} is the map
    \begin{align*}
        1 &\mapsto \sum_{v_1,\dotsc,v_n \in \bB_\sV} \sD_q^{(n)} \left( \bF_q(f)(v_1 \otimes \dotsb \otimes v_n) \otimes v_n^\vee \otimes \dotsb \otimes v_1^\vee \right)
        \\
        &\overset{\mathclap{\cref{trolloc}}}{=}\
        \sum_{v_1,\dotsc,v_n \in \bB_\sV} \sD_q^{(n)} \left( v_n^\vee \otimes \dotsb \otimes v_1^\vee \otimes K_{-2\rho}^{\otimes n} \bF_q(f)(v_1 \otimes \dotsb \otimes v_n) \right)
        \\
        &\overset{\mathclap{\cref{HopfK}}}{=}\
        \sum_{v_1,\dotsc,v_n \in \bB_\sV} \sD_q^{(n)} \left( v_n^\vee \otimes \dotsb \otimes v_1^\vee \otimes (\bF_q(f) \circ K_{-2\rho})(v_1 \otimes \dotsb \otimes v_n) \right)
        \\
        &= \tr_q(\bF_q(f)).
        \qedhere
    \end{align*}
\end{proof}

It follows from the decomposition \cref{lizard} that the identity morphism of $\sV_q^{\otimes 2}$ is a sum of five orthogonal idempotents, corresponding to projection onto the five summands on the right-hand side of \cref{lizard}.  We can describe these idempotents explicitly.  Let
\begin{gather*}
    e_0 = \frac{1}{\delta}\, \hourglass,\qquad
    e_{\omega_1} = \frac{q^{-4}[4]}{[3]^2[8]}\, \jail
        + \frac{(q^{10}-q^4-q^2)[4][6][9]}{[3]^3[8][18]}\, \hourglass
        + \frac{q^{-1}[4]^2}{[3]^2[8]} \Hmor
        + \frac{q^2[4]^2}{[2][3]^2[8]} \Imor
        - \frac{[4]}{[3][8]} \poscross,
    \\
    e_{\omega_3} = \frac{q^{-1}[4]}{[3]^2}\, \jail
        + \frac{q^8[4]}{[3]^2[8]}\, \hourglass
        - \frac{q^2[4]}{[2][3]^2} \Hmor
        + \frac{(q^8-q^4-1)[4][6]}{[2][3]^2[12]} \Imor
        - \frac{1}{[3]} \poscross,\qquad
    e_{\omega_4} = \frac{1}{\phi}\, \Imor,
    \\
    e_{2\omega_4} = \frac{(q^6+q^2+q^{-4})[4]}{[3][8]}\, \jail
        - \frac{(q^{18}+q^{14}+q^{12}+q^8-1)[4]}{[3][8][13]}\, \hourglass
        + \frac{(q^4-q^2)[4]^2}{[2][3][8]} \Hmor
        \\ \qquad \qquad
        - \frac{(q^{10}+q^4+1)[4]^2}{[2][3][7][8]} \Imor
        + \frac{[4][6]}{[2][3][8]} \poscross.
\end{gather*}

\begin{prop}
    For $\lambda \in \{0,\omega_1,\omega_3,\omega_4,2\omega_4\}$, the morphism $\bF_q(e_\lambda) \in \End_{U_q}(\sV_q^{\otimes 2})$ corresponds to projection onto the summand $L_q(\lambda)$ in the decomposition \cref{lizard}.
\end{prop}

\begin{proof}
    Direct computation using SageMath shows that, in $\Fcat_q(\go^{\otimes 2}, \go^{\otimes 2})$,
    \[
        1_\go^{\otimes 2} = e_0 + e_{\omega_1} + e_{\omega_3} + e_{\omega_4} + e_{2 \omega_4}
    \]
    is a decomposition of $1_\go^{\otimes 2}$ as a sum of orthogonal idempotents.  Thus
    \[
        1_{\sV_q^{\otimes 2}} = \bF_q(e_0) + \bF_q(e_{\omega_1}) + \bF_q(e_{\omega_3}) + \bF_q(e_{\omega_4}) + \bF_q(e_{2 \omega_4})
    \]
    is a decomposition into orthogonal idempotents.  It thus suffices to prove that the quantum trace of $\bF_q(e_\lambda)$ is the quantum dimension of $L_q(\lambda)$ for $\lambda \in \{0,\omega_1,\omega_3,\omega_4,2\omega_4\}$.  To do this, we first compute
    \[
        \begin{tikzpicture}[centerzero]
            \draw (-0.2,0) arc(180:-180:0.2);
            \draw (-0.4,0) arc(180:-180:0.4);
        \end{tikzpicture}
        = \delta^2,
        \qquad
        \begin{tikzpicture}[centerzero]
            \draw (-0.1,0.2) arc(180:360:0.1) arc(180:0:0.1) -- (0.3,-0.2) arc(360:180:0.1) arc(0:180:0.1) arc(180:360:0.3) -- (0.5,0.2) arc(0:180:0.3);
        \end{tikzpicture}
        = \delta,
        \qquad
        \begin{tikzpicture}[centerzero]
            \draw (-0.15,0) -- (0.15,0) arc(180:-180:0.15);
            \draw (-0.15,0) arc(180:-180:0.45);
        \end{tikzpicture}
        = 0,
        \qquad
        \begin{tikzpicture}[centerzero]
            \draw (0,-0.1) -- (0,0.1) -- (0.1,0.2) to[out=45,in=90] (0.4,0) to[out=-90,in=-45] (0.1,-0.2) -- (0,-0.1);
            \draw (0,0.1) -- (-0.1,0.2) to[out=135,in=90,looseness=2] (0.6,0) to[out=-90,in=-135,looseness=2] (-0.1,-0.2) -- (0,-0.1);
        \end{tikzpicture}
        = \delta \phi,
        \qquad
        \begin{tikzpicture}[centerzero]
            \draw (0.4,0.3) arc(90:-90:0.3) to[out=left,in=right] (-0.4,0.3) arc(90:270:0.3);
            \draw[wipe] (-0.4,-0.3) to[out=right,in=left] (0.4,0.3);
            \draw (-0.4,-0.3) to[out=right,in=left] (0.4,0.3);
        \end{tikzpicture}
        = q^{-24} \delta.
    \]
    Then, using \cref{lounge}, direct computation using SageMath verifies that
    \[
        \tr_q \left( \bF_q(e_\lambda) \right)
        = \qdim L_q(\lambda)
        \qquad \text{for } \lambda \in \{0,\omega_1,\omega_3,\omega_4,2\omega_4\},
    \]
    as desired.
\end{proof}

\section{Alternative presentations}

In our treatment of the category $\Fcat_q$, we have made some choices.  We have essentially given preference to the morphisms
\[
    \jail\ ,\qquad \hourglass\ ,\qquad \Hmor\ ,\qquad \Imor\ ,\qquad \poscross\ .
\]
For example, the relation \cref{qsquare} expresses the square in terms of these five morphisms.  However, for some purposes, it is preferable to replace $\poscross$ by the rotationally invariant morphism $\poscross + \negcross$.  We did this, for example, in \cref{expanse}.  Then one can verify with SageMath that relation \cref{qsquare} becomes
\begin{equation}
    \begin{multlined}
        2 \sqmor
        = \frac{q^8 + q^4 + 2q^2 + 2 + 2q^{-2} + q^{-4} + q^{-8}}{(q^2+q^{-2})^2} \left(\, \jail\, +\, \hourglass \right)
        \\
        + (q^2 - 1 + q^{-2})(q^2 + 3 + q^{-2}) \left( \Imor + \Hmor \right)
        - \frac{[2][3][6]}{[4]^2} \left( \poscross + \negcross \right).
    \end{multlined}
\end{equation}

Yet another approach to the presentation of $\Fcat_q$ would be to eliminate crossings entirely, opting to use the square instead.  In particular, we can use \cref{qsquare}, and its rotation, to eliminate all crossings.  For example, we can replace the pentagon relation \cref{qpent} by the following result giving the pentagon in terms of simpler planar diagrams.

\begin{lem}
    We have
    \begin{equation} \label{planarpent}
        \begin{multlined}
            \pentmor =
            \frac{1}{[2]^2} \left(
                \begin{tikzpicture}[anchorbase]
                    \draw (-0.2,0) -- (0,0.25) -- (0.2,0);
                    \draw (0,0.25) -- (0,0.4);
                    \draw (-0.2,-0.25) -- (-0.2,0) -- (-0.3,0.4);
                    \draw (0.2,-0.25) -- (0.2,0) -- (0.3,0.4);
                \end{tikzpicture}
                +
                \begin{tikzpicture}[anchorbase]
                    \draw (-0.3,0.3) -- (0,0) -- (0.3,0.3);
                    \draw (0,0.3) -- (-0.15,0.15);
                    \draw (0,0) -- (0,-0.15) -- (-0.15,-0.3);
                    \draw (0,-0.15) -- (0.15,-0.3);
                \end{tikzpicture}
                +
                \begin{tikzpicture}[anchorbase]
                    \draw (0.3,0.3) -- (0,0) -- (-0.3,0.3);
                    \draw (0,0.3) -- (0.15,0.15);
                    \draw (0,0) -- (0,-0.15) -- (0.15,-0.3);
                    \draw (0,-0.15) -- (-0.15,-0.3);
                \end{tikzpicture}
                +
                \begin{tikzpicture}[anchorbase]
                    \draw (-0.3,-0.3) -- (0.3,0.3);
                    \draw (-0.3,0.3) -- (-0.15,-0.15);
                    \draw (0,0.3) -- (0.15,0.15);
                    \draw (0,0) -- (0.3,-0.3);
                \end{tikzpicture}
                +
                \begin{tikzpicture}[anchorbase]
                    \draw (0.3,-0.3) -- (-0.3,0.3);
                    \draw (0.3,0.3) -- (0.15,-0.15);
                    \draw (0,0.3) -- (-0.15,0.15);
                    \draw (0,0) -- (-0.3,-0.3);
                \end{tikzpicture}
            \right)
            + \frac{[2][3]^2}{[4]^2[6]}
            \left(
                \begin{tikzpicture}[centerzero]
                    \draw (-0.15,-0.3) -- (-0.15,-0.23) arc(180:0:0.15) -- (0.15,-0.3);
                    \draw (-0.3,0.3) -- (0,0.08) -- (0.3,0.3);
                    \draw (0,0.3) -- (0,0.08);
                \end{tikzpicture}
                +
                \begin{tikzpicture}[centerzero]
                    \draw (-0.2,-0.3) -- (-0.2,0.3);
                    \draw (0,0.3) -- (0.15,0) -- (0.3,0.3);
                    \draw (0.15,0) -- (0.15,-0.3);
                \end{tikzpicture}
                +
                \begin{tikzpicture}[centerzero]
                    \draw (0.2,-0.3) -- (0.2,0.3);
                    \draw (0,0.3) -- (-0.15,0) -- (-0.3,0.3);
                    \draw (-0.15,0) -- (-0.15,-0.3);
                \end{tikzpicture}
                +
                \begin{tikzpicture}[centerzero]
                    \draw (-0.3,0.3) -- (-0.3,0.23) arc(180:360:0.15) -- (0,0.3);
                    \draw (0.3,0.3) -- (0.15,0) -- (-0.2,-0.3);
                    \draw (0.2,-0.3) -- (0.15,0);
                \end{tikzpicture}
                +
                \begin{tikzpicture}[centerzero]
                    \draw (0.3,0.3) -- (0.3,0.23) arc(360:180:0.15) -- (0,0.3);
                    \draw (-0.3,0.3) -- (-0.15,0) -- (0.2,-0.3);
                    \draw (-0.2,-0.3) -- (-0.15,0);
                \end{tikzpicture}
            \right)
            \\
            - \frac{[3]}{[2][6]}
            \left(
                \begin{tikzpicture}[centerzero]
                    \draw (-0.15,-0.15) -- (-0.15,0.15) -- (0.15,0.15) -- (0.15,-0.15) -- cycle;
                    \draw (-0.15,0.15) -- (-0.15,0.25) -- (0,0.35);
                    \draw (-0.15,0.25) -- (-0.3,0.35);
                    \draw (0.15,0.15) -- (0.3,0.35);
                    \draw (-0.15,-0.15) -- (-0.2,-0.3);
                    \draw (0.15,-0.15) -- (0.2,-0.3);
                \end{tikzpicture}
                +
                \begin{tikzpicture}[centerzero]
                    \draw (-0.15,-0.15) -- (-0.15,0.15) -- (0.15,0.15) -- (0.15,-0.15) -- cycle;
                    \draw (0.15,0.15) -- (0.15,0.25) -- (0,0.35);
                    \draw (0.15,0.25) -- (0.3,0.35);
                    \draw (-0.15,0.15) -- (-0.3,0.35);
                    \draw (-0.15,-0.15) -- (-0.2,-0.3);
                    \draw (0.15,-0.15) -- (0.2,-0.3);
                \end{tikzpicture}
                +
                \begin{tikzpicture}[centerzero]
                    \draw (0,0.15) --  (0.15,0) -- (0,-0.15) -- (-0.15,0) -- cycle;
                    \draw (-0.15,0) -- (-0.2,0.3);
                    \draw (0,0.15) -- (0,0.3);
                    \draw (0.15,0) -- (0.25,0) -- (0.25,0.3);
                    \draw (0.25,0) -- (0.25,-0.3);
                    \draw (0,-0.15) -- (0,-0.3);
                \end{tikzpicture}
                +
                \begin{tikzpicture}[anchorbase]
                    \draw (0,0.15) --  (0.15,0) -- (0,-0.15) -- (-0.15,0) -- cycle;
                    \draw (0,-0.15) -- (0,-0.25) -- (0,-0.25) -- (-0.15,-0.4);
                    \draw (0,-0.25) -- (0.15,-0.4);
                    \draw (-0.15,0) -- (-0.3,0.25);
                    \draw (0,0.15) -- (0,0.25);
                    \draw (0.15,0) -- (0.3,0.25);
                \end{tikzpicture}
                +
                \begin{tikzpicture}[centerzero]
                    \draw (0,0.15) --  (-0.15,0) -- (0,-0.15) -- (0.15,0) -- cycle;
                    \draw (0.15,0) -- (0.2,0.3);
                    \draw (0,0.15) -- (0,0.3);
                    \draw (-0.15,0) -- (-0.25,0) -- (-0.25,0.3);
                    \draw (-0.25,0) -- (-0.25,-0.3);
                    \draw (0,-0.15) -- (0,-0.3);
                \end{tikzpicture}
            \right).
        \end{multlined}
    \end{equation}
\end{lem}

\begin{proof}
    Using \cref{qsquare}, we have
    \[
        \poscross = p_1\, \jail + p_2\, \hourglass + p_3\, \Hmor + p_4\, \Imor + p_5\, \sqmor,
    \]
    where
    \begin{gather*}
        p_1 = \frac{(q^{10} + q^8 + q^6 + q^4 + 1)[4]^2}{q^4(q^4+1)^2[2][3][6]},\quad
        p_2 = \frac{q^2(q^{10} + q^6 + q^4 + q^2 + 1)[4]^2}{(q^4+1)^2[2][3][6]},\\
        p_3 = \frac{(q^{10} + 2q^6 + q^2 + 1)[4]^2}{(q^8 + q^4)[2][3][6]},\quad
        p_4 = \frac{(q^{10} + q^8 + 2q^4 + 1)[4]^2}{(q^6 + q^2)[2][3][6]},\quad
        p_5 = - \frac{[4]^2}{[2][3][6]}.
    \end{gather*}
    Thus,
    \[
        \begin{tikzpicture}[centerzero]
            \draw (0,0.3) -- (0,-0.15) -- (-0.15,-0.3);
            \draw (0,-0.15) -- (0.15,-0.3);
            \draw[wipe] (-0.2,0.3) -- (-0.2,0.25) arc(180:360:0.2) -- (0.2,0.3);
            \draw (-0.2,0.3) -- (-0.2,0.25) arc(180:360:0.2) -- (0.2,0.3);
        \end{tikzpicture}
        = p_1
        \begin{tikzpicture}[centerzero]
            \draw (-0.3,0.3) -- (-0.3,0.23) arc(180:360:0.15) -- (0,0.3);
            \draw (0.3,0.3) -- (0.15,0) -- (-0.2,-0.3);
            \draw (0.2,-0.3) -- (0.15,0);
        \end{tikzpicture}
        + p_2
        \begin{tikzpicture}[centerzero]
            \draw (0.3,0.3) -- (0.3,0.23) arc(360:180:0.15) -- (0,0.3);
            \draw (-0.3,0.3) -- (-0.15,0) -- (0.2,-0.3);
            \draw (-0.2,-0.3) -- (-0.15,0);
        \end{tikzpicture}
        + p_3
        \begin{tikzpicture}[anchorbase]
            \draw (-0.3,0.3) -- (0,0) -- (0.3,0.3);
            \draw (0,0.3) -- (-0.15,0.15);
            \draw (0,0) -- (0,-0.15) -- (-0.15,-0.3);
            \draw (0,-0.15) -- (0.15,-0.3);
        \end{tikzpicture}
        + p_4
        \begin{tikzpicture}[anchorbase]
            \draw (0.3,0.3) -- (0,0) -- (-0.3,0.3);
            \draw (0,0.3) -- (0.15,0.15);
            \draw (0,0) -- (0,-0.15) -- (0.15,-0.3);
            \draw (0,-0.15) -- (-0.15,-0.3);
        \end{tikzpicture}
        + p_5 \sqmor.
    \]
    Summing over all rotations gives
    \begin{multline*}
        \begin{tikzpicture}[centerzero]
            \draw (0,0.3) -- (0,-0.15) -- (-0.15,-0.3);
            \draw (0,-0.15) -- (0.15,-0.3);
            \draw[wipe] (-0.2,0.3) -- (-0.2,0.25) arc(180:360:0.2) -- (0.2,0.3);
            \draw (-0.2,0.3) -- (-0.2,0.25) arc(180:360:0.2) -- (0.2,0.3);
        \end{tikzpicture}
        +
        \begin{tikzpicture}[centerzero]
            \draw (-0.3,0.3) -- (0,0) -- (0.3,0.3);
            \draw (0,0) -- (-0.15,-0.3);
            \draw[wipe] (0,0.3) to[out=-45,in=70] (0.15,-0.3);
            \draw (0,0.3) to[out=-45,in=70] (0.15,-0.3);
        \end{tikzpicture}
        +
        \begin{tikzpicture}[centerzero]
            \draw (0.3,0.3) -- (0,0) -- (-0.3,0.3);
            \draw (0,0) -- (0.15,-0.3);
            \draw[wipe] (0,0.3) to[out=225,in=110] (-0.15,-0.3);
            \draw (0,0.3) to[out=225,in=110] (-0.15,-0.3);
        \end{tikzpicture}
        +
        \begin{tikzpicture}[centerzero]
            \draw (-0.3,0.3) -- (-0.15,0.15) -- (0,0.3);
            \draw (-0.15,0.15) -- (0.15,-0.3);
            \draw[wipe] (0.3,0.3) -- (-0.15,-0.3);
            \draw (0.3,0.3) -- (-0.15,-0.3);
        \end{tikzpicture}
        +
        \begin{tikzpicture}[centerzero]
            \draw (0.3,0.3) -- (0.15,0.15) -- (0,0.3);
            \draw (0.15,0.15) -- (-0.15,-0.3);
            \draw[wipe] (-0.3,0.3) -- (0.15,-0.3);
            \draw (-0.3,0.3) -- (0.15,-0.3);
        \end{tikzpicture}
        =
        (p_3 + p_4)
        \left(
            \begin{tikzpicture}[anchorbase]
                \draw (-0.2,0) -- (0,0.25) -- (0.2,0);
                \draw (0,0.25) -- (0,0.4);
                \draw (-0.2,-0.25) -- (-0.2,0) -- (-0.3,0.4);
                \draw (0.2,-0.25) -- (0.2,0) -- (0.3,0.4);
            \end{tikzpicture}
            +
            \begin{tikzpicture}[anchorbase]
                \draw (-0.3,0.3) -- (0,0) -- (0.3,0.3);
                \draw (0,0.3) -- (-0.15,0.15);
                \draw (0,0) -- (0,-0.15) -- (-0.15,-0.3);
                \draw (0,-0.15) -- (0.15,-0.3);
            \end{tikzpicture}
            +
            \begin{tikzpicture}[anchorbase]
                \draw (0.3,0.3) -- (0,0) -- (-0.3,0.3);
                \draw (0,0.3) -- (0.15,0.15);
                \draw (0,0) -- (0,-0.15) -- (0.15,-0.3);
                \draw (0,-0.15) -- (-0.15,-0.3);
            \end{tikzpicture}
            +
            \begin{tikzpicture}[anchorbase]
                \draw (-0.3,-0.3) -- (0.3,0.3);
                \draw (-0.3,0.3) -- (-0.15,-0.15);
                \draw (0,0.3) -- (0.15,0.15);
                \draw (0,0) -- (0.3,-0.3);
            \end{tikzpicture}
            +
            \begin{tikzpicture}[anchorbase]
                \draw (0.3,-0.3) -- (-0.3,0.3);
                \draw (0.3,0.3) -- (0.15,-0.15);
                \draw (0,0.3) -- (-0.15,0.15);
                \draw (0,0) -- (-0.3,-0.3);
            \end{tikzpicture}
        \right)
        \\
        + (p_1 + p_2)
        \left(
            \begin{tikzpicture}[centerzero]
                \draw (-0.15,-0.3) -- (-0.15,-0.23) arc(180:0:0.15) -- (0.15,-0.3);
                \draw (-0.3,0.3) -- (0,0.08) -- (0.3,0.3);
                \draw (0,0.3) -- (0,0.08);
            \end{tikzpicture}
            +
            \begin{tikzpicture}[centerzero]
                \draw (-0.2,-0.3) -- (-0.2,0.3);
                \draw (0,0.3) -- (0.15,0) -- (0.3,0.3);
                \draw (0.15,0) -- (0.15,-0.3);
            \end{tikzpicture}
            +
            \begin{tikzpicture}[centerzero]
                \draw (0.2,-0.3) -- (0.2,0.3);
                \draw (0,0.3) -- (-0.15,0) -- (-0.3,0.3);
                \draw (-0.15,0) -- (-0.15,-0.3);
            \end{tikzpicture}
            +
            \begin{tikzpicture}[centerzero]
                \draw (-0.3,0.3) -- (-0.3,0.23) arc(180:360:0.15) -- (0,0.3);
                \draw (0.3,0.3) -- (0.15,0) -- (-0.2,-0.3);
                \draw (0.2,-0.3) -- (0.15,0);
            \end{tikzpicture}
            +
            \begin{tikzpicture}[centerzero]
                \draw (0.3,0.3) -- (0.3,0.23) arc(360:180:0.15) -- (0,0.3);
                \draw (-0.3,0.3) -- (-0.15,0) -- (0.2,-0.3);
                \draw (-0.2,-0.3) -- (-0.15,0);
            \end{tikzpicture}
        \right)
        + p_5
        \left(
            \begin{tikzpicture}[centerzero]
                \draw (-0.15,-0.15) -- (-0.15,0.15) -- (0.15,0.15) -- (0.15,-0.15) -- cycle;
                \draw (-0.15,0.15) -- (-0.15,0.25) -- (0,0.35);
                \draw (-0.15,0.25) -- (-0.3,0.35);
                \draw (0.15,0.15) -- (0.3,0.35);
                \draw (-0.15,-0.15) -- (-0.2,-0.3);
                \draw (0.15,-0.15) -- (0.2,-0.3);
            \end{tikzpicture}
            +
            \begin{tikzpicture}[centerzero]
                \draw (-0.15,-0.15) -- (-0.15,0.15) -- (0.15,0.15) -- (0.15,-0.15) -- cycle;
                \draw (0.15,0.15) -- (0.15,0.25) -- (0,0.35);
                \draw (0.15,0.25) -- (0.3,0.35);
                \draw (-0.15,0.15) -- (-0.3,0.35);
                \draw (-0.15,-0.15) -- (-0.2,-0.3);
                \draw (0.15,-0.15) -- (0.2,-0.3);
            \end{tikzpicture}
            +
            \begin{tikzpicture}[centerzero]
                \draw (0,0.15) --  (0.15,0) -- (0,-0.15) -- (-0.15,0) -- cycle;
                \draw (-0.15,0) -- (-0.2,0.3);
                \draw (0,0.15) -- (0,0.3);
                \draw (0.15,0) -- (0.25,0) -- (0.25,0.3);
                \draw (0.25,0) -- (0.25,-0.3);
                \draw (0,-0.15) -- (0,-0.3);
            \end{tikzpicture}
            +
            \begin{tikzpicture}[anchorbase]
                \draw (0,0.15) --  (0.15,0) -- (0,-0.15) -- (-0.15,0) -- cycle;
                \draw (0,-0.15) -- (0,-0.25) -- (0,-0.25) -- (-0.15,-0.4);
                \draw (0,-0.25) -- (0.15,-0.4);
                \draw (-0.15,0) -- (-0.3,0.25);
                \draw (0,0.15) -- (0,0.25);
                \draw (0.15,0) -- (0.3,0.25);
            \end{tikzpicture}
            +
            \begin{tikzpicture}[centerzero]
                \draw (0,0.15) --  (-0.15,0) -- (0,-0.15) -- (0.15,0) -- cycle;
                \draw (0.15,0) -- (0.2,0.3);
                \draw (0,0.15) -- (0,0.3);
                \draw (-0.15,0) -- (-0.25,0) -- (-0.25,0.3);
                \draw (-0.25,0) -- (-0.25,-0.3);
                \draw (0,-0.15) -- (0,-0.3);
            \end{tikzpicture}
        \right).
    \end{multline*}
    Direct computation using SageMath verifies that
    \[
        - 1 + \frac{[3]^2}{[4]^2} (p_3 + p_4 ) = \frac{1}{[2]^2},\quad
        - \frac{[7]}{[4]^2} + \frac{[3]^2}{[4]^2} (p_1+p_2) = \frac{[2][3]^2}{[4]^2[6]},\quad
        \frac{[3]^2}{[4]^2} p_5 = - \frac{[3]}{[2][6]}.
    \]
    Therefore, \cref{planarpent} follows from \cref{qpent}.
\end{proof}

In this approach, we see that $\Fcat_q$ is a \emph{trivalent category} in the sense of \cite{MPS17}.  All computations could then be done with planar graphs.  For example, one can show by direct computation using SageMath that the idempotents of \cref{sec:idempotents} involving crossings can be expressed instead as
\begin{align*}
    e_{\omega_1}
    &= - \frac{[4]^2}{[3]^2[6][8]}\ \jail
    - \frac{(q^4-q^2+1-q^{-2}+q^{-4})[2]^2[4]^2[9]}{[3]^3[8][18]}\, \hourglass
    + \frac{[4]^2}{[2][3]^3[8]}\, \Hmor
    \\ &\qquad \qquad
    - \frac{[4]^3}{[2]^2[3]^3[8]}\, \Imor
    + \frac{[4]^3}{[2][3]^2[6][8]}\, \sqmor,
    \\
    e_{\omega_3}
    &= \frac{[2][7]}{[3]^2[6]}\ \jail
    - \frac{[2][4][7]}{[3]^2[6][8]}\, \hourglass
    - \frac{[4]^2}{[3]^3}\, \Hmor
    - \frac{([9]+[5]-1)[4][6]}{[2][3]^3[12]}\, \Imor
    + \frac{[4]^2}{[2][3]^2[6]}\, \sqmor,
    \\
    e_{2\omega_4}
    &= \frac{[4]^2[7]}{[2][3]^2[8]}\ \jail
    + \frac{(q^4-q^2+1-q^{-2}+q^{-4})[2][4]^2[7]}{[3]^2[8][13]}\, \hourglass
    + \frac{(q^6+q^4+1+q^{-4}+q^{-6})[4]^2}{[2][3]^2[8]}\, \Hmor
    \\ &\qquad \qquad
    + \frac{([9]+[5]-1)[4]^3}{[2]^2[3]^2[7][8]}\, \Imor
    - \frac{[4]^3}{[2]^2[3]^2[8]}\, \sqmor.
\end{align*}
Note that this approach has the advantage that all coefficients are invariant under the map $q \mapsto q^{-1}$.  When showing directly (that is, without using the crossings) that the above elements are idempotents, one uses the relation
\begin{multline} \label{ladder}
    \begin{tikzpicture}[anchorbase]
        \draw (-0.15,-0.5) -- (-0.15,0.5);
        \draw (0.15,-0.5) -- (0.15,0.5);
        \draw (-0.15,0.25) -- (0.15,0.25);
        \draw (-0.15,0) -- (0.15,0);
        \draw (-0.15,-0.25) -- (0.15,-0.25);
    \end{tikzpicture}
    =
    -\frac{[2]^2[7]}{[4]^2}\ \jail\
    + \frac{[2]^3[7][10]}{[4]^2[5]}\ \hourglass\
    + \frac{[2]^2(q^8+q^4+2q^2+1+2q^{-2}+q^{-4}+q^{-8}}{[4]^2}\ \Hmor
    \\
    - \frac{([3]-2)([9]+[5]-1)}{[3]-1}\ \Imor
    + ([5]-2)\ \sqmor,
\end{multline}
which we computed using SageMath.

\section{Affinization\label{sec:affine}}

As a final application of the results of the current paper, we introduce an affine version of the diagrammatic category $\Fcat_q$ and define its action on the category of $U_q$-modules.  This affine category is an $F_4$ analogue of the affine oriented skein category of \cite[\S 4]{Bru17} in type $A$ and the affine Kauffman skein category of \cite{GRS20} in types $BCD$.  To the best of our knowledge, affine categories have not yet appeared in exceptional type.  Our definition uses the general affinization procedure described in \cite{MS21}.  As noted in the introduction, this procedure requires that we work in the quantum setting.  Our treatment is brief, leaving further investigation to future work.

\begin{defin}
    Let $\AFcat_q$ be the strict $\kk$-linear monoidal category obtained from $\Fcat_q$ by adjoining morphisms
    \[
        \posstrand,\ \negstrand
        \colon \go \to \go
    \]
    subject to the relations
    \begin{equation} \label{affrel}
        \begin{tikzpicture}[centerzero]
            \draw (-0.3,-0.3) -- (0.3,0.3);
            \draw[wipe] (0.3,-0.3) -- (-0.3,0.3);
            \draw (0.3,-0.3) -- (-0.3,0.3);
            \posaff{-0.18,-0.18};
        \end{tikzpicture}
        =
        \begin{tikzpicture}[centerzero]
            \draw (0.3,-0.3) -- (-0.3,0.3);
            \draw[wipe] (-0.3,-0.3) -- (0.3,0.3);
            \draw (-0.3,-0.3) -- (0.3,0.3);
            \posaff{0.15,0.15};
        \end{tikzpicture}
        \ ,\qquad
        \begin{tikzpicture}[centerzero]
            \draw (-0.3,-0.3) -- (0,0) -- (0.3,-0.3);
            \draw (0,0) -- (0,0.42);
            \posaff{0,0.21};
        \end{tikzpicture}
        =
        \begin{tikzpicture}[centerzero]
            \draw (-0.3,-0.3) -- (0,0) -- (0.3,-0.3);
            \draw (0,0) -- (0,0.42);
            \posaff{-0.16,-0.16};
            \posaff{0.16,-0.16};
        \end{tikzpicture}
        \ ,\qquad
        \begin{tikzpicture}[centerzero]
            \draw (0,-0.3) -- (0,0.3);
            \posaff{0,0.12};
            \negaff{0,-0.12};
        \end{tikzpicture}
        =
        \begin{tikzpicture}[centerzero]
            \draw (0,-0.3) -- (0,0.3);
            \negaff{0,0.12};
            \posaff{0,-0.12};
        \end{tikzpicture}
        =
        \begin{tikzpicture}[centerzero]
            \draw (0,-0.3) -- (0,0.3);
        \end{tikzpicture}
        \ ,\qquad
        \begin{tikzpicture}[centerzero]
            \draw (-0.2,-0.2) -- (-0.2,0) arc(180:0:0.2) -- (0.2,-0.2);
            \posaff{-0.2,0};
        \end{tikzpicture}
        =
        \begin{tikzpicture}[centerzero]
            \draw (-0.2,-0.2) -- (-0.2,0) arc(180:0:0.2) -- (0.2,-0.2);
            \negaff{0.2,0};
        \end{tikzpicture}
        \ .
    \end{equation}
\end{defin}

\begin{prop}
    The following relations hold in $\AFcat_q$:
    \begin{gather} \label{jaguar}
        \begin{tikzpicture}[centerzero]
            \draw (0.3,-0.3) -- (-0.3,0.3);
            \draw[wipe] (-0.3,-0.3) -- (0.3,0.3);
            \draw (-0.3,-0.3) -- (0.3,0.3);
            \posaff{0.18,-0.18};
        \end{tikzpicture}
        =
        \begin{tikzpicture}[centerzero]
            \draw (-0.3,-0.3) -- (0.3,0.3);
            \draw[wipe] (0.3,-0.3) -- (-0.3,0.3);
            \draw (0.3,-0.3) -- (-0.3,0.3);
            \posaff{-0.15,0.15};
        \end{tikzpicture}
        \ ,\quad
        \begin{tikzpicture}[centerzero]
            \draw (0.3,-0.3) -- (-0.3,0.3);
            \draw[wipe] (-0.3,-0.3) -- (0.3,0.3);
            \draw (-0.3,-0.3) -- (0.3,0.3);
            \negaff{-0.15,-0.15};
        \end{tikzpicture}
        =
        \begin{tikzpicture}[centerzero]
            \draw (-0.3,-0.3) -- (0.3,0.3);
            \draw[wipe] (0.3,-0.3) -- (-0.3,0.3);
            \draw (0.3,-0.3) -- (-0.3,0.3);
            \negaff{0.18,0.18};
        \end{tikzpicture}
        \ ,\quad
        \begin{tikzpicture}[centerzero]
            \draw (-0.3,-0.3) -- (0.3,0.3);
            \draw[wipe] (0.3,-0.3) -- (-0.3,0.3);
            \draw (0.3,-0.3) -- (-0.3,0.3);
            \negaff{0.15,-0.15};
        \end{tikzpicture}
        =
        \begin{tikzpicture}[centerzero]
            \draw (0.3,-0.3) -- (-0.3,0.3);
            \draw[wipe] (-0.3,-0.3) -- (0.3,0.3);
            \draw (-0.3,-0.3) -- (0.3,0.3);
            \negaff{-0.18,0.18};
        \end{tikzpicture}
        \ ,\quad
        \\ \label{wolverine}
        \begin{tikzpicture}[centerzero]
            \draw (-0.2,-0.2) -- (-0.2,0) arc(180:0:0.2) -- (0.2,-0.2);
            \negaff{-0.2,0};
        \end{tikzpicture}
        =
        \begin{tikzpicture}[centerzero]
            \draw (-0.2,-0.2) -- (-0.2,0) arc(180:0:0.2) -- (0.2,-0.2);
            \posaff{0.2,0};
        \end{tikzpicture}
        \ ,\quad
        \begin{tikzpicture}[centerzero]
            \draw (-0.2,0.2) -- (-0.2,0) arc(180:360:0.2) -- (0.2,0.2);
            \posaff{-0.2,0};
        \end{tikzpicture}
        =
        \begin{tikzpicture}[centerzero]
            \draw (-0.2,0.2) -- (-0.2,0) arc(180:360:0.2) -- (0.2,0.2);
            \negaff{0.2,0};
        \end{tikzpicture}
        \ ,\quad
        \begin{tikzpicture}[centerzero]
            \draw (-0.2,0.2) -- (-0.2,0) arc(180:360:0.2) -- (0.2,0.2);
            \negaff{-0.2,0};
        \end{tikzpicture}
        =
        \begin{tikzpicture}[centerzero]
            \draw (-0.2,0.2) -- (-0.2,0) arc(180:360:0.2) -- (0.2,0.2);
            \posaff{0.2,0};
        \end{tikzpicture}
        \ ,
        \\ \label{moose}
        \begin{tikzpicture}[centerzero]
            \draw (-0.3,-0.3) -- (0,0) -- (0.3,-0.3);
            \draw (0,0) -- (0,0.42);
            \negaff{0,0.21};
        \end{tikzpicture}
        =
        \begin{tikzpicture}[centerzero]
            \draw (-0.3,-0.3) -- (0,0) -- (0.3,-0.3);
            \draw (0,0) -- (0,0.42);
            \negaff{-0.16,-0.16};
            \negaff{0.16,-0.16};
        \end{tikzpicture}
        \ ,\quad
        \begin{tikzpicture}[centerzero]
            \draw (-0.3,-0.3) -- (0,0) -- (0.3,-0.3);
            \draw (0,0) -- (0,0.42);
            \posaff{-0.16,-0.16};
        \end{tikzpicture}
        =
        \begin{tikzpicture}[centerzero]
            \draw (-0.3,-0.3) -- (0,0) -- (0.3,-0.3);
            \draw (0,0) -- (0,0.42);
            \posaff{0,0.21};
            \negaff{0.16,-0.16};
        \end{tikzpicture}
        \ ,\quad
        \begin{tikzpicture}[centerzero]
            \draw (-0.3,-0.3) -- (0,0) -- (0.3,-0.3);
            \draw (0,0) -- (0,0.42);
            \negaff{-0.16,-0.16};
        \end{tikzpicture}
        =
        \begin{tikzpicture}[centerzero]
            \draw (-0.3,-0.3) -- (0,0) -- (0.3,-0.3);
            \draw (0,0) -- (0,0.42);
            \negaff{0,0.21};
            \posaff{0.16,-0.16};
        \end{tikzpicture}
        \ ,\quad
        \begin{tikzpicture}[centerzero]
            \draw (-0.3,-0.3) -- (0,0) -- (0.3,-0.3);
            \draw (0,0) -- (0,0.42);
            \posaff{0.16,-0.16};
        \end{tikzpicture}
        =
        \begin{tikzpicture}[centerzero]
            \draw (-0.3,-0.3) -- (0,0) -- (0.3,-0.3);
            \draw (0,0) -- (0,0.42);
            \posaff{0,0.21};
            \negaff{-0.16,-0.16};
        \end{tikzpicture}
        \ ,\quad
        \begin{tikzpicture}[centerzero]
            \draw (-0.3,-0.3) -- (0,0) -- (0.3,-0.3);
            \draw (0,0) -- (0,0.42);
            \negaff{0.16,-0.16};
        \end{tikzpicture}
        =
        \begin{tikzpicture}[centerzero]
            \draw (-0.3,-0.3) -- (0,0) -- (0.3,-0.3);
            \draw (0,0) -- (0,0.42);
            \negaff{0,0.21};
            \posaff{-0.16,-0.16};
        \end{tikzpicture}
        \ .
    \end{gather}
\end{prop}

\begin{proof}
    The first relation in \cref{jaguar} follows from the first relation in \cref{affrel} by composing on the top with $\poscross$ and on the bottom with $\negcross$.  The second relation in \cref{jaguar} then follows from the first relation in \cref{jaguar} by composing on top and bottom of both strands with $\negstrand$.  The third relation in \cref{jaguar} follows from the second relation in \cref{jaguar} by composing on the top with $\negcross$ and on the bottom with $\poscross$.

    The first relation in \cref{wolverine} follows from the last relation in \cref{affrel} by composing on the bottom with $\negstrand\ \posstrand$.  Next we compute
    \[
        \begin{tikzpicture}[centerzero]
            \draw (-0.2,0.2) -- (-0.2,0) arc(180:360:0.2) -- (0.2,0.2);
            \posaff{-0.2,0};
        \end{tikzpicture}
        \overset{\cref{vortex}}{=}
        \begin{tikzpicture}[anchorbase]
            \draw (-1,0.3) -- (-1,0) arc(180:360:0.2) arc(180:0:0.2) arc(180:360:0.2) -- (0.2,0.3);
            \posaff{-0.2,0};
        \end{tikzpicture}
        =
        \begin{tikzpicture}[anchorbase]
            \draw (-1,0.3) -- (-1,0) arc(180:360:0.2) arc(180:0:0.2) arc(180:360:0.2) -- (0.2,0.3);
            \negaff{-0.6,0};
        \end{tikzpicture}
        =
        \begin{tikzpicture}[centerzero]
            \draw (-0.2,0.2) -- (-0.2,0) arc(180:360:0.2) -- (0.2,0.2);
            \negaff{0.2,0};
        \end{tikzpicture}.
    \]
    This proves the second relation in \cref{wolverine}.  Then composing on the top with $\negstrand\ \posstrand$ gives the third relation in \cref{wolverine}.

    The first relation in \cref{moose} follows from the second relation in \cref{affrel} by composing on the top with $\negstrand$ and on the bottom with $\negstrand\ \negstrand$.  The remaining relations in \cref{moose} follow similarly, by composing with appropriate combinations of $\negstrand$ and $\posstrand$.
\end{proof}

In the language of \cite{MS21}, $\AFcat_q = \operatorname{Aff}(\Fcat_q)$ is the \emph{affinization} of $\Fcat_q$; see \cite[Rem.~4.4]{MS21}.  Intuitively, it can be thought of as the category of $\Fcat_q$ string diagrams on the cylinder.  In this interpretation, we have
\[
    \begin{tikzpicture}[centerzero]
        \draw (0,-0.5) -- (0,0.5);
        \posaff{0,0};
        \draw (0.3,-0.5) -- (0.3,0.5);
        \draw (0.6,-0.5) -- (0.6,0.5);
        \draw (0.9,-0.5) -- (0.9,0.5);
        \draw (-0.3,-0.5) -- (-0.3,0.5);
        \draw (-0.6,-0.5) -- (-0.6,0.5);
        \draw (-0.9,-0.5) -- (-0.9,0.5);
        \draw (-1.2,-0.5) -- (-1.2,0.5);
    \end{tikzpicture}
    \ =\
    \begin{tikzpicture}[centerzero]
        \draw (0.3,-0.5) -- (0.3,0.5);
        \draw (0.6,-0.5) -- (0.6,0.5);
        \draw (0.9,-0.5) -- (0.9,0.5);
        \draw[wipe] (0,-0.5) to[out=up,in=200] (1.2,0);
        \draw (0,-0.5) to[out=up,in=200] (1.2,0);
        \draw (-1.5,0) to[out=20,in=down] (0,0.5);
        \draw[wipe] (-0.3,-0.5) -- (-0.3,0.5);
        \draw[wipe] (-0.6,-0.5) -- (-0.6,0.5);
        \draw[wipe] (-0.9,-0.5) -- (-0.9,0.5);
        \draw[wipe] (-1.2,-0.5) -- (-1.2,0.5);
        \draw (-0.3,-0.5) -- (-0.3,0.5);
        \draw (-0.6,-0.5) -- (-0.6,0.5);
        \draw (-0.9,-0.5) -- (-0.9,0.5);
        \draw (-1.2,-0.5) -- (-1.2,0.5);
        \draw[red,dashed] (-1.5,-0.5) -- (-1.5,0.5);
        \draw[red,dashed] (1.2,-0.5) -- (1.2,0.5);
    \end{tikzpicture}
    \ ,\qquad
    \begin{tikzpicture}[centerzero]
        \draw (0,-0.5) -- (0,0.5);
        \negaff{0,0};
        \draw (0.3,-0.5) -- (0.3,0.5);
        \draw (0.6,-0.5) -- (0.6,0.5);
        \draw (0.9,-0.5) -- (0.9,0.5);
        \draw (-0.3,-0.5) -- (-0.3,0.5);
        \draw (-0.6,-0.5) -- (-0.6,0.5);
        \draw (-0.9,-0.5) -- (-0.9,0.5);
        \draw (-1.2,-0.5) -- (-1.2,0.5);
    \end{tikzpicture}
    \ =\
    \begin{tikzpicture}[centerzero]
        \draw (0.3,-0.5) -- (0.3,0.5);
        \draw (0.6,-0.5) -- (0.6,0.5);
        \draw (0.9,-0.5) -- (0.9,0.5);
        \draw (0,-0.5) to[out=up,in=-20] (-1.5,0);
        \draw[wipe] (1.2,0) to[out=160,in=down] (0,0.5);
        \draw (1.2,0) to[out=160,in=down] (0,0.5);
        \draw[wipe] (-0.3,-0.5) -- (-0.3,0.5);
        \draw[wipe] (-0.6,-0.5) -- (-0.6,0.5);
        \draw[wipe] (-0.9,-0.5) -- (-0.9,0.5);
        \draw[wipe] (-1.2,-0.5) -- (-1.2,0.5);
        \draw (-0.3,-0.5) -- (-0.3,0.5);
        \draw (-0.6,-0.5) -- (-0.6,0.5);
        \draw (-0.9,-0.5) -- (-0.9,0.5);
        \draw (-1.2,-0.5) -- (-1.2,0.5);
        \draw[red,dashed] (-1.5,-0.5) -- (-1.5,0.5);
        \draw[red,dashed] (1.2,-0.5) -- (1.2,0.5);
    \end{tikzpicture}
    \ ,
\]
where, on the right-hand side of each equality, the dashed vertical lines are identified.

By \cite[Th.~4.3]{MS21}, the cup $\cupmor$ and cap $\capmor$ endow $\AFcat_q$ with the structure of a strict pivotal category.  Furthermore, by \cite[Cor.~3.4]{MS21}, the natural map $\Fcat_q \to \AFcat_q$, sending $\go$ to $\go$ and sending each of the generating morphisms \cref{lego} to the morphism in $\AFcat_q$ denoted by the same diagram, is faithful.

Recall that, for any category $\cN$, the category $\cEnd(\cN)$ of endofunctors and natural transformations of $\cN$ is a strict monoidal category, with tensor product given by composition of functors.  An \emph{action} of a strict monoidal category $\cC$ on $\cN$ is a monoidal functor $\bA \colon \cC \to \cEnd(\cN)$.  We adopt the notation $X \cdot N = \bA(X)(N)$ for $X \in \Ob(\cC)$ and $N \in \Ob(\cN)$.

The category $\Fcat_q$ acts on $\Mcat_q$ via the action
\begin{equation}
    X \cdot M = \bF_q(X) \otimes M,\qquad
    f \cdot g = \bF_q(f) \otimes g,
\end{equation}
for $X \in \Ob(\Fcat_q)$, $M \in \Ob(\Mcat_q)$, $f \in \Mor(\Fcat_q)$, and $g \in \Mor(\Mcat_q)$.  This action can be extended to an action of $\AFcat_q$ as follows.

\begin{prop}
    There is an action of $\AFcat_q$ on $\Mcat_q$ uniquely determined by
    \begin{gather*}
        X \cdot M = \bF_q(X) \otimes M,\qquad
        f \cdot g = \bF_q(f) \otimes g,
        \\
        \posstrand \cdot g = q^{-24} R_{\sV_q,N}^{-1} \circ (g \otimes 1_{\sV_q}) \circ R_{M,\sV_q}^{-1},\qquad
        \negstrand \cdot g = q^{24} R_{N,\sV_q} \circ (g \otimes 1_{\sV_q}) \circ R_{\sV_q,M},
    \end{gather*}
    for all $X \in \Ob(\Fcat_q)$, $f \in \Mor(\Fcat_q)$, $M,N \in \Ob(\Mcat_q)$, and $g \in \Mcat_q(M,N)$.
\end{prop}

\begin{proof}
    This follows from \cite[Th.~3.2]{MS21} after we note that the twist $\theta_\sV$ used there is the image under $\bF_q$ of
    \[
        \begin{tikzpicture}[centerzero]
        	\draw (0,0.6) to (0,0.3);
        	\draw (0.3,-0.2) to [out=0,in=-90](.5,0);
        	\draw (0.5,0) to [out=90,in=0](.3,0.2);
        	\draw (0,-0.3) to (0,-0.6);
        	\draw (0,0.3) to [out=-90,in=180] (.3,-0.2);
        	\draw[wipe] (0.3,.2) to [out=180,in=90](0,-0.3);
        	\draw (0.3,.2) to [out=180,in=90](0,-0.3);
        \end{tikzpicture}
        \ \overset{\cref{turvy}}{=} q^{-24}\
        \begin{tikzpicture}[centerzero]
            \draw (0,-0.6) -- (0,0.6);
        \end{tikzpicture}
        \ .
        \qedhere
    \]
\end{proof}

Just as the endomorphism algebras $\End_{\Fcat_q}(\go^{\otimes n})$ should be thought of as type $F_4$ analogues of Iwahori--Hecke algebras or BMW algebras, the affine endomorphism algebras $\End_{\AFcat_q}(\go^{\otimes n})$ should be thought of as type $F_4$ analogues of affine Hecke algebras or affine BMW algebras.  We leave a more detailed study, including connections to toroidal link invariants, cyclotomic quotients, and Jucys--Murphy elements, to future work.  We also note that one could, in a similar way, define an affine analogue of Kuperberg's $G_2$ category \cite{Kup94,Kup96}.

The commutative algebra $\End_{\AFcat_q}(\one)$ of closed diagrams in $\AFcat_q$ gives rise to elements in the center $Z(U_q)$ of $U_q$.  The algebra $\End_{\cEnd(\Mcat_q)}(\one) = \End(\Id_{U_q\md})$ of natural endomorphisms of the identity functor on $U_q$-mod can be naturally identified with the center $Z(U_q)$ of $U_q$.  More precisely, evaluation on the identity element of the regular representation defines a canonical algebra isomorphism $\rho \colon \End(\Id_{U_q\md}) \xrightarrow{\cong} Z(U_q)$.  We thus have a homomorphism of algebras
\[
    \rho \circ \bF_q \colon \End_{\AFcat_q}(\one) \to Z(U_q).
\]
In this way, every closed diagram in $\AFcat_q$ gives rise to an element of $Z(U_q)$.  For example, if we define, for $n \ge 0$,
\[
    \begin{tikzpicture}[centerzero]
        \draw (0,-0.2) -- (0,0.2);
        \multdot{0,0}{east}{n};
    \end{tikzpicture}
    = \left( \posstrand \right)^{\circ n}
    \qquad \text{and} \qquad
    \begin{tikzpicture}[centerzero]
        \draw (0,-0.2) -- (0,0.2);
        \multdot{0,0}{east}{-n};
    \end{tikzpicture}
    = \left( \negstrand \right)^{\circ n},
\]
then the image under $\rho^{-1} \circ \bF_q$ of the diagrams
\[
    \begin{tikzpicture}[centerzero]
        \draw (0,0) arc(0:360:0.2);
        \multdot{0,0}{west}{r};
    \end{tikzpicture}
    ,\quad
    \begin{tikzpicture}[centerzero]
        \draw (0,-0.2) to[out=up,in=down] (0.1,0) to[out=up,in=down] (0,0.2) arc(0:180:0.2) -- (-0.4,-0.2) arc(180:360:0.2);
        \draw (0,-0.2) to[out=up,in=down] (-0.1,0) to[out=up,in=down] (0,0.2);
        \opendot{0.1,0};
        \node at (0.25,0) {$\scriptstyle{s}$};
        \opendot{-0.1,0};
        \node at (-0.25,0) {$\scriptstyle{r}$};
    \end{tikzpicture}
    ,\quad
    \begin{tikzpicture}[centerzero]
        \draw (0,0.1) to[out=45,in=180] (0.15,0.2) arc (90:-90:0.2) to[out=180,in=-45] (0,-0.1) -- (0,0.1) to[out=135,in=0] (-0.15,0.2) arc (90:270:0.2) to[out=0,in=-135] (0,-0.1);
        \multdot{-0.35,0}{east}{r};
        \multdot{0.35,0}{west}{s};
    \end{tikzpicture}
    ,\qquad r,s \in \Z,
\]
are elements in $Z(U_q)$.  We leave for future work the question of giving an explicit description of $Z(U_q)$ using such diagrams.


\bibliographystyle{alphaurl}
\bibliography{qF4}

\end{document}